\documentclass[11pt]{amsart}
\usepackage[utf8]{inputenc}
\usepackage[english]{babel}
\usepackage[T1]{fontenc}
\usepackage{latexsym}
\usepackage{amsmath}
\usepackage{amsthm}
\usepackage{amsfonts}
\usepackage{amssymb}
\usepackage{mathrsfs} 
\usepackage{graphicx}
\usepackage[lofdepth,lotdepth]{subfig}
\usepackage{enumerate}
\usepackage{fullpage}
\usepackage{esint}
\usepackage{dsfont}
\usepackage{hyperref}
\usepackage{cleveref}
\usepackage{color,soul,xcolor}

\usepackage[normalem]{ulem}

\newtheorem{theorem}{Theorem} 
\newtheorem{lemma}[theorem]{Lemma} 
\newtheorem{proposition}[theorem]{Proposition} 
\newtheorem{definition}[theorem]{Definition} 
 
\newtheorem{remark}[theorem]{Remark} 

\numberwithin{equation}{section}
\numberwithin{theorem}{section}

\newcommand{\Ai}{\operatorname{Ai}}

\newcommand{\R}{\mathbb{R}}
\newcommand{\bff}{\mathbf{f}}
\newcommand{\bp}{\mathbf{p}}
\newcommand{\bq}{\mathbf{q}}
\newcommand{\bgamma}{\boldsymbol \gamma}
\newcommand{\mA}{\mathcal{A}}

\newcommand{\bQ}{\mathbf{Q}}
\newcommand{\mE}{\mathcal{E}}
\newcommand{\mF}{\mathcal{F}}
\newcommand{\mR}{\mathcal{R}}

\newcommand{\bL}{\mathscr{L}}
\newcommand{\bA}{\mathscr{A}}
\newcommand{\bH}{\mathscr{H}}

\newcommand{\bepsilon}{\boldsymbol{\epsilon}}
\newcommand{\bxi}{\boldsymbol{\xi}}
\newcommand{\btheta}{\boldsymbol{\theta}}
\newcommand{\bTheta}{\boldsymbol{\Theta}}
\newcommand{\beeta}{\boldsymbol{\eta}}
\newcommand{\bx}{\mathbf{x}}
\newcommand{\bX}{\mathbf{X}}
\newcommand{\by}{\mathbf{y}}

\newcommand{\1}{\mathds{1}}
\newcommand{\e}{\epsilon}
\newcommand{\sh}{s_{\vdash}}
\newcommand{\tauh}{\tau_{\vdash}}

\newcommand{\Thetastar}{\theta^{\diamond}}
\newcommand{\shh}{s_{\vdash}}
\newcommand{\tauhh}{\tau_{\vdash}}
\newcommand{\htheta}{\theta_{\vdash}}

\DeclareMathOperator{\supp}{supp}

\definecolor{bl}{rgb}{0,0,0}

\makeatletter
\let\@wraptoccontribs\wraptoccontribs
\makeatother

\title{Non-local competition slows down front acceleration during dispersal evolution}
\author{Vincent Calvez}
\address{Institut Camille Jordan, UMR 5208 CNRS \& Universit\'{e} Claude Bernard Lyon 1,  France}
\email{vincent.calvez@math.cnrs.fr, tdumont@math.univ-lyon1.fr}
\author{Christopher Henderson}
\address{Department of Mathematics, University of Arizona, Tucson, AZ 85721}
\email{ckhenderson@math.arizona.edu}
\author{Sepideh Mirrahimi}
\address{Institut de Math\'ematiques de Toulouse; UMR 5219, Universit\'e de Toulouse; CNRS, UPS, F-31062 Toulouse Cedex 9, France}
\email{sepideh.mirrahimi@math.univ-toulouse.fr}
\author{Olga Turanova}
\address{Michigan State University, East Lansing, MI 48824}
\email{turanova@msu.edu}

\contrib[with a numerical appendix by]{Thierry Dumont}

\begin{document}

\begin{abstract}
We investigate the super-linear spreading in a reaction-diffusion model analogous to the Fisher-KPP equation, but in which the population is heterogeneous with respect to the dispersal ability of individuals, and the saturation factor is non-local with respect to one variable. We prove that the rate of acceleration is slower than the rate of acceleration predicted by the linear problem, that is, without saturation. This hindering phenomenon is the consequence of a subtle interplay between   the non-local saturation and the non-trivial dynamics of some particular curves that carry the mass at the front. A careful analysis of these trajectories allows us to identify the value of the rate of acceleration. The article is complemented with numerical simulations that illustrate some behavior of the model that is beyond our analysis.    
\end{abstract}


\maketitle

\section{Introduction and Main result}\label{sec:intro}

It is commonly acknowledged that the rate of front propagation for logistic reaction-diffusion equations is determined by the linear problem, that is, without growth saturation. This is indeed the case for the celebrated Fisher-KPP equation,
\begin{equation}\label{eq:FKPP}
	n_t = \theta n_{xx} + n(1-n).
\end{equation}
It is known~\cite{Fisher,KPP,AronsonWeinberger2} that the level lines of the solution propagate with speed $2\sqrt{\theta}$, provided the initial data is localized (e.g., compactly supported). This coincides with the spreading speed of the linear problem $\bar n_t =  \theta \bar n_{xx} + \bar n$, which can be seen, for instance, from its fundamental solution,
\begin{equation}\label{eq:fund sol FKPP}
\bar n(t,x) = \dfrac{1}{2\sqrt{\pi \theta t}}\exp\left( -\dfrac{t}{4\theta} \left [ \left( \dfrac{x}t\right )^2 - 4\theta \right ]\right)\, . 
\end{equation}
The linear determinacy of the wave speed for reaction-diffusion equation is a long-standing question, see {\em e.g.} \cite{hadeler_travelling_1975,crooks_travelling_2003} and \cite{lucia_linear_2004,weinberger_sufficient_2012} for recent developments on scalar homogeneous equations. 
It has been established in many other contexts, such as for related inhomogeneous models (see, e.g., 
\cite{BerestyckiHamel02,berestycki_speed_2005}, and the  recent \cite{nadin_generalized_2017} and references therein) as well as for systems  under certain conditions (see, e.g.
\cite{lewis_spreading_2002,weinberger_analysis_2002,
li_spreading_2005}, the recent work in~\cite{girardin_non-cooperative_2018-1,girardin_non-cooperative_2018}, and references therein).  More recently, linear determinacy has been established for many non-local equations as well~(see, e.g., \cite{berestycki_non-local_2009,AlfaroCovilleRaoul, BerestyckiJinSilvestre, HamelRyzhik}).  This is necessarily only a small sampling of the enormous body of literature utilizing the relationship between spreading speeds and linearization in reaction-diffusion equations arising in ecology and evolution.

In the present work, we report on a similar equation, called the cane toads equation, that describes a  population that is heterogeneous with respect to its dispersal ability. Namely, we consider the population density $f(t,x,\theta)$ whose dynamics are described by the following equation:
\begin{equation}\label{eq:intro_toads}
	\begin{cases}
		f_t = \theta f_{xx} + f_{\theta\theta} + f(1-\rho) \quad &\text{ in } \R_+^* \times \R \times (1,\infty),\\
		 f_\theta = 0   & \text{ on }  \R_+^*\times \R\times \left\{1\right\},
	\end{cases}
\end{equation}
where $\rho(t,x) = \int_{1}^\infty f(t,x,\theta) d\theta$ is the spatial density.    The zeroth order term $f(1-\rho)$ is referred to as the reaction term. The equation is complemented with an initial datum $f_0$ such that, for some $C_0\geq 1$,
\begin{equation}\label{eq:initial_data}
	C_0^{-1}\1_{(-\infty, - C_0) \times (1,1 + C_0^{-1})}
		\leq f_0
		\leq C_0 \1_{(-\infty, C_0) \times (1,1 + C_0)}.
\end{equation}

Equation \eqref{eq:intro_toads} was proposed as a minimal model to describe the interplay between ecological processes (population growth and migration) and evolutionary processes (here, dispersal evolution) during the invasion of an alien species in a new environment, see \cite{BenichouEtAl} following earlier work in \cite{ChampagnatMeleard} and \cite{ArnoldDesvillettesPrevost}. 
The population is structured with respect to the dispersal ability of individuals, which is encoded in the trait $\theta>1$. Offspring may differ from their parents with respect to mobility. Deviation of mobility at birth is accounted for as $f + f_{\theta\theta}$, with Neumann boundary conditions at $\theta = 1$. Finally, growth becomes saturated as the population density $\rho(t,x)$ reaches unit capacity locally in space. 
We note that we use the trait space $\theta \in (1,\infty)$ for simplicity, but our proof applies to the case when the trait space is $(\underline\theta, \infty)$ for any $\underline\theta>0$.

Problem (\ref{eq:intro_toads}) shares some similarities with kinetic equations
(see, for example, the review \cite{Villani}), as the structure variable $\theta$ acts on the higher
order differential operator. However, here the differential operator is of
second order, whereas it is of first order (transport) in the case
of kinetic equations.

%
%
%
%

The goal of this study is to understand spreading in~\eqref{eq:intro_toads}, and, in particular, to emphasize the comparison with the  rate of propagation of the linearized problem $\bar f_t = \theta \bar f_{xx} + \bar f_{\theta\theta} + \bar f$. Indeed, our main results, Theorems \ref{thm:propagation} and \ref{thm:alpha^*}, imply that propagation in (\ref{eq:intro_toads}) is slower than that predicted by the  linearized  problem.
This is surprising at first glance, as  linearization in homogeneous scalar reaction-diffusion equations always {\em underestimates} the actual wave speed \cite{hadeler_travelling_1975,
crooks_travelling_2003,lucia_linear_2004}, whereas it {\em overestimates} the spreading rate in our case study.

Another noticeable fact is that we are able to characterize algebraically   the critical value for the rate of expansion, despite the fact that the problem becomes genuinely non-linear because of the impact of competition on spreading.

%
%
%
%
%
%

It is important to note that, although~\eqref{eq:intro_toads} and~\eqref{eq:FKPP} seem strongly related at first, the two have deep structural differences stemming from the interaction of the non-local saturation term $-f\rho$ and the unbounded diffusivity  $\theta \in (1,\infty)$.   The most obvious consequence of  the former is the fact that~\eqref{eq:intro_toads} lacks a comparison principle.  This is a serious technical issue that forces us to rely on and extend earlier techniques of Bouin, Henderson, and Ryzhik~\cite{BHR_acceleration}.  There are, however, further phenomenonological differences between the two models, leading to additional difficulties which are discussed in greater detail below.

One  salient feature of \eqref{eq:intro_toads} is the accelerated propagation that results from the interplay between ecology and evolution. One may heuristically derive the rate of acceleration from the linear equation as follows: first, we ignore the ecological part, so that we are reduced to the linear Fisher-KPP equation in the $\theta$ direction: $\bar f_t =  \bar f_{\theta\theta} + \bar f$, and we find that $\bar \theta(t) = \mathcal O(t)$, where $\bar \theta(t)$ is roughly the location of the front (with respect to $\theta$); second, we focus on the ecological part: $\bar f_t =  \bar \theta(t) \bar f_{xx} + \bar f$, and we find that $\bar x(t) = t \mathcal O(\bar \theta(t)^{1/2}) =  \mathcal O(t^{3/2}) $. This heuristic argument can be rephrased as a ``spatial sorting'' phenomenon: individuals with higher dispersal abilities travel far away, where they give birth to possibly better dispersers, yielding sustained acceleration of the front. 

Acceleration was reported in a series of studies about the cane toads invasion~\cite{phillips2006invasion,urban_toad_2008}  after their introduction in the 1930's in Queensland, Australia.   
It is hypothesized that spatial sorting is one of the major causes for this acceleration \cite{shine}.
Our analysis  enables us to quantify this interplay between ecology (species invasion) and dispersal evolution.

\subsection*{Super-linear spreading (front acceleration)}

In \cite{BCMetal}, Bouin et.~al.\ argued formally that the linear problem (omitting the quadratic saturation term) should propagate super-linearly as $(4/3) t^{3/2}$ at the leading order.
This prediction was rigorously confirmed for the local version of~\eqref{eq:intro_toads}, that is, when $f(1-\rho)$ is replaced by $f(1-f)$,
by Berestycki, Mouhot and  Raoul~\cite{BerestyckiMouhotRaoul} using  probabilistic techniques, and by Bouin, Henderson and Ryzhik \cite{BHR_acceleration} using PDE arguments (see also 
Henderson, Perthame and Souganidis~\cite{HendersonPerthameSouganidis} for a  more general model).   While the local model is unrealistic for the context of spatial sorting, it allows the difficulties due to the unbounded diffusion to be isolated from those caused by the non-local saturation.  In particular, the comparison principle is not available for~\eqref{eq:intro_toads} but is for the local version.

Due to the inherent difficulties, less precise information is known about  the full non-local model.  In~\cite{BerestyckiMouhotRaoul}, the authors investigated and established the same spreading result for a related model in which the saturation term is limited in range; that is, $\rho$ is replaced with $\rho_A(t,x,\theta) = \int_{(\theta-A)\vee 1}^{\theta+A} f(t,x,\theta')d\theta'$ for any $A>0$.  On the other hand, in~\cite{BHR_acceleration}, the authors studied~\eqref{eq:intro_toads} and showed that the front is located, roughly speaking, between $(8/3^{7/4}) t^{3/2}$ and $(4/3) t^{3/2}$ in a weak sense (see~\cite[Theorem~1.2]{BHR_acceleration}).

Here, we establish that, contrary to immediate intuition, the propagation of level lines is slower than $(4/3) t^{3/2}$ for \eqref{eq:intro_toads}.  Namely, there exists a constant $\alpha^* \in (0,4/3)$ such that the front is located around $\alpha^* t^{3/2}$ in a weak sense (see Theorem \ref{thm:propagation} for a precise statement). By refining some calculations performed in \cite{BHR_acceleration}, we can prove without too much effort that $\alpha^* > 5/4 > 8/3^{7/4}$.  Characterizing $\alpha^*$ requires more work. We find eventually that $\alpha^*$ is the root of an algebraic equation involving the Airy function and its first derivative. This allows to get a numerical value for $\alpha^*$ of arbitrary precision, {\em e.g.} $\alpha^* \approx 1.315135$. It is immediate to check that this value is compatible with all previous bounds. Indeed, we notice that $\alpha^*$ is much closer to $4/3$ than any of the above lower bounds, so that the relative difference is below $2\%$.

\begin{figure}
\begin{center}
\includegraphics[width=.8\linewidth]{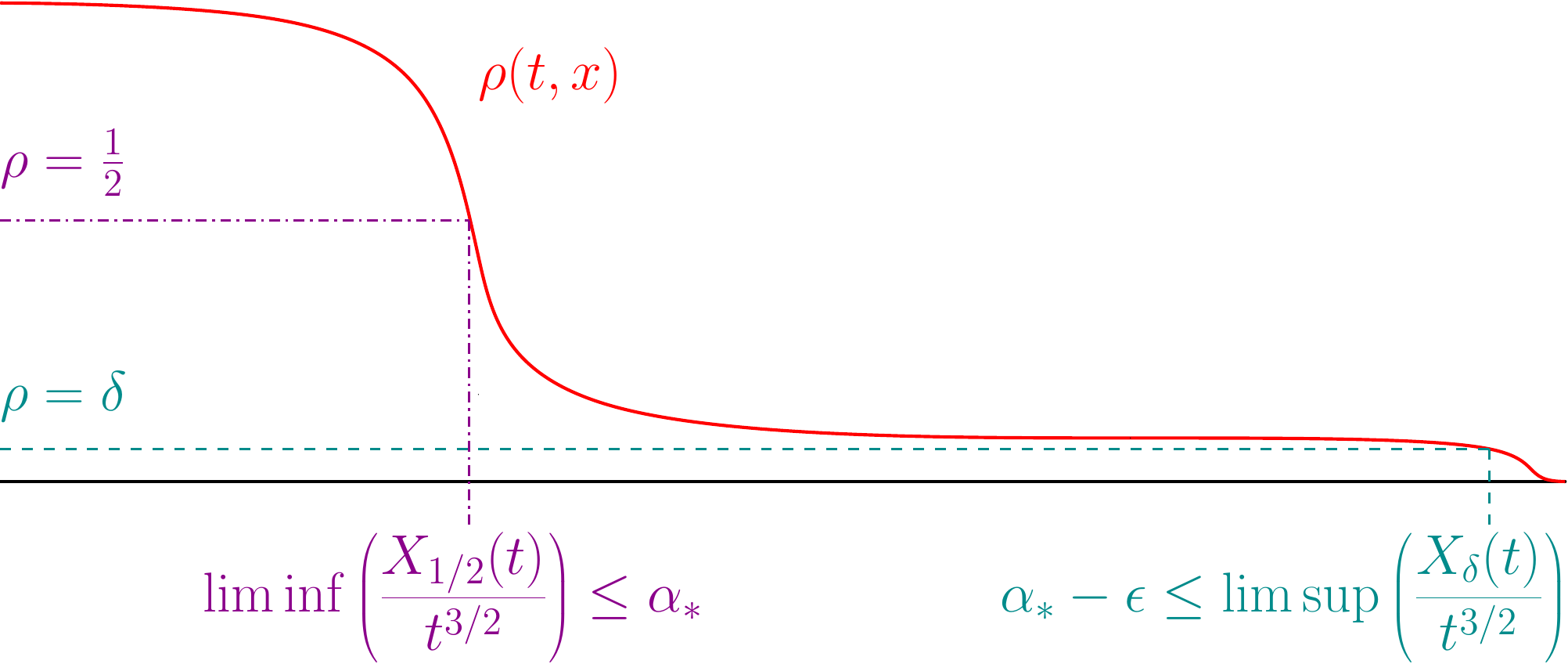} 
\caption{Illustration of Theorem \ref{thm:propagation}, by means of a cartoon picture that cannot be ruled out by \Cref{thm:propagation}. The problem is to disprove the fact that different level sets may propagate at different rates.}\label{fig:stretched_front}
\end{center}
\end{figure}

\subsection*{Abstract characterization of the critical value $\alpha^*$}

In order to give precise results, we need some notation. Let $\alpha\in[0,4/3]$ and $\mu \geq 1/2$.   
Let $U_{\alpha,\mu}$ denote the value function of the following variational problem:
\begin{equation}\label{eq:minimization}
	U_{\alpha,\mu}(x,\theta)
		= \inf \left\{ \int_0^1 L_{\alpha,\mu}(t,\bx(t),\btheta(t),\dot \bx(t), \dot\btheta(t))dt 
				: (\bx(\cdot),\btheta(\cdot)) \in \mA(x,\theta)\right\},
\end{equation}
where the Lagrangian is given by
\begin{equation*}
	L_{\alpha,\mu}(t,x,\theta,v_x,v_\theta) = \frac{v_x^2}{4 \theta} + \frac{v_\theta^2}{4} - 1 + \mu \1_{\left\{ x <  \alpha t^{3/2}\right\}}\,,
\end{equation*}
and $\mA(x,\theta)$ denotes the set of trajectories $\bgamma: [0,1] \to \R\times \R_+$ such that $\bgamma(0) = (0,0), \bgamma(1) = (x,\theta)$ and the integral quantity in \eqref{eq:minimization} is well-defined. 
\textcolor{bl}{We use the shorter notations $U_\alpha$ and $L_\alpha$ for $U_{\alpha,1}$ and $L_{\alpha,1}$ respectively.}
	It is one of our important results \textcolor{bl}{(see Proposition \ref{lem:good_trajectories})} that $U_{\alpha,\mu}(x, \theta)$ does not depend on the value of $\mu$, when $\mu\geq 1/2$ and $x\geq \alpha$. \textcolor{bl}{As a consequence $U_{\alpha,\mu}(x, \theta)=U_{\alpha}(x, \theta)$ for all $\mu\geq 1/2$ and $x\geq \alpha$.} In this context, when there is no possible ambiguity, we write $U_\alpha$ and $L_\alpha$ for $U_{\alpha,\mu}$ and $L_{\alpha,\mu}$, respectively.  The critical value is
\begin{equation}\label{eq:alpha^*}
\begin{split}
	\alpha^*
		= \sup\left\{ \alpha \in [0,4/3]: \min_\theta U_\alpha(\alpha, \theta) \leq 0\right\}.
\end{split}
\end{equation}

We have the following properties of $\alpha^*$:
\begin{proposition}
\label{prop:alpha}
The constant $\alpha^*$ is well-defined and it satisfies
$5/4 < \alpha^* < 4/3$.
\end{proposition}
We refer to Subsection \ref{sec:prop:alpha} and Section \ref{sec:basic_trajectories} for the proof of this statement.

\subsection*{Exact rate of acceleration (in a weak sense)}

In order to state our first result, we introduce the following time-dependent spatial locations, as in Figure \ref{fig:stretched_front},
\begin{equation*}
\underline{X}_{1/2}(t) = \min\left\{x : \rho(t,x) \leq 1/2\right\}  \, , \quad \overline{X}_{\delta}(t) =  \max\left\{x: \rho(t,x) \geq \delta\right\}.
\end{equation*}
Interestingly, the value $\alpha^*$ gives a reasonable (but weak) description of the spreading properties of \eqref{eq:intro_toads}. 
\begin{theorem}\label{thm:propagation}
Suppose that $f$ satisfies~\eqref{eq:intro_toads} with initial data $f_0$ localized in the sense of~\eqref{eq:initial_data}. Then,
\begin{equation*}
	\liminf_{t\to\infty}  \left ( \frac{\underline{X}_{1/2}(t)}{t^{3/2}}\right ) \leq \alpha_*  .
\end{equation*}
Moreover, for all $\e>0$ there exists $\delta>0$ such that
\begin{equation}\label{eq:lower_bound}
	\limsup_{t\to\infty} \left ( \frac{\overline{X}_{\delta}(t)}{t^{3/2}} \right ) \geq \alpha^* - \e.
\end{equation}
\end{theorem}

Roughly, following \Cref{thm:propagation}, there exist infinite sequences of times $t_1, t_2, \dots \to \infty$ and $s_1,s_2,\dots \to\infty$ such that, the level line $\left\{\rho = \delta\right\}$ has reached $\alpha^* t_i^{3/2}$ at $t_i$, whereas the  level line $\left\{\rho = 1/2\right\}$ is no further than $\alpha^* s_i^{3/2}$ at $s_i$. A few comments  are in order.  First, as with other non-local Fisher-KPP-type equations that lack the comparison principle~\cite{berestycki_non-local_2009}, we are unable to establish spatial monotonicity of $\rho$ and $f$. Thus, we cannot rule out that the front oscillates, in contrast to what is depicted in Figure \ref{fig:stretched_front}. 
Second, 
we cannot rule out front stretching, even along sequences of times \textcolor{bl}{(this is the situation depicted in Figure \ref{fig:stretched_front})}. The reason is that the upper threshold value $1/2$ cannot be made arbitrarily small in our approach. Our result is compatible with a monotonic front in which $\underline{X}_{1/2}(t)$ is moving at rate $(5/4)t^{3/2}$ and $\overline{X}_{1/10}(t)$ is moving at rate $(4/3)t^{3/2}$, for instance. Getting stronger results following our approach seems out of reach  at present. 

The appendix is devoted to numerical computations that indicate that the front profile is monotonically decreasing and that all level lines move at the same rate.  Together with Theorem \ref{thm:propagation}, this suggests that the front propagates at rate $\alpha^* t^{3/2}$ in the usual sense.

\subsection*{Further characterization of the critical value $\alpha^*$}

We give two other characterizations of $\alpha^*$. The following definitions are required.  Let $\Xi_0\approx -2.34$ be the largest zero of the Airy function $\Ai$.  For $\xi > \Xi_0$, we define the function
\begin{equation}\label{eq:riccati}
	\mR(\xi) = - \dfrac{ \Ai'(\xi)}{  \Ai(\xi) } = - \dfrac{d \log \Ai}{d\xi} (\xi) \, .
\end{equation}
Note that $\Xi_0$ is a singular point for $\mR$, and that 
 $\mR$ is well-defined and smooth on $(\Xi_0,\infty)$. We provide further discussion of these and related functions in Subsection \ref{subsubsec:Airy}.

In addition, we define the following algebraic function $V$  for $\tau \in (0,1)$,
\begin{equation}\label{eq:V}
	V(\tau) = \left[\frac{(1-\tau)^{1/2}(2+\tau)}{2(1+\tau)^{3/2}}\right]^\frac13.
\end{equation}

\begin{theorem}\label{thm:alpha^*}
The constant $\alpha^*$ has the following two characterizations.
\begin{enumerate}[(i)]
\item For all $\alpha \in (0,4/3]$, we have
\begin{equation}\label{eq:U_scaling}
	\min_\theta U_\alpha(\alpha,\theta) = \left(\frac{3\alpha}{4}\right)^{4/3} \min_\theta U_{\frac43}\left ( \frac43 , \theta\right ) -1  + \left(\frac{3\alpha}{4}\right)^{4/3} \,.
\end{equation}
Hence,
\[
	\alpha^* = \frac{4}{3} \left(\frac{1}{\displaystyle\min_\theta U_{4/3}(4/3, \theta) + 1}\right)^{3/4}.
\]
\item There is a unique solution $\tau_0\in (0,1)$ of
\begin{equation}\label{eq:V_condition}
	V(\tau_0)^2 = \mR\left( V(\tau_0)^4 - \frac{\tau_0 V(\tau_0)}{\left (1-\tau_0^2\right )^\frac12}\right)
\end{equation}
such that the argument of $\mR$ belongs to $(\Xi_0,\infty)$.
Then, 
\begin{equation}\label{eq:alpha^* TH}
	\alpha^*
		= \frac{4}{3} \left[\left(\frac{2(1-\tau_0)}{2+\tau_0}\right)^\frac13 \frac{2(1+\tau_0)^2}{2 + 3 \tau_0 - \tau_0^2} \right]^{\frac34}.
\end{equation}
\end{enumerate}
\end{theorem}

The main purpose of \Cref{thm:alpha^*} is to provide an analytic formula for $\alpha^*$ that can be easily (numerically) computed. It is from this representation that we obtain the decimal approximation $\alpha^* \approx 1.315135$ given above.

We mention that, in fact, a stronger scaling relationship than~\eqref{eq:U_scaling} holds that takes into account the scaling in $\theta$.  This is straightforward to obtain and \Cref{thm:alpha^*}.(i) follows directly from it.  The advantages of~\Cref{thm:alpha^*}.(i) are twofold.  First, it provides a direct way of determing that $\alpha^* < 4/3$.  Second, it reduces the task of finding $\alpha^*$ to computing only $U_{4/3}$, instead of the whole range of $U_\alpha$ for $\alpha \in [0,4/3]$.  \Cref{thm:alpha^*}.(i) also allows us to simplify the proof of \Cref{thm:alpha^*}.(ii).

\subsection*{Motivation and state of the art}

The interplay between evolutionary processes and spatial heterogeneities has a long history in theoretical biology \cite{clobert_dispersal_2012}. It is commonly accepted that migration of individuals can dramatically alter the local adaptation of species when colonizing new environments \cite{macarthur_theory_2001}. This phenomenon is of particular importance at the margin of a species' range where individuals experience very low competition with conspecifics, and where gene flow plays a key role. An important related issue in evolutionary biology is dispersal evolution, see {\em e.g.} \cite{ronce_how_2007}.

An evolutionary spatial sorting process has been
described in \cite{shine}. Intuitively, individuals with higher dispersal reach new areas first, and there they produce offspring
having possibly higher abilities. Based on numerical simulations of an individual-based model, it has been predicted that this process generates biased phenotypic distributions towards high dispersive phenotypes at the expanding
margin \cite{travis_dispersal_2002,travis_accelerating_2009}. As a by-product, the front accelerates, at least transiently before the process stabilizes. Evidence of biased distributions and accelerating fronts have been reported \cite{thomas_ecological_2001,phillips2006invasion}. It is worth noticing that ecological (species
invasion) and evolutionary processes (dynamics of phenotypic distribution) can arise over similar time scales.

Equation \eqref{eq:intro_toads} was introduced in \cite{BenichouEtAl} and built off the previous contributions \cite{ChampagnatMeleard} and \cite{ArnoldDesvillettesPrevost}.  It has proven amenable to mathematical analysis.
In the case of bounded dispersal ($\theta\in (1,10)$, say) Bouin and Calvez \cite{BouinCalvez} constructed travelling waves which are stationary  solutions in the moving frame $x-c^*t$ for a well-chosen (linearly determined) speed $c^*$. Turanova obtained uniform $L^\infty$ bounds for the Cauchy problem and deduced the first spreading properties, again for bounded $\theta$ \cite{Turanova}. The propagation result was later refined by Bouin, Henderson and Ryzhik \cite{BHR_LogDelay} using Turanova's $L^\infty$ estimate. We highlight this point since no uniform $L^\infty$ bound is known for the unbounded case \eqref{eq:intro_toads}.  In addition, their strategy depended on a ``local-in-time Harnack inequality'' that is not applicable in our setting.    It is also interesting to note that the spreading speed is determined by the linear problem in the case of bounded $\theta$. The same conclusion was drawn by Girardin who investigated a general model which is continuous in the space variable, but discrete in the trait (diffusion) variable \cite{girardin_non-cooperative_2018-1,girardin_non-cooperative_2018}.

On the one hand, our work belongs to the wider field of structured reaction-diffusion equations, which are combinations of ecological and evolutionary processes. In a series of works initiated by \cite{AlfaroCovilleRaoul}, and followed by \cite{BouinMirrahimi,BerestyckiJinSilvestre,
AlfaroBerestyckiRaoul}, various authors studied reaction-diffusion models structured by a phenotypic trait, including a non-local competition similar to (and possibly more general than) $-f\rho$. However, in that series of studies, the trait is assumed not to impact dispersion, but reproduction, as for {\em e.g.} $f_t = f_{xx} + f_{\theta\theta} + f(a(\theta)-\rho)$. In particular, no acceleration occurs, and the linear determinacy of the speed is always valid. Note that more intricate dependencies are studied, in particular the presence of an environmental cline, which results in a mixed dependency on the growth rate $a(\theta - Bx)$, as, for instance, in \cite{AlfaroCovilleRaoul,
AlfaroBerestyckiRaoul} \textcolor{bl}{(see also \cite{BouinMirrahimi})}.

On the other hand, our work also fits into the analysis of accelerating fronts in reaction-diffusion equations. We refer to \cite{CabreRoquejoffre,CoulonRoquejoffre} for the variant of the Fisher-KPP equation where the spatial diffusion is replaced with a non-local fractional diffusion operator. In this case, the front propagates exponentially fast, see also \cite{meleard_singular_2015,mirrahimi_singular_2018}. The case where spatial dispersion is described by a  convolution operator with a fat-tailed kernel was first analyzed in \cite{Garnier}, see also \cite{bouin_thin_2018}. The rate of acceleration depends on the asymptotics of the kernel tails. In \cite{BouinCalvezNadin,bouin_large_2016}, the authors investigated the acceleration dynamics of a kinetic variant of the Fisher-KPP equation, structured with respect to the velocity variable. The main difference with the current study is that the kinetic model of \cite{BouinCalvezNadin} (see also \cite{cuesta_traveling_2012}) enjoys the maximum principle.

A natural question in front propagation, related to the issue of linear determinacy, is whether the long-time and long-range behavior can be described via a geometric equation. This is not always possible -- see, for example, \cite[Chapter 6.2]{Freidlin1} and   \cite[examples B(i) and B(ii)]{MajdaSouganidis}. In our setting, finding a PDE governing the asymptotic behavior of \eqref{eq:intro_toads} remains an interesting open problem (see the discussion in Section 2).

\subsection*{Notation}   We use $C$ to denote a general constant independent of all parameters but $C_0$ in \eqref{eq:initial_data}.  In addition,  when there is a limit being taken, we use $A = \mathcal O(B)$ and $A = o(B)$ to mean that $A \leq C B$ and $A/B \to 0$ respectively. 

We set $\R_\pm = \left\{x \in \R : \pm x \geq 0\right\}$ and $\R_\pm^* = \R_\pm \setminus\left\{0\right\}$.  All function spaces are as usual; for example, $L^2(X)$ refers to square integrable functions on a set $X$.

In order to avoid confusion between trajectories and their endpoints, we denote trajectories with bold fonts $(\bx,\btheta)$.  In general, we use $x$, $\bx$, $\theta$, and $\btheta$ for points and trajectories in the original variables and $y$, $\by$, $\eta$, and $\beeta$ for points and trajectories in the self-similar variables, which shall be introduced in Subsection \ref{sec:self-similar}.

\medskip
\subsection*{Acknowledgements.} The authors are grateful to Emmanuel Tr\'elat for shedding light on the connection with sub-Riemannian geometry. Part of this work was completed when VC was on temporary leave to the PIMS (UMI CNRS 3069) at the University of British Columbia.  This project has received funding from the European Research Council (ERC) under the European Union’s Horizon 2020 research and innovation programme (grant agreement No 639638).  CH was partially supported by NSF grants DMS-1246999 and DMS-1907853.  Part of this work was performed within the framework of the LABEX MILYON (ANR10-LABX-0070) of Universit\'e de Lyon, within the program ``Investissements d’Avenir'' (ANR-11-IDEX-0007) operated by the French National Research Agency (ANR).  SM was partially supported by the french ANR projects KIBORD ANR-13-BS01-0004
and MODEVOL ANR-13-JS01-0009. OT was partially supported by NSF grants DMS-1502253 and DMS-190722, and by the Charles Simonyi Endowment at the Institute for Advanced Study.

\section{Strategy of proof}\label{sec:strategy}

This section is intended to sketch the main ideas underlying the  proof. As a by-product, we explain, in rough terms, the reason for slower propagation than linearly determined.

\subsection*{Approximation of geometric optics from the PDE viewpoint}

Our argument follows pioneering works by Freidlin  on the approximation of geometric optics for reaction-diffusion equations \cite{Freidlin2,freidlin_geometric_1986}. In fact, we follow the PDE viewpoint of Evans and Souganidis~\cite{EvansSouganidis}. The approach is based on a long-time, long-range rescaling of the equation that captures the front while ``ignoring'' the microscopic details.  In our case, this involves defining the rescaled functions
$$
	f_h(t,x,\theta)
		= f\left(\frac{t}{h},\frac{x}{h^{3/2}},\frac{\theta}{h}\right) \quad \text{ and } \quad
	\rho_h(t,x )
		= \rho\left(\frac{t}{h},\frac{x}{h^{3/2}}\right).
$$
Note that this scaling comes from the fact that the expected position of the front $\overline x(t)$ scales like $\mathcal{O}(t^{3/2})$ and the expected mean phenotypic trait of the  population at the front $\bar \theta(t)$ scales like $\mathcal{O}(t)$. We then use the Hopf-Cole transformation; that is, we let
\begin{equation}\label{eq:u_h}
u_h(t,x,\theta) = - h\log f_h(t,x,\theta). 
\end{equation} 
Then, after a simple computation,
$$
\begin{cases}
 u_{h,t} +\theta | u_{h,x}|^2+|u_{h,\theta}|^2+1-\rho_h=h (\theta u_{h,\theta\theta} + u_{h,xx}) &\text{ in } \R_+^* \times \R \times (h,\infty),\\ 
 u_h=\infty  & \text{ on } \left\{0\right\} \times [(-\infty, C_0 h) \times (h, (1 + C_0)h)]^c,\\
 |u_h|\leq h |\log C_0| & \text{ on } \left\{0\right\}\times (-\infty,-C_0h )\times  (h, h(1+C_0^{-1})),
\end{cases}
$$
where the complement of the set is taken in $\R\times[h,\infty)$.  Formally passing to the limit as $h\to 0$, we find 
\begin{equation}\label{eq:formal_equation}
\begin{cases}
u_t+  \theta| u_{x}|^2+|u_{\theta}|^2+1-\rho(t,x) =0 &\text{ in } \R_+^* \times \R \times \R_+^*,\\ 
u=\infty  & \text{ on } \left\{0\right\}\times [(-\infty,0)\times\left\{0\right\}]^c,\\
\textcolor{bl}{u= 0} & \text{ on } \left\{0\right\} \times (-\infty,0)\times \left\{0\right\},
\end{cases}
\end{equation}
where the complement of the set is now taken in $\R\times[0,\infty)$. We note that no bounds on the regularity of $\rho$ are available in the literature so it is not clear that \eqref{eq:formal_equation} can be interpreted rigorously.

Another informal candidate for the global limiting Hamilton-Jacobi problem, if any, is the following equation:
\begin{equation}\label{eq:good candidate}
u_t+  \theta| u_{x}|^2+|u_{\theta}|^2+\1_{\left\{\min_{\theta'} u(t,x,\theta')>0\right\}} = 0\,.
\end{equation}
It is based on the heuristics that the spatial density $\rho$ saturates to the value 1 at the back of the front, {\em i.e.} when $\min_{\theta'} u(t,x,\theta')=0$, see Figure \ref{fig:data-thierry-2} for a numerical evidence. However, we are lacking tools to address this issue. Actually, the uniform boundedness of $\rho$ is not yet established to the best of our knowledge (but see \cite{Turanova} for a uniform $L^\infty$ bound in the case of bounded $\theta$). 
One could also search for an alternative formulation of \eqref{eq:good candidate} in the framework of obstacle Hamilton-Jacobi equations (see, e.g. \cite{EvansSouganidis,bouin_large_2016}); however, the non-trivial behavior behind the front complicates this.

Nevertheless, in our proof we are still able to use the toolbox of Hamilton-Jacobi equations, such as variational representations and half-relaxed limits (see below) \cite{barles_discontinuous_1987} in order to reach our goal.
We explain in the next paragraph our strategy to replace $\rho$ by a given indicator function $\mu \1_{\left\{ x <  \alpha t^{3/2}\right\}}$. This translates into the
equation 
\begin{equation}
u_t+  \theta| u_{x}|^2+|u_{\theta}|^2+1 - \mu \1_{\left\{x<\alpha t^{3/2}\right\}} = 0\, ,
\end{equation}
which admits a comparison principle, together with the variational representation  \eqref{eq:minimization}.



\subsection*{Arguing by contradiction to obtain the spreading rate}

Clearly, $\rho$ is the important unknown here. As we mentioned above, no information on $\rho$ other than nonnegativity has been established;  even a uniform $L^\infty$ bound is lacking.   Thus, it is necessary to take an alternate approach, initiated in \cite{BHR_acceleration}. 
We argue by contradiction: 

\begin{itemize}
\item On the one hand, suppose  that the front is spreading ``too fast,''  that is, at least as fast as $\alpha_1 t^{3/2}$ for some $\alpha_1>\alpha^*$.   Roughly, we take this to mean that $\rho$ is uniformly bounded below by  $1/2$ behind $\alpha_1 t^{3/2}$ (the value of the threshold $\mu = 1/2$ matters). 
With this information at hand, $\rho$ can be replaced by $0$ for $x > \alpha_1 t^{3/2}$ and by $1/2$ for $x < \alpha_1 t^{3/2}$, at the expense of $f$ being a \emph{subsolution} of the new problem (because the actual $\rho$ is certainly worse than this crude estimate). In other words, we can replace $\rho$ by $\frac12 \1_{\left\{x< \alpha_1 t^{3/2}\right\}}$ in the  problem \eqref{eq:formal_equation} as in the definition of the variational solution $U_{\alpha,\mu}$ \eqref{eq:minimization}. Therefore, we have $\lim_{h\to0} u_h(1,x,\theta) \geq 
U_{\alpha_1,\frac12}(x,\theta)$. Letting, $h = t^{-1} \ll 1$, we thus expect
\begin{equation}\label{eq:parabolic_HJ_connection}
	f(t, x t^{3/2}, \theta t)
		= f_h(1,x,\theta)
		\lesssim \exp \left(- \frac{U_{\alpha_1,\frac12}(x,\theta)}{h} \right).
\end{equation}
We then notice that  $\min_\theta U_{\alpha_1,\frac12}(\alpha_1,\theta)>0$ by the very definition of $\alpha^*$ \eqref{eq:alpha^*}. This implies that $\rho(t, x t^{3/2}) = \rho_h(1,x)$ is exponentially small around $x=\alpha_1$, which is a contradiction. We recall that the derivation of the above Hamilton-Jacobi equation was only formal. To make this proof rigorous, we use the method of half-relaxed limits that is due to Barles and Perthame~\cite{barles_discontinuous_1987}. 
 
\item On the other hand, suppose that the front is spreading ``too slow,'' that is, no faster than $\alpha_2 t^{3/2}$ for some $\alpha_2<\alpha^*$.   Roughly, we take this to mean that $\rho$ is uniformly bounded above by some   small $\delta$ ahead of $\alpha_2 t^{3/2}$. 
With this information at hand, $\rho$ can be replaced by $\delta $  for $x > \alpha_2 t^{3/2}$ and $+\infty$ for $x < \alpha_2 t^{3/2}$, at the expense of $f$ being a \emph{supersolution} of the new problem (because the actual $\rho$ is certainly better than this crude estimate). In other words, we can replace $-1+\rho(t,x)$ by $-1+\delta+\infty   \1_{\left\{ x<\alpha_2 t^{3/2}\right\}}$ in the variational problem above.  This property implies that 
$\lim_{h\to0} u_h(1,x,\theta) \leq U_{\alpha_2}(x,\theta)+\delta$.
 Choosing $\delta$ small enough, with a similar reasoning as in the previous item, we get another contradiction, since then the front should have emerged ahead of $\alpha_2 t^{3/2}$ because $\min_\theta U_{\alpha_2}(\alpha_2,\theta) +\delta<0$.  While the arguments to prove the emergence of the front are still based on the variational problem, we do not study directly the function $u_h$.  Instead, we build subsolutions on moving, widening ellipses -- these are actually balls following geodesics in the Riemannian metric associated to the diffusion operator -- for the parabolic problem \eqref{eq:intro_toads} following the optimal trajectories in~\eqref{eq:minimization}, using the ``time-dependent'' principle eigenvalue problem of~\cite{BHR_acceleration}.
\end{itemize}

Three important comments are to be made.  First, this argument is made of two distinct pieces: the rigorous connection between the variational formulation~\eqref{eq:minimization} and the parabolic problem \eqref{eq:intro_toads} and the precise characterization of $U_\alpha(\alpha,\theta)$ in order to determine $\alpha^*$.  Second, the effective value of $\rho$ is always small on the right side of $\alpha t^{3/2}$ in both cases (either 0 or small $\delta$), but it takes very different values on the left hand side (either $1/2$ or $+\infty$). Note that $\rho$ is assigned a $+\infty$ value in the absence of any $L^\infty$ bound. At first glance, it is striking that the same threshold $\alpha^*$ could arise in both arguments. What saves the day is that the value of $U_{\alpha,\mu}(\alpha,\theta)$ does not depend on $\mu$ provided that $\mu \geq 1/2$. With our method, the latter bound could be lowered at the expense of more complex computations, but certainly not down to any arbitrary small number. Finally, we make a technical but useful comment. To study the variational problem \eqref{eq:minimization} we often use the following self-similar variables $t = e^s$, $x = t^{3/2} y$,  and $\theta = t \eta$ and study the problem written in terms of such variables. One of  immediate advantage, beyond the compatibility with the problem, is that $\mu \1_{\left\{x< \alpha t^{3/2}\right\}}$, the indicator function in~\eqref{eq:minimization}, becomes $\mu \1_{\left\{y<\alpha \right\}}$, which is stationary. Now, the problem can be seen as the propagation of curved rays in a discontinuous medium with index $\mu$ on the left side, $\left\{y < \alpha\right\}$, and $0$ (or small $\delta$) on the right side $\left\{y\geq \alpha\right\}$.

\subsection*{Optimal trajectories}

In the sequel, the qualitative and quantitative description of the trajectories in~\eqref{eq:minimization} play an important role.  We say that a trajectory $(\bx,\btheta) \in \mA(x,\theta)$ is {\em optimal} if it is a minimizer in~\eqref{eq:minimization} with endpoint $(x,\theta)$. 
We note that the existence and uniqueness of these minimizers is not obvious since the Lagrangian is discontinuous; however, we establish this fact below using lower semi-continuity  (see \Cref{lem:minimizer_existence}).

\begin{figure}
\begin{center}
\includegraphics[width = 0.5\linewidth]{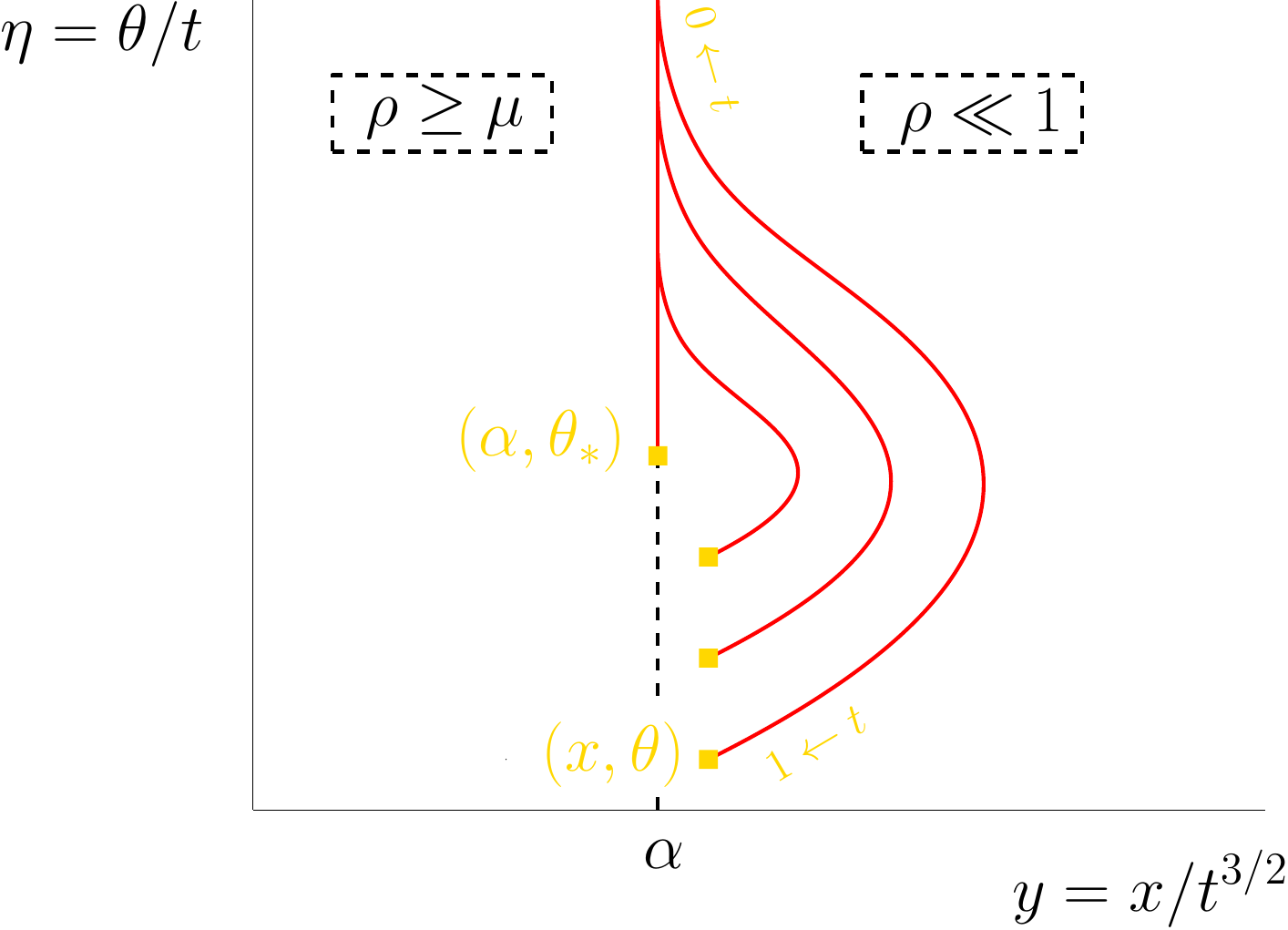} 
\caption{Typical optimal trajectories of \eqref{eq:minimization} depicted in the self-similar variables plane $(y,\eta) = (x/t^{3/2}, \theta/t)$. The endpoint at time $t = 1$ is $(x,\theta)$. The line $\left\{y = \alpha\right\}$ acts as a barrier due to the jump discontinuity in $\rho$ in our argumentation by contradiction. The trajectories with endpoints on the left side of the line $\left\{y = \alpha\right\}$ may come from the right side (not shown). However, the trajectories with endpoints on the right side  never visit the left side. Moreover, for  times $t\to 0$, they  stick to the line $\left\{y = \alpha\right\}$, together with $\beeta\to +\infty$. This behavior holds true if $\alpha\leq \frac43$ and $\mu \geq \frac12$. }
\label{fig:cartoon trajectories}
\end{center}
\end{figure}

Why is it that  $U_{\alpha,\mu}(\alpha,\theta)$ does not depend on (large) $\mu$? The answer lies in the optimal trajectories associated with the variational problem \eqref{eq:minimization}.  It happens that the  optimal trajectories  in~\eqref{eq:minimization} cannot cross, from right to left, the interface $\left\{y = \alpha\right\}$ if the jump discontinuity is too large ($\mu \geq 1/2$).  
In short,  we prove that the trajectories having their endpoint on the right side of the interface (including the interface itself) resemble exactly Figure \ref{fig:cartoon trajectories}. The nice feature is that they never fall into the left-side of the interface, but they ``stick to it'' for a while.  During the proof, we can, thus, replace the minimization problem \eqref{eq:minimization} by the state constraint problem where the curves are forced to stay on the right-side of the interface, so that the actual value of $\mu$ does not matter. 

In fact,  we obtain analytical expressions for the optimal trajectories that lead to the formula for $\alpha^*$ involving polynomials of Airy functions.

\subsection*{Evidence for the hindering phenomenon}

\begin{figure}
\begin{center}
\subfloat[]{
\includegraphics[width = 0.48\linewidth]{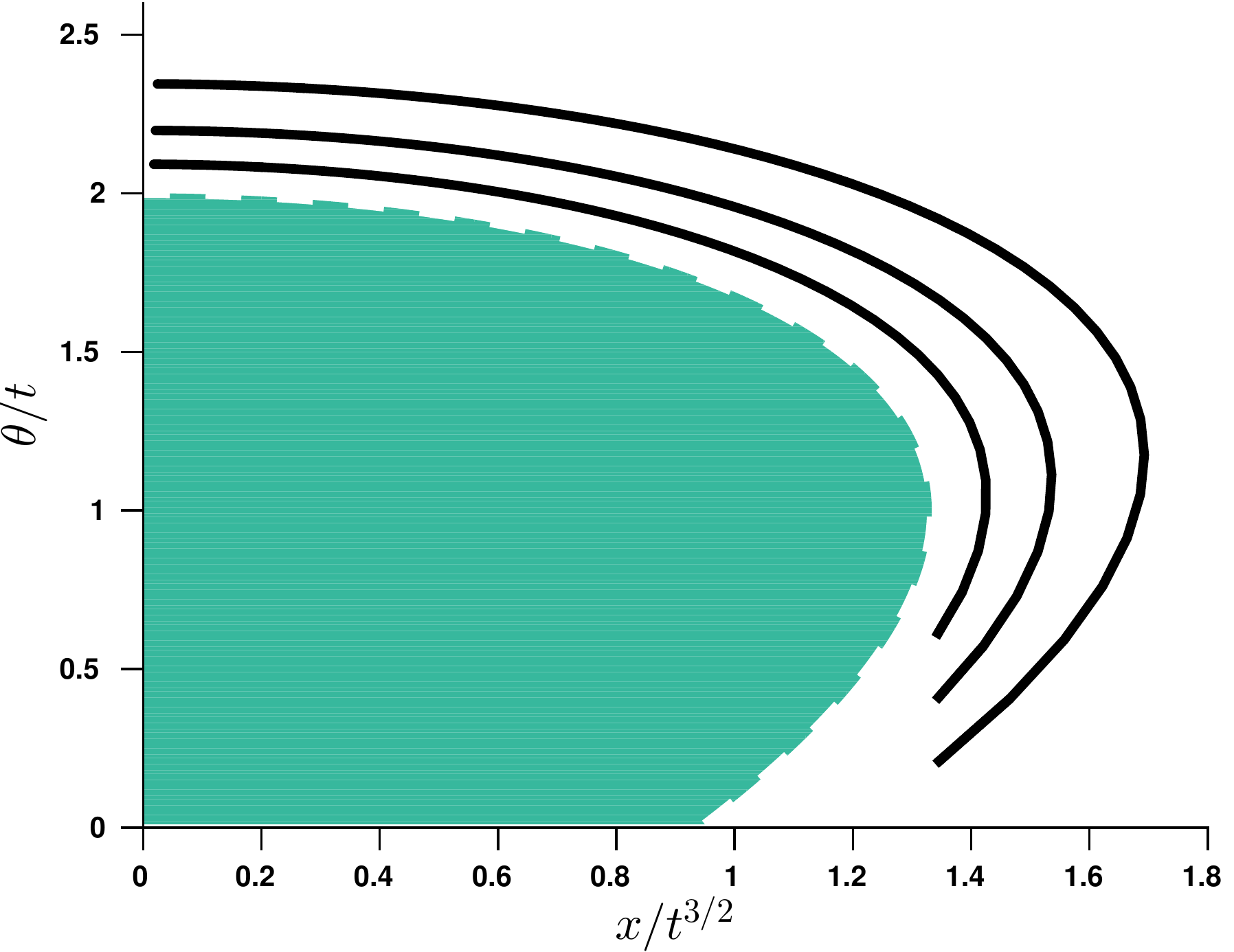}}
\;
\subfloat[]{
\includegraphics[width = 0.48\linewidth]{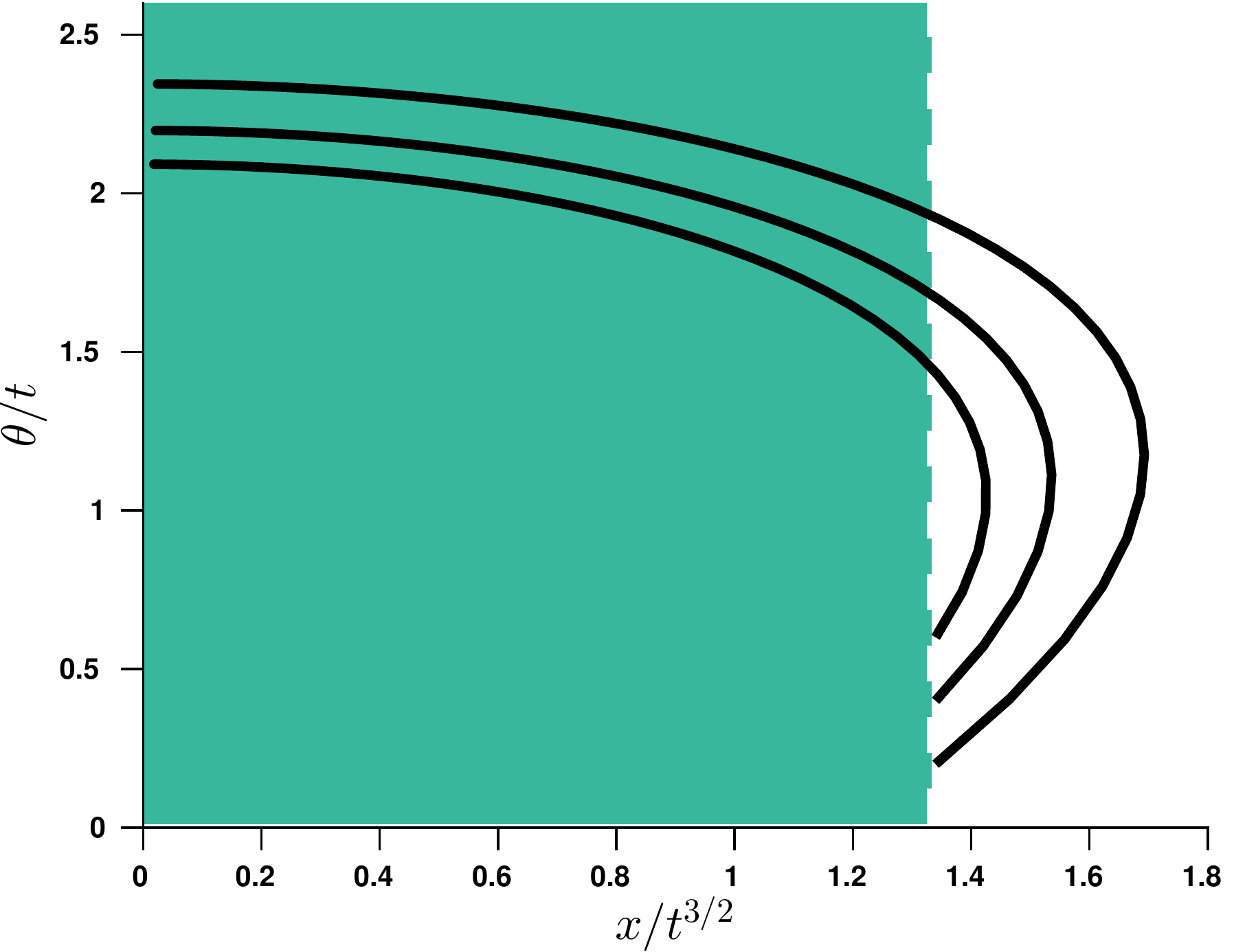}} \\
\subfloat[]{
\includegraphics[width = 0.48\linewidth]{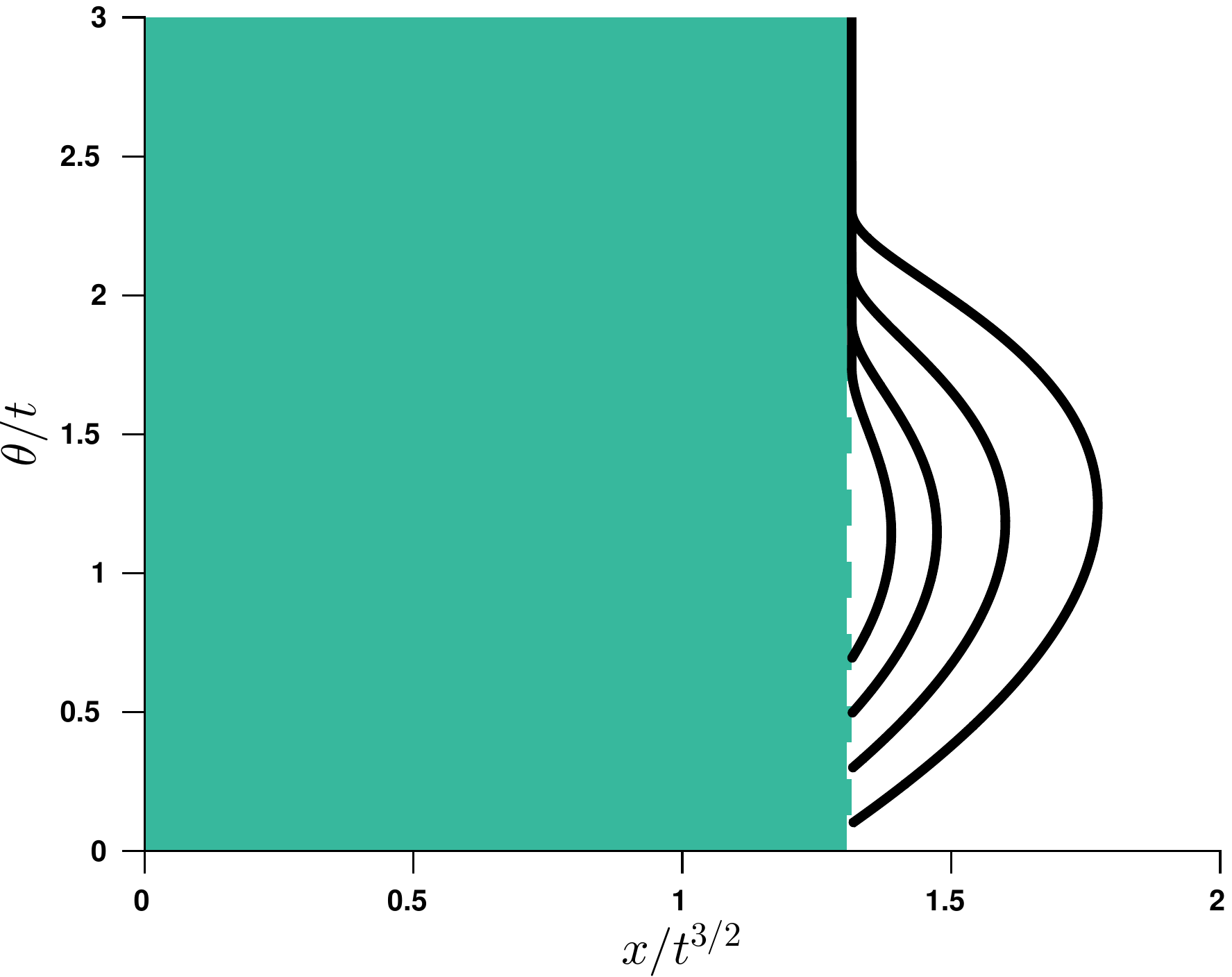}} 
\caption{Illustration of the hindering phenomenon. The green shaded area represents the saturation zone, and the bold lines represent a sampling of optimal trajectories ending beyond the saturation zone. (A) In the case of a local saturation, that is, when $f(1-\rho)$ is replaced by $f(1-f)$ in \eqref{eq:intro_toads}, the saturation zone  $\left\{f\geq \delta\right\}$, for a small $\delta>0$, is genuinely a curved area in the phase plane $(x/t^{3/2},\theta/t)$. The optimal trajectories associated with \eqref{eq:minimization} without saturation ($\mu = 0$) are curved in a similar way. It can be shown that they do not intersect the saturation zone if their endpoint is outside the saturation zone \cite{BHR_acceleration,BerestyckiMouhotRaoul}. (B) In the case of a non-local saturation, the saturation zone $\left\{\rho\geq \delta\right\}$ is a strip along the vertical direction. The main observation is that the optimal trajectories without saturation ($\mu = 0$) intersect the saturation zone. This yields a contradiction as they are computed by ignoring the effect of saturation. (C) The optimal trajectories \textcolor{bl}{of the nonlocal problem} with high enough saturation ($\mu \geq 1/2$) do not intersect the saturation zone. Instead, they stick to the interface for some interval of time. The discrepancy between the ``local'' trajectories (A) and the "non-local" trajectories (C) induces a change in the value function $U_\alpha$, which itself is responsible for the lowering of the critical value from $4/3$ (in the local version) to $\alpha^* \approx 1.315$ (in the non-local version).}
\label{fig:linear determinacy}
\end{center}
\end{figure}

We now explain why the non-local and local saturation act differently.  It is useful to begin by discussing why \textcolor{bl}{in the local saturation problem the speed of propagation is determined by the linear problem}.  Recall that the linear problem is subject to the same asymptotics as~\eqref{eq:minimization} but with the choice of $\mu = 0$ everywhere, simply because saturation has been ignored. The optimal curves \textcolor{bl}{of the linear problem} were computed in \cite{BCMetal}. Rather than giving formulas, we draw them in self-similar variables, see Figure \ref{fig:linear determinacy}. Beneath the trajectories, we also draw the zero level line of the value function $U_{0}$, which separates small from large values of $\bar f$ (the solution of the linear problem).  An important observation is that \textcolor{bl}{an optimal trajectory with ending point at the zone where $\bar f$ is small,} remains on the good side of the curved interface at all intermediate times.  This means that the trajectory stays in the unsaturated zone where the growth is $f(1-f) \approx f$.  Hence, the trajectory only ``sees'' the linear problem \textcolor{bl}{implying that the optimal trajectories of the linear and the nonlinear problem coincide}.

In the non-local problem \eqref{eq:intro_toads}, the characterization of the interface does not involve $\theta$.  Indeed, the saturated region is given by $\rho \approx 1$ and the unsaturated region by $\rho \ll 1$ (or, better, by $\min_\theta U_0(x,\theta) = 0$ and $\min_\theta U_0(x,\theta) > 0$ respectively).  
 However, the trajectories \textcolor{bl}{of the linear problem with ending points on the saturated zone} cannot remain on the right side of any stationary interface                                                                                                                                                                                                                                                                                                                                                                                                                                                                                                                                                      as \textcolor{bl}{illustrated in Figure \ref{fig:linear determinacy}}. 
  Therefore, the saturation term does matter, and it is expected that the location of the interface is a delicate balance between growth, dispersion and saturation.

This impeding phenomena could be rephrased in a more sophisticated formulation, saying that Freidlin's (N) condition~\cite{Freidlin2} is not satisfied. Indeed the optimal trajectories \textcolor{bl}{of the linear problem} ending ahead of the front do not stay ahead at all intermediate times. Therefore, they must have experienced saturation at some time, and so it is not possible to ignore it.

What is more subtle in our case (and leads to explicit results), is that the optimal trajectories \textcolor{bl}{of} the non-local problem hardly experience saturation, as can be viewed on Figure \textcolor{bl}{\ref{fig:linear determinacy}(C)}: they get deformed by the presence of the putative saturated area, but they do not pass through it so that a uniform lower bound $\rho\geq 1/2$ in the saturation zone is sufficient to compute all important features explicitly.

\subsection*{Connection with sub-Riemannian geometry}

The connection between $f$ and $U_\alpha$ that is seen, for example, in~\eqref{eq:parabolic_HJ_connection} solicits some comment about a connection with geometry that was first leveraged in~\cite{BouinChanHendersonKim}. We ignore the zeroth order term in \eqref{eq:intro_toads} and focus on the diffusion part of the equation. Anticipating the details of the proof in \Cref{sec:contradiction}, let $t\in (0,T)$ for $T\gg1$ (fixed), and consider the rescaling $t = T^2\tau$, $x = T^{3/2} X $ and $\theta  = T \Theta $. Notice the anomalous $T^2$ in the change of time, so that $\tau$ is small, $\tau<1/T$. The diffusion part of the equation   \eqref{eq:intro_toads} does not change due to the homogeneity of the second order operator, 
\begin{equation}\label{sec:hypo}
F_{\tau} =    \Theta F_{XX} + F_{\Theta\Theta}\, ,  \quad \tau\in (0,1/T)\,,\;X\in \R \,,\; \Theta\in (1/T,\infty) \, ,  
\end{equation}
and the initial data shrinks to the indicator function $\1_{(-\infty, \mathcal O(1/T^{3/2})] \times (0,\mathcal O(1/T))}$. In particular, the problem is not uniformly elliptic in the limit $T\to +\infty$. However, it is hypoelliptic in the sense of H\"{o}rmander. Moreover, it is a Gurshin operator as the sum of the squares of $\sqrt{\Theta} \partial_X$ and $\partial_\Theta$ respectively. In particular, it satisfies the strong H\"{o}rmander condition of hypoellipticity.

Therefore, after appropriate rescaling, our problem relies on short time asymptotics of the hyppoelliptic heat kernel \eqref{sec:hypo}. Precise results are known since the 1980's. In particular, from L\'eandre~\cite{leandre_majoration_1987,leandre_minoration_1987}, see also \cite{ben_arous_developpement_1988}, we find
\begin{equation*}
\lim_{\tau\to 0} -\tau \log P(\tau,(X_1,\Theta_1),(X_0,\Theta_0)) = \mathrm{dist}((X_1,\Theta_1),(X_0,\Theta_0))^2,
\end{equation*}
where $P$ is the heat kernel associated to~\eqref{sec:hypo} and $\mathrm{dist}$ is the geodesic distance associated with the appropriate sub-Riemannian metric, which coincides with \eqref{eq:minimization} up to the zeroth order terms.  In particular, we find
\[
	F(\tau,X, \Theta) \approx \exp\left(- \frac{\mathrm{dist}(X,\Theta, (0,0))^2}{\tau}\right)
		= \exp\left(- T \mathrm{dist}\left(\frac{x}{T^{3/2}}, \frac{\theta}{\textcolor{bl}{T}}, (0,0)\right)^2\right)\,,
\]
where the second equality follows by reversing the scaling at time $t = T$.  Notice that this is the formulation of~\eqref{eq:parabolic_HJ_connection} but in the absence of reaction terms.

\section{The propagation rate}

\label{sec:contradiction}

\subsection{Some basic properties of trajectories and the proof of Proposition \ref{prop:alpha}}
\label{sec:prop:alpha}
In this subsection, we collect some results about $U_\alpha$ along with the associated optimal trajectories. In particular, we state two lemmas, which are the main elements of the proof of Proposition \ref{prop:alpha}. We also state a proposition that is crucial in the proof of \Cref{thm:propagation}.   The proofs of these facts may be found in \Cref{sec:basic_trajectories}.

First, we note that minimizing trajectories exist.  The uniqueness of the minimizer associated with an endpoint $(x,\theta)\in [\alpha,\infty)\times\R_+^*$ is addressed in \Cref{sec:alpha^*}.
\begin{lemma}\label{lem:minimizer_existence}
Fix any $\alpha \in [0,4/3]$ and $\mu \geq 0$. 
Fix any endpoint $(x,\theta) \in [\alpha,\infty)\times\R_+^*$.  There exists a minimizing trajectory $(\bx,\btheta) \in \mA(x,\theta)$ of the action $U_{\alpha,\mu}(x,\theta)$.
\end{lemma}
Second, we provide a lemma that implies Proposition \ref{prop:alpha}.
\begin{lemma}\label{lem:well-defined-alpha}
	For $\alpha \in [0,4/3]$ and $x\geq \alpha$, the map $U_\alpha(x, \theta)$ is increasing in $\alpha$ and is strictly increasing in $x$.  Hence, $\min_\theta U_\alpha(\alpha,\theta)$ is strictly increasing in $\alpha$.  Further $\min_\theta U_{4/3}(4/3, \theta) >0$ and $\min_\theta U_{5/4}(5/4,\theta) <0$.
\end{lemma}

Next, we show that the optimal trajectories with endpoints on the right side of the front always stay to the right of the front.  This is crucial, since, if this were not true, a uniform upper bound on $\rho$ would be required in order to proceed.  
\begin{proposition}\label{lem:good_trajectories}
	Let $\alpha \in [0,4/3]$ and $\mu \geq 1/2$.  Let $(x,\theta)$ be the endpoint of a minimizing trajectory $(\bx, \btheta)$ with $x \geq \alpha$ and $\theta > 0$.  Then, for all $t \in [0,1]$, $\bx(t) \geq \alpha t^{3/2}$. \textcolor{bl}{As a consequence, $U_{\alpha,\mu}(x,\theta)= U_{\alpha }(x,\theta)$ for all $x\geq \alpha$ and $\mu\geq \frac 12$.}
\end{proposition}

For the purposes of the proof in the next section, we also mention a technical result that is established after a careful description of the minimizing trajectory associated with any endpoint $(\alpha,\theta)$.  We show (cf.~\Cref{lem:anomalous}) that \textcolor{bl}{the optimal trajectories are} such that, for $t \ll 1$,
\begin{equation}\label{eq:theta scaling sec 3}
\begin{cases}
\bx(t) = \alpha t^{3/2}\medskip
\\
\btheta(t)\sim \dfrac32 \alpha^{2/3}\; t |\log t|^{1/3} \end{cases}
\quad \mathrm{as}\; t\to 0. 
\end{equation}
We make two comments.  First, such an anomalous scaling is not obvious at first glance. In fact, it arises when the optimal trajectory comes into contact with the barrier $\left\{x = \alpha t^{3/2}\right\}$. Second, we do not believe such an elaborate result is required in the proof in the following section; however, as the result was readily available, we use it.

\subsection{Proof of the lower bound in \Cref{thm:propagation}}

 The proof of the lower bound in \Cref{thm:propagation} follows almost along the lines of the work in~\cite{BHR_acceleration}.  

\begin{proof}
\textbf{Step one: definition of some useful trajectories.} Fix $\e>0$. 
Using the definition of $\alpha^*$ \textcolor{bl}{and} \Cref{lem:well-defined-alpha}, there exists $r >0$ and $\bx, \btheta$, depending only on $\e$, such that $\bx(0) = 0$, $\btheta(0) = 0$, $\bx(1) = \alpha^* - \e/2$, and 
\begin{equation}\label{eq:lb_Hamiltonian}
	 \int_0^1 L_{\alpha^*-\frac\e2}(s,\bx(s), \btheta(s), \dot \bx(s), \dot\btheta(s)) ds \leq - r .
\end{equation}
One may worry about the behavior of the integral as $s\ll 1$, but we see that the peculiar behavior \eqref{eq:theta scaling sec 3} guarantees that $L_{\alpha^*-\frac\e2}(s,\bx(s), \btheta(s), \dot \bx(s), \dot\btheta(s))$ is integrable at $s = 0$. By a density argument, up to reducing the value of $r>0$, we may assume that $\bx, \btheta\in \mathcal C^2([0,1])$, keeping the behavior $\btheta(t) \geq  C t$ for some constant $C>0$ as $t\to 0$.
In addition, from Proposition \ref{lem:good_trajectories} we get that
\begin{equation}\label{eq:lemma3.3}
 \bx(s) \geq (\alpha^*-\e/2)s^{3/2} \quad \text{ for all $s\in [0,1]$}.
\end{equation}
For $T>0$, $x_0\in \R$, $\theta_0\geq 1$, and $t\in [0,T]$, define the scaled functions, 
\begin{equation}\label{eq:trajectories}
	\bX_{T,x_0}(t) = T^{3/2} \bx\left(\frac{t}{T}\right)+x_0
		\quad \text{and} \quad
	\bTheta_{T,\theta_0}(t) = T \btheta\left( \frac{t}{T}\right) +\theta_0.
\end{equation}
The parameters $x_0$ and $\theta_0$ are determined in the sequel.  For notational ease, we refer to $\bX_{T,x_0}$ and $\bTheta_{T,\theta_0}$ simply as $\bX$ and $\bTheta$ in the sequel. 
By changing variables in~\eqref{eq:lb_Hamiltonian} and using the definition of $L_{\alpha^*-\frac\e2}$, we notice the crucial fact,
\begin{equation}
\label{eq:consq of L bd}
\int_0^T \frac{|\dot{\bX}(t)|^2}{4\bTheta(t)}+\frac{|\dot{\bTheta}(t)|^2}{4}\, dt\leq T -r  T.%
\end{equation}
Further, we may assume without loss of generality that $\dot\bTheta(t) \geq 0$ for all $t \in[0,T]$.

\noindent\textbf{Step two: a subsolution in a Dirichlet ball along the above trajectories.} Let
\begin{equation*}
\delta = r/3. 
\end{equation*}
We now argue by contradiction.  Assume that~\eqref{eq:lower_bound} does not hold.  Then there exists $t_0$ such that, for all $t\geq t_0$,
\begin{equation}\label{eq:lb_contradiction}
	\rho(t,x) < \delta  \quad \text{for all}\quad x \geq (\alpha^* - \epsilon) t^{3/2}.
\end{equation}
We may assume, by simply shifting in time, that $t_0=0$ and that $f_0$ is positive everywhere.  Further, using~\eqref{eq:lb_contradiction}, we have,
\begin{equation}\label{eq:super_solution}
	f_t \geq \theta f_{xx} + f_{\theta\theta} + f\left (  (1-\rho) \1_{ \left\{x<(\alpha^* - \e)t^{3/2}\right\}} + 
	(1-\delta) \1_{\left\{ x \geq (\alpha^* - \e)t^{3/2}\right\}}\right ).
\end{equation}
Next we find a subsolution of~\eqref{eq:super_solution} in a Dirichlet ball that moves along the above trajectories. To this end, we  define 
$$
E_{x_0,\theta_0,R} := \left\{(x,\theta): |x-x_0|^2/\theta_0 + |\theta-\theta_0|^2 \leq R^2\right\}.
$$ 
We use the following lemma, which is very similar to \cite[Lemma 4.1]{BHR_acceleration} (see also \cite[Lemma 13]{BouinChanHendersonKim}). Its proof is postponed, but we use it now to conclude the proof of the theorem.  

\begin{lemma}
\label{lem:vsubsoln}
Let $\delta$, $\bX$, and $\bTheta$ be as above.  There exists  positive constants $C(\delta)$, $C(R,\delta)$, and $\omega(R)$ such that, for all  $R\geq C(\delta)$, $\theta_0\geq C(R,\delta)$, and $T\geq C(R,\delta)$, and for all $x_0\in \R$, there is a function $v$
satisfying
\begin{equation}\label{eq:v-subsol}
\begin{cases}
v_t \leq \theta v_{xx} + v_{\theta\theta}+(1-\delta)v, & (t,x,\theta) \in (0,T)\times E_{\bX(t),\bTheta(t),R},\medskip\\
v(t,x,\theta)=0, & (t,x,\theta) \in [0,T]\times \partial E_{\bX(t),\bTheta(t),R},\medskip\\
v(0,x,\theta) \leq 1, & (x,\theta) \in E_{x_0,\theta_0,R},
\end{cases}
\end{equation}
such that $v(T,x,\theta) \geq \omega(R) e^{\delta T}$ for all $(x,\theta) \in E_{\bX(T),\bTheta(T),R/2}$.
\end{lemma}

We aim to apply \Cref{lem:vsubsoln}.  To that end, choose $R>\max\left\{1,C(\delta)\right\}$ and then $\theta_0 > C(\delta,R)$.  Let $T\geq C(\delta,R)$ be arbitrary.

Next, we find $x_0$ that is independent of $T$ such that, for all $t\in [0,T]$,
\begin{equation}
\label{ball-good-side}
E_{\bX(t),\bTheta(t),R} \subset \left\{(t,x,\theta) |  (\alpha^{*}-\e)t^{3/2} \leq x \right\}.
\end{equation}
Let $(x, \theta)\in E_{\bX(t),\bTheta(t),R}$. Then,
$$
\bX(t) - R(\bTheta(t))^{1/2}\leq x.
$$
Hence, for \eqref{ball-good-side} to hold it is   enough to show that, for all $t\in [0,T]$,
\begin{equation}
\label{eq:implies ball good side}
 (\alpha^{*}-\e)t^{3/2} \leq \bX(t)-R(\bTheta(t))^{1/2}.
\end{equation}
From \eqref{eq:lemma3.3} and \eqref{eq:trajectories}, we see that 
\begin{equation}
\label{eq:Xatright}
(\alpha^*-\e/2)t^{3/2} + x_0 \leq \bX(t) \quad \text{ for all $t\in [0,T]$}.
\end{equation}
Since $\btheta$ is Lipschitz continuous, then there exists a constant $A$, independent of $t$ and $T$, such that $\bTheta(t)\leq At+\theta_0$. Thus, we can choose $x_0$ large enough, independently of $t$ and $T$, such that 
$$
 -\frac{\e}{2}t^{3/2}  +R(\bTheta(t))^{1/2} \leq x_0 \quad \text{ for all $t\in [0,T]$}.
$$
The combination of this and~\eqref{eq:Xatright} \textcolor{bl}{implies} (\ref{eq:implies ball good side}), and hence (\ref{ball-good-side}) holds.

\noindent\textbf{Step three. Obtaining a contradiction.} With the choice of $x_0$ such that \eqref{ball-good-side} holds, we then define
\[
	\beta = \frac12\left ( \min_{E_{x_0,\theta_0,R}} f(0,x,\theta)\right )>0,
\]
and the subsolution $v_\beta = \beta v$ given by Lemma \ref{lem:vsubsoln}.

According to \eqref{ball-good-side} and \eqref{eq:super_solution}, $f$ is a supersolution to the linear parabolic equation satisfied by $v_\beta$ in \eqref{eq:v-subsol}. In addition, 
\[
f(t,x,\theta)> v_\beta (t,x,\theta), \quad \text{for\;} (t,x,\theta) \in \left ( [0,T]\times \partial E_{\bX(t),\bTheta(t),R}\right ) \cup\left ( \left\{0\right\} \times E_{x_0,\theta_0,R}\right ).
\]
From the comparison principle we deduce that $f \geq v_\beta $ in $[0,T]\times  E_{\bX(t),\bTheta(t),R}$. In particular, 
\[
f(T,x,\theta) \geq 
	\beta \omega(R) e^{\delta   T} \quad \text{in }E_{\bX(T),\bTheta(T),R/2}\,.
\]
The previous line, together with the definitions of $\rho$ and $E_{\bX(t),\bTheta(t),R/2}$, yields,
\begin{align*}
\rho(T,\bX(T))&\geq \int_{\bTheta(T)-R/2}^{\bTheta(T)+R/2}\beta \omega(R) e^{\textcolor{bl}{\delta T}}\, d\theta = \textcolor{bl}{\beta}\omega(R) R e^{\delta T}.
\end{align*}
As the constant $\omega(R)$ depends only on $R$, we can enlarge the value of $T$ such that $\rho(T,\bX(T)) \geq \beta \omega(R) R e^{\delta T} \geq 2 \delta$. This is a contradiction, as the combination of \eqref{eq:lb_contradiction} and \eqref{eq:Xatright}, evaluated at $t = T$, \textcolor{bl}{implies} that  $\rho(T,\bX(T))<\delta$.
\end{proof}

Finally, we establish Lemma \ref{lem:vsubsoln}. The proof is very similar to those of \cite[Lemma 4.1]{BHR_acceleration} and \cite[Lemma 13]{BouinChanHendersonKim}; however, it does not immediately follow from either, so we provide a sketch. To this end, we need the following auxiliary lemma.
\begin{lemma}
\label{lem:aux to v subsoln}
Let $\delta>0$. Let 
\[
A(t,y,\eta)=\frac{y}{2}\frac{\dot{\bTheta}(t)}{\bTheta(t)} - \eta \frac{\dot{\bX}(t)}{(\bTheta(t))^{3/2}},  \   \    D(t,y,\eta)= 1+\frac{\eta}{\bTheta(t)}, \text{ and }
\mathcal{L}_t = A\partial_y+D\partial_{yy}+\partial_{\eta \eta}.
\]
There exists a constant $C'(\delta)$ such that, for all $R\geq C'(\delta)$, $\theta_0\geq C'(\delta)$, and $T\geq C'(\delta)$, then there is a function $w(t,y,\eta)$ satisfying
\begin{align}
\label{eq:w}
\partial_tw -\mathcal{L}_t w\leq \delta w  \quad&\text{ in } (0,T)\times B_R(0,0),\\
\label{eq:w on bdry}
w(t,y,\eta)=0 \quad&\text{ on }[0,T]\times \partial B_R(0,0),\\
\label{eq:w ub}
w(0,y,\eta)\leq 1 \quad&\text{ on }B_R(0,0),
\end{align}
and
\begin{equation}
\label{eq:w lb}
\min_{(y,\eta)\in B_{R/2}(0,0)}w(T,y,\eta)\geq \omega'(R),
\end{equation}
where $\omega'(R)$ depends only on $R$.
\end{lemma}
\begin{proof}
This is essentially a restatement of \cite[page 745]{BHR_acceleration}, which in turn uses \cite[Lemmas 5.1, 5.2]{BHR_acceleration}. What we denote $w$ here is denoted by $w_{T,H}$ in \cite{BHR_acceleration}. The only thing we need to verify is that the hypothesis of \cite[Lemma 5.1]{BHR_acceleration} holds in our situation. That is, we must verify
\[
\quad \lim_{|T| + |\theta_0| \rightarrow\infty}||A||_{L^\infty((0,T)\times B_R(0,0))}+||D-1||_{L^\infty((0,T)\times B_R(0,0))} = 0 \quad \text{ for all $R$}.
\]
We show that the second term  in $A$ converges to zero; the rest are handled similarly (in fact, more easily).  Using the definitions of  $\bX$ and $\bTheta$ \eqref{eq:trajectories}, we find,
\begin{equation}
\label{eq:term of A}
\frac{\dot{\bX}(t)}{(\bTheta(t))^{3/2}} = \frac{T^{1/2}\dot{\bx}\left(t/T\right)}{\left(T\btheta\left(t/T\right) + \theta_0\right)^{3/2}}.
\end{equation}
Next, according to the choice of the reference trajectory $(\bx,\btheta)$, there exists a constant $C$ such that 
\begin{equation}
\label{eq:Theta near 0}
\bx(t) = (\alpha^* - \epsilon/2) t^{3/2}\, , \quad\text{and}\quad  \btheta(t) \geq  C t \quad \text{ as $t\to 0$}.
\end{equation}
When $t/T$ is small, we use Young's inequality to see that $t^{1/3} \theta_0^{2/3} \leq t/3 + 2 \theta_0/3$ and, thus, find
\[
\frac{\dot{\bX}(t)}{(\bTheta(t))^{3/2}} \leq  C\frac{T^{1/2}\left(t/T\right)^{1/2}}{\left(T\left(t/T\right) + \theta_0\right)^{3/2}} = C\frac{t^{1/2}}{(t+\theta_0)^{3/2}} \leq \frac{C}{\theta_0} .
\]
Notice that this tends to zero as $\theta_0 \to \infty$.    
When $t/T$ is away from 0, $\textcolor{bl}{\btheta} (t/T)$ is uniformly strictly positive, and so \eqref{eq:term of A} converges to zero as $T\to \infty$. 
\end{proof}

\begin{proof}[Proof of Lemma \ref{lem:vsubsoln}]
Before beginning, we point out that, using \eqref{eq:Theta near 0}, as in the proof of Lemma \ref{lem:aux to v subsoln}, we find that there exists  a constant $\bar{C}(R)$ that depends only on $R$ such that
\begin{equation}
\label{eq:Cbar}
\frac{1}{2}\left\lvert\left\lvert y \frac{\dot{\bX}(t)}{(\bTheta(t))^{1/2}}  +\eta \dot{\bTheta}(t) \right\lvert\right\lvert_{L^\infty((0,T)\times B_R(0,0))}\leq \bar C(R).
\end{equation}
Let $w$ be as given by Lemma \ref{lem:aux to v subsoln}. Define, for $(y,\eta) \in B_R(0,0)$,
\begin{equation*}
\widetilde{v}(t,y,\eta) = w(t,y,\eta)\exp\left ( -\frac{1}{2}\left(y \frac{\dot{\bX}(t)}{(\bTheta(t))^{1/2}}  +\eta \dot{\bTheta}(t)\right) - \bar{C}(R) -g(t)\right ),
\end{equation*}
where
\[
g(t) =-t+2\delta t+ \int_0^t \frac{|\dot{\bX}|^2}{4\bTheta} +\frac{|\dot{\bTheta}|^2}{4} + R\left(\frac{|\ddot{\bX}|}{2\bTheta^{1/2}}+ \frac{|\dot{\bX}\dot{\bTheta}|}{4\bTheta^{3/2}} +\frac{|\dot{\bX}|^2}{4\bTheta^2}+\frac{\textcolor{bl}{|\ddot{\bTheta}|}}{2}\right)\, dt'.
\]
A direct computation, together with the fact that $w$ is a subsolution of (\ref{eq:w}), shows that $\widetilde{v}$ is a subsolution of 
\[
\widetilde{v}_t - \left(\frac{y}{2}\frac{\dot{\bTheta}}{\bTheta} + \frac{\dot{\bX}}{\bTheta^{1/2}}\right) \widetilde{v}_y - \dot{\bTheta} \widetilde{v}_\eta  \leq D \widetilde{v}_{yy} + \widetilde{v}_{\eta\eta} + (1-\delta)\widetilde{v}.
\]
In addition, according to (\ref{eq:w on bdry}) we have $\widetilde{v}\equiv 0$ on $\partial B_R(0,0)$. Also, by the definition of $\widetilde{v}$, the fact that $g(0) = 0$, \eqref{eq:w ub}, and~\eqref{eq:Cbar}, we have, for $(y,\eta)\in B_R(0,0)$,
\[
\widetilde{v}(0,y,\eta)\leq   \exp\left ( -\frac{1}{2}\left( y \frac{\dot{\bX}(0)}{\theta_0^{1/2}}  +\eta \dot{\bTheta}(0)\right) - \bar{C}(R) \right )\leq 1.
\]
Next we find a lower bound for $\widetilde{v}(T, y, \eta)$ on $ B_{R/2}(0,0)$, for which we first bound $g(T)$ from above. Using (\ref{eq:consq of L bd}), it follows that
\begin{equation*}
g(T)\leq -r T+2\delta T + R\int_0^T \frac{|\ddot{\bX}|}{2\bTheta^{1/2}}+ \frac{|\dot{\bX}\dot{\bTheta}|}{4\bTheta^{3/2}} +\frac{|\dot{\bX}|^2}{4\bTheta^2}+\frac{\textcolor{bl}{|\ddot{\bTheta}|}}{2}\, dt.
\end{equation*}
Applying again (\ref{eq:Theta near 0})  in the manner of the proof of Lemma \ref{lem:aux to v subsoln}, there is another constant $\bar{C}(R)$ (that we do not relabel) such that,
\[
g(T) \leq -r  T +2\delta  T +\bar{C}(R).
\]
Together with (\ref{eq:Cbar}), the definition of $\widetilde{v}$, and (\ref{eq:w lb}), we find, for $(y,\eta)\in B_{R/2}(0,0)$,
\[
\widetilde{v}(T, y, \eta)\geq w(T, y, \eta) \exp\left ( (r - 2\delta)T  - \bar C(R)  \right )  \geq \omega'(R)  \exp\left ( (r - 2\delta)T - \bar C(R)  \right ). 
\]
Finally, we recover that $v$ is the desired subsolution of  (\ref{eq:v-subsol}) by making the change of variables from $\widetilde{v}$ to $v$ in the moving frame; that is, we let
\begin{equation*}
v(t,x,\theta) = \widetilde{v}\left ( t, \dfrac{x - \bX(t)}{(\bTheta(t))^{1/2}}, \theta - \bTheta(t) \right )\, . 
\end{equation*}
We recover the last conclusion in \Cref{lem:vsubsoln} by letting $\omega(R) = \omega'(R) e^{-\bar C(R)}$, concluding the proof.
\end{proof}

\subsection{Proof of the upper bound in \Cref{thm:propagation}}

\begin{proof}
We wish to prove by contradiction that, for all $\epsilon>0$, 
\begin{equation*}
	\liminf_{t\to\infty} \left ( \frac{\underline{X}_{1/2}(t)}{t^{3/2}} \right ) \leq \alpha_* + \epsilon  .
\end{equation*}
Suppose on the contrary that there exists $\epsilon>0$ and $t_0$ such that, for all $t\geq t_0$ and all $x\leq (\alpha^* + \e) t^{3/2}$,
\begin{equation}\label{eq:contradiction_assumption}
	\rho(t,x) \geq 1/2.
\end{equation}
In this case, we see that, for all $t \geq t_0$,
\begin{equation}\label{eq:upper_bound_ssoln}
	f_t \leq \theta f_{xx} + f_{\theta\theta} + \left(1 - \frac{1}{2}\1_{ \left\{x < (\alpha^*+\e )t^{3/2}\right\}} \right)f.
\end{equation}
The work in~\cite[Section 3]{BHR_acceleration} implies that there exists a constant $C$, dependingly only on $C_0$ in~\eqref{eq:initial_data}, such that
\begin{equation*}
	f(t,x,\theta) \leq C \exp\left(t - \frac{\psi(x,\theta)}{4(t + 1)}\right).
\end{equation*}
Here $\psi$ is a positive function, defined piecewise in~\cite[Section 3]{BHR_acceleration}, whose exact form is unimportant, but which is positive when $\max\left\{x,\theta\right\} > 0$ and satisfies, for any $h>0$,
\begin{equation}\label{eq:psi_scaling}
	h^2 \psi(x h^{-3/2}, \theta h^{-1}) = \psi(x,\theta).
\end{equation}
We  use this particular scaling for two purposes.  First, up to shifting in time, we may assume, without loss of generality, that $t_0 = 0$ and that, for all $(x,\theta) \in \R \times (1,\infty)$,
\begin{equation}\label{eq:n_shift}
	f(t,x,\theta) \leq C(t_0) \exp\left(t - \frac{\psi(x,\theta)}{(t + C(t_0))} \right).
\end{equation}
Second, for any small parameter $h>0$, we define 
\begin{equation}\label{eq:Hopf_Cole}
f^h(t,x,\theta) = f\left (\dfrac t h, \dfrac x{h^{3/2}}, \dfrac \theta h\right )\,,\quad \text{and}\quad	u^h = h \log f^h.
\end{equation}
Then, $u^h$ satisfies both the bound
\begin{equation}
\label{eq:u_0^h}
	u^h(t,x,\theta) \leq  h \log C(t_0) - \frac{\psi(x,\theta)}{ C (t+ h C(t_0))},
\end{equation}
 and the equation
\begin{equation*}
	\begin{cases}
\displaystyle 		u^h_t - \theta |u^h_x|^2 - |u^h_\theta|^2 - h \theta u^h_{xx} - h u^h_{\theta\theta} \leq 1 - \frac{1}{2}\1_{\left\{x< (\alpha^*+\e)t^{3/2}\right\}} ,  &\text{in } (0,\infty)\times \R \times (h ,\infty),\smallskip\\
		u^h_\theta = 0, &\text{on } (0,\infty)\times\R\times\left\{h\right\}.
	\end{cases}
\end{equation*}
The differential inequality is due to~\eqref{eq:upper_bound_ssoln}. 
The bound on the initial data comes from~\eqref{eq:psi_scaling} and \eqref{eq:n_shift}.

We define the half-relaxed limit $u^* = \limsup_{h\rightarrow 0} u^h$.  We claim that, for any $\delta>0$, $u^*$  satisfies (in the viscosity sense)
\begin{equation}\label{eq:half_relaxed}
	\begin{cases}
	\displaystyle	u^*_t - \theta |u^*_x|^2 - |u^*_\theta|^2 - 1 + \frac{1}{2}\1_{\left\{x< (\alpha^*+\e/2)t^{3/2}\right\}}  \leq 0, \qquad &\text{in } (0,\infty)\times\R\times(0,\infty),\smallskip\\
\displaystyle		\min\left( -u^*_\theta, u^*_t - \theta |u^*_x|^2 - |u^*_\theta|^2 - 1 + \frac{1}{2}\1_{\left\{x<(\alpha^*+\e/2)t^{3/2}\right\}} \right) \leq 0, &\text{on } (0,\infty)\times\R\times\left\{0\right\},\smallskip\\
		u^*_0 \leq -\infty\1_{D_\delta^c}, &\text{on } \left\{0\right\}\times \R\times(0,\infty),
	\end{cases}
\end{equation}
where $D_\delta = \left\{(x,\theta) \in \R\times \R_+ : \max\left\{x,0\right\}^2 + \theta^2 \leq \delta^2\right\}$.
 We point out that we have reduced $\e$ to $\e/2$.  The first two inequalities follow from standard arguments in the theory of viscosity solutions, see, e.g., \cite[Section~3.2]{HendersonPerthameSouganidis} for a similar setting.  The third inequality follows directly from the upper bound~\eqref{eq:u_0^h} and the fact that $\psi(x,\theta)$ is positive for $\max\left\{x,\theta\right\}>0$.  The restriction to the outside of a ball of radius $\delta$ (for arbitrary $\delta>0$) might look unnecessary. However, in~\cite{CrandallLionsSouganidis}, which is applied in the sequel, only ``maximal functions'' with support on smooth, open sets are considered.

Using~\eqref{eq:half_relaxed} along with theory of maximal functions~\cite{CrandallLionsSouganidis} (see also~\cite{HendersonPerthameSouganidis} for a discussion of the boundary conditions and the degeneracy in the Hamiltonian near the boundary $\left\{\theta =0\right\}$, both of which are not considered in~\cite{CrandallLionsSouganidis}), along with the Lax-Oleinik formula, we see, for all $(x,\theta) \in \R\times \R_+$,
\[\begin{split}
	u^*(t,x,\theta)
		&\leq -\inf \left\{ \int_0^t \left[\frac{|\dot \bx(s)|^2}{4\btheta(s)} + \frac{|\dot\btheta(s)|^2}{4} - \left(1 -  \frac{1}{2}\1_{\left (-\infty,(\alpha^* + \e/2)s^{3/2}\right )}(\bx(s))\right) \right] ds \right.\\
			&\left.\qquad\qquad\qquad : \bx(\cdot), \btheta(\cdot) \in H^1, (\bx(0),\btheta(0)) \in D_\delta(0,0), (\bx(t),\btheta(t)) = (x, \theta)\vphantom{\int}\right\}.
\end{split}\]
Taking the limit $\delta \to 0$ and setting $t=1$, we find
\[\begin{split}
	u^*(1,x,\theta)
		&\leq -\inf \Bigg\{ \int_0^1 \Big[\frac{|\dot \bx(s)|^2}{4\btheta(s)} + \frac{|\dot\btheta(s)|^2}{4} - \Big(1 -  \frac{1}{2}\1_{\left (-\infty,(\alpha^* + \e/2)s^{3/2}\right )}(\bx(s))\Big) \Big] ds : (\bx, \btheta) \in \mA(x,\theta) \Bigg\}\\
		&= - U_{\alpha^* + \e/2,1/2}(x,\theta).
\end{split}\]
Fix any $x \geq \alpha^* + \e/2$.  Using \Cref{lem:good_trajectories}, the trajectory $(\bx,\btheta)$ satisfies $\bx(s)\geq (\alpha^* + \e/2) s^{3/2}$ for all $s \in [0,1]$.  It follows that $U_{\alpha^* + \e/2, 1/2}(x,\theta) = U_{\alpha^*+\e/2}(x,\theta)$, which implies,
\[
	u^*(1,x,\theta) \leq -  U_{\alpha^* + \e/2}(\alpha^* + \e/2, \theta).
\]
For notational ease, let $r = \min_{\theta'} U_{\alpha^* + \e/2}(\alpha^*+\e/2,\theta')$.  According to \Cref{lem:well-defined-alpha} and the definition of $\alpha^*$, we have $r>0$; thus, we find
\[
	u^*(1,x,\theta) \leq
		-r  < 0.
\]

We now use the negativity of $u^*$ to show that $f$ is small beyond $(\alpha^* + \e) t^{3/2}$ for large times, which provides a contradiction.  From the definition of $u^*$, it follows that there exists $h_0>0$ such that if $h \leq h_0$, then, for all $x \in (\alpha^*+2\e/3,2)$ and all $\theta \in (h,4)$,
\[
	f\left (\dfrac 1h,\dfrac x{h^{3/2}}, \dfrac \theta h\right )
		= \exp\left ( \dfrac{u^h(1,x,\theta)}h\right 
	)	\leq \exp\left(-\frac{r}{2h}\right).
\]
Hence, if $t \geq 1/h_0$, $x \in \left ((\alpha^* + 2\e/3)t^{3/2}, 2t^{3/2}\right )$ and $\theta \in (1, 4t)$, then
\[
	f(t,x, \theta)
		\leq \exp\left ( -\dfrac{r t}2\right ),
\]
which implies,
\begin{equation}\label{eq:rho1}
	\int_1^{4t} f(t,x,\theta) d\theta
		\leq  (4t-1) \exp\left ( -\dfrac{r t}2\right ) .
\end{equation}
On the other hand, by \cite[Equation (3.5)]{BHR_acceleration}, there exists a positive constant $C$, depending only on the initial data $f_0$ such that $f(t,x,\theta) \leq C e^{t - \theta^2/4t}$, for all $(t,x,\theta)$.  Hence, 
\begin{equation}\label{eq:rho2}
	\int_{4t}^\infty f(t,x,\theta) d\theta
		\leq \int_{4t}^\infty C \exp\left ( t- \frac{\theta^2}{4t}\right  )d\theta
		\leq C e^{-3t}.
\end{equation}
The combination of~\eqref{eq:rho1} and~\eqref{eq:rho2} implies
\[
	\limsup_{t\to\infty} \left (  \sup_{x \in ((\alpha^* + 2\e/3)t^{3/2}, 2t^{3/2})} \rho(t,x) \right ) = 0.
\]
To rule out the other part of the domain, we apply \cite[Theorem 1.2]{BHR_acceleration}, which implies,
\[
	\limsup_{t\to\infty} \left (  \sup_{x > (4/3)t^{3/2}} \rho(t,x)\right ) = 0.
\]
Combining these two estimates yields
\[
	\limsup_{t\to\infty} \left (  \sup_{x > (\alpha^* + 2\e/3)t^{3/2}} \rho(t,x) \right ) = 0.
\]
This contradicts~\eqref{eq:contradiction_assumption}, since the latter condition implies that
\[
	\liminf_{t\to\infty}\left (  \min_{x \leq (\alpha^* + \e)t^{3/2}} \rho(t,x) \right ) \geq 1/2.
\]
The proof is complete.
\end{proof}

\section{Basic properties of the minimizing problem $U_{\alpha,\mu}$}\label{sec:basic_trajectories}

In this section we prove some basic properties of the trajectories.  Namely, we give the proofs of \Cref{lem:minimizer_existence}, and \Cref{lem:well-defined-alpha}.  We also conclude with the reformulation of the minimization problem in the self-similar variables.

\subsection{The existence of a minimizing trajectory -- \Cref{lem:minimizer_existence}}
\label{sec:41}

The existence of minimizers is a delicate issue due to the discontinuity in the Lagrangian $L_{\alpha,\mu}$. From our qualitative analysis in the sequel, we show that optimal trajectories eventually stick to the line of discontinuity for periods of time. Therefore, the value of the Lagrangian on this line matters. As an illustration of the subtlety of this issue, notice that replacing $\1_{\left\{ x < \alpha t^{3/2}\right\}} $ by $\1_{\left\{ x \leq \alpha t^{3/2}\right\}}$ would break down the existence of minimizers. In the latter case, a minimizing sequence would approach the line without sticking to it (details not shown).

\begin{proof}
Take any minimizing sequence $(\bx_n, \btheta_n) \in \mA(x,\theta)$ such that
\begin{equation}\label{eq:c2}
	U_{\alpha,\mu}(x,\theta)
		= \lim_{n\to\infty} \int_0^1 L_{\alpha,\mu}(t,\bx_n(t),\btheta_n(t), \dot \bx_n(t), \dot\btheta_n(t))\, dt.
\end{equation}
It is clear that $\btheta_n(t)$ remains uniformly bounded above. Further, for any $\epsilon>0$, $\btheta_n(t)$ remains uniformly bounded away from $0$ for all $t \in [\epsilon,1]$.  These two facts are heuristically clear; for a proof see \cite[Appendix A]{HendersonPerthameSouganidis}.

As a result of the uniform upper bound on $\btheta_n$, we obtain a uniform $H^1$ bound on $(\bx_n,\btheta_n)$, implying that, up to extraction of a subsequence, $(\bx_{n}, \btheta_{n})  \rightharpoonup  (\bx,\btheta)$ for some trajectory $(\bx,\btheta) \in H^1$.  This convergence is strong in $\mathcal C^0$ due to the Sobolev embedding theorem.   In addition, because $\btheta_n$ is bounded away from zero and $\btheta_n\to \btheta$ in $\mathcal C^0$, $\btheta$ is bounded away from zero.  It thus follows that $\dot \bx_n /  \sqrt{\btheta_n} \rightharpoonup \dot \bx / \sqrt{\btheta}$.  

From above, we obtain two important facts that allow to conclude.  First, $(\bx, \btheta) \in \mA(x,\theta)$.  Second, using~\eqref{eq:c2}, Fatou's lemma, and the lower semi-continuity of $L_{\alpha,\mu}$, we see that 
\begin{equation*}
\begin{split}
	U_{\alpha,\mu}(x,\theta)
		&= \lim_{n\to\infty} \int_0^1 L_{\alpha,\mu}(t, \bx_n(t), \btheta_n(t), \dot \bx_n(t), \dot \btheta_n(t)) dt\\
		&\geq \int_0^1 \liminf_{n\to\infty} L_{\alpha,\mu}(t, \bx_n(t), \btheta_n(t), \dot \bx_n(t), \dot \btheta_n(t)) dt\\
		& \geq \int_0^1 L_{\alpha,\mu}(t, \bx(t), \btheta(t), \dot \bx(t), \dot \btheta(t)) dt
		\geq U_{\alpha,\mu}(x,\theta).
\end{split}
\end{equation*}
The last inequality follows from the definition of $U_{\alpha,\mu}$ and the fact that $(\bx,\btheta) \in \mA(x,\theta)$.  Hence, the inequalities must all be equalities above, implying that $(\bx,\btheta)$ is truly a minimizing trajectory, which finishes the proof.
\end{proof}

\subsection{Proof that $\alpha^*$ is well-defined -- \Cref{lem:well-defined-alpha}}

\begin{proof}
First we observe that the $U_\alpha(x,\theta)$ is increasing in $\alpha$ simply because $L_\alpha$ is increasing in $\alpha$.  To see the fact that it is strictly increasing in $x$ when $x\geq \alpha$, we fix $x\geq \alpha$,  $\theta>0$ and any $h>0$. Consider an admissible minimizing trajectory $(\bx_h(s),\btheta_h(s))$ such that $(\bx_h(0),\btheta_h(0)) = (0,0)$ and $(\bx_h(1),\btheta_h(1)) = (x+h,\theta)$.

Define $s_h = \sup\left\{s: \bx_h(s) = x\right\}$.  Notice that $s_h$ is well-defined due to the continuity of $\bx_h$ established above, along with the fact that $\bx_h(0) < x < \bx_h(1)$.  We also note that $\bx_h(s) \geq \alpha s^{3/2}$ for all $s\geq s_h$.

We construct a trajectory connecting the origin and $(x,\theta)$.  Let $\bx(s) = \int_0^{\min\left\{s,s_h\right\}} \dot \bx_h(s')ds'$.  It follows from the definition of $s_h$ that $\bx(1) = x$, $\bx_h(s) = \bx(s)$ for all $s \leq s_h$, and $\bx_h(s) = x \geq  \alpha s^{3/2}$ for all $s \in  [ s_h,1]$.  Further, it is clear that $(\bx,\btheta_h) \in \mA(x,\theta)$.  Hence,
\begin{equation}
\label{eq: Ualpha}
\begin{split}
	U_\alpha(x,\theta)
		&\leq \int_0^1 L_\alpha(s, \bx(s), \btheta_h(s), \dot \bx(s), \dot \btheta_h(s)) ds\\
		&= \int_0^{s_h} L_\alpha(s, \bx_h, \btheta_h(s), \dot \bx_h(s), \dot \btheta_h(s)) ds
			+ \int_{s_h}^1 \left[\frac{|\dot\btheta_h(s)|^2}{4} - 1 \right ] ds\\
		&= \int_0^1 L_\alpha(s, \bx_h, \btheta_h(s), \dot \bx_h(s), \dot \btheta_h(s)) ds
			- \int_{s_h}^1  \frac{|\dot \bx_h(s)|^2}{4\btheta_h(s)}  ds.
\end{split}
\end{equation}
Since $\bx_h(1) = x+h > x = \bx_h(s_h)$ and since $s_h < 1$, it follows that
\[
	\int_{s_h}^1 \frac{|\dot \bx_h(s)|^2}{4\btheta_h(s)} ds > 0.
\]
Using these two facts to bound the right-hand side of the last line in (\ref{eq: Ualpha}) from above yields
\[
	U_\alpha(x,\theta)
		< \int_0^1 L_\alpha(s, \bx_h(s), \btheta_h(s), \dot \bx_h(s), \dot \btheta_h(s)) ds
		= U_\alpha(x_0+h,\theta_0),
\]
finishing the proof that $U_\alpha$ is strictly increasing with respect to $x \geq \alpha$.

We now prove that $\min_\theta U_{4/3}(4/3,\theta)>0$. For this, we first recall the particular trajectories that were computed in \cite{BHR_acceleration}, in the case without growth saturation,  i.e., when $\alpha = 0$ (those computations were originally derived for~\cite{BCMetal}, though they are not explicitly written there, so we provide~\cite{BHR_acceleration} as a reference instead).    
It was shown that the minimum of $U_0(4/3,\cdot)$ is reached at $\theta =1$, \textcolor{bl}{with $U_0(4/3,1)=0$}.  Let $(\bx_0,\btheta_0)$ be the optimal trajectory associated with the endpoint $(4/3,1)$. Then, 
$\bx_0$ has the following simple expression: 
\[
\bx_0(t)= \frac{4}{3}\left(\frac{3-t}{2}\right) t^2.
\]
A crucial observation is that $\bx_0$ is always to the left of the barrier associated with $\alpha = 4/3$, i.e.,
\begin{equation}\label{eq:barrier 4/3}
\bx_0(t) < \frac43 t^{3/2} \quad \text{ for all } t\in (0,1).
\end{equation}
Indeed, 
\[
	\frac{4}{3}t^{3/2} - \bx_0(t)
		= \frac{4}{3}t^{3/2} - \frac{4}{3}\left(\frac{3-t}{2}\right) t^2
		= \frac{2}{3} t^{3/2} \left (t^{1/2} - 1\right )^2 \left (t^{1/2} + 2\right ) > 0.
\]
Next, let $(\bx,\btheta)$ be a minimizing trajectory associated with $\alpha = 4/3$, that is,
\[
	\min_\theta U_{4/3}(4/3,\theta)
		= \int_0^1 L_{4/3}(t,\bx(t),\btheta(t),\dot \bx(t),\dot\btheta(t))dt.
\]
There are two options. On the one hand, assume that $(\bx,\btheta) = (\bx_0,\btheta_0)$. 
Then, we deduce from \eqref{eq:barrier 4/3} that saturation is always at play, hence
\[
	\min_\theta U_{4/3}(4/3,\theta)
		= \int_0^1\left[ \frac{|\dot \bx(t)|^2}{4\btheta(t)} + \frac{|\dot \btheta(t)|^2}{4} \right] dt > 0.
\]
On the other hand, assume  that $(\bx,\btheta)\neq (\bx_0, \btheta_0)$.  Then
\[	\min_\theta U_{4/3}(4/3,\theta)
		\geq \int_0^1\left[ \frac{|\dot \bx(t)|^2}{4\btheta(t)} + \frac{|\dot \btheta(t)|^2}{4} -1 \right] dt
		 >   \int_0^1 \left [ \frac{|\dot{\bx_0}(t)|^2}{4\btheta_0(t)}+\frac{|\dot{\btheta_0}(t)|^2}{4}- 1\right ]\, dt 
		 = 0.
\]
Here, the strict inequality follows from the uniqueness of the minimizing trajectory $(\bx_0,\btheta_0)$ for the associated minimizing problem.  This concludes the proof of the positivity of $\min_\theta U_{4/3}(4/3,\theta)$.

The last step consists in proving that $\min_\theta U_{5/4}(5/4, \theta)< 0$. 
To this end, we define a particular trajectory $(\bx,\btheta)\in \mathcal{A}(5/4, 1)$ by,
\[
\bx(t) = \frac{5}{4}t^{3/2}, \   \   
\btheta(t) = 
\begin{cases}
\dfrac{3}{2}t & \text{ for } 0\leq t<\dfrac{1}{3},\medskip\\
\dfrac{3}{4}t+\dfrac{1}{4}& \text{ for } \dfrac{1}{3}<t\leq 1.
\end{cases}
\]
We establish 
\begin{equation}
\label{eq:L5/4}
\int_0^1L_{5/4}(t,\bx(t), \btheta(t), \dot{\bx}(t), \dot{\btheta}(t))\, dt < 0,
\end{equation}
which allows us to conclude.

By construction, we have $\1_{\left (-\infty, (5/4) t^{3/2}\right )}(\bx(t))=0$ for all $t$, and $\dot{\bx}(t)=(15/8)t^{1/2}$, and 
\[
\dot{\btheta}(t) = 
\begin{cases}
\dfrac{3}{2} & \text{ for } 0\leq t<\dfrac{1}{3},\medskip\\
\dfrac{3}{4}& \text{ for } \dfrac{1}{3}<t\leq 1.
\end{cases}
\]
Using this in the definition of $L_{5/4}$ yields,
\begin{align*}
L_{5/4}(t,\bx(t), \btheta(t), \dot{ \bx}(t), \dot{\btheta}(t)) &= 
\begin{cases}
\dfrac{5^2\cdot 3+3^2\cdot 2^3}{2^7} -1 & \text{ for } 0\leq t<\dfrac{1}{3},\medskip\\
\displaystyle \left(\frac{15}{8}\right)^2\frac{t}{3t+1} + \frac{3^2}{2^6}-1  & \text{ for } \dfrac{1}{3}<t\leq 1.
\end{cases}
\end{align*}
Integrating and then rearranging gives,
\begin{align*}
\int_0^1L_{5/4}(t,\bx(t), \btheta(t), \dot{ \bx}(t), \dot{\btheta}(t))\, dt &=\frac{1}{3}\left(\frac{5^2\cdot 3+3^2\cdot 2^3}{2^7}\right) + \frac{2}{3}\frac{3^2}{2^6} + \left(\frac{15}{8}\right)^2\int_{1/3}^{1}\frac{t}{3t+1}\, dt-1
\\& = \frac{61}{ 2^7}+\frac{5^2}{8^2}(2-\ln 2)-1 = \frac{25}{64}\left ( \dfrac{33}{50} - \ln 2 \right )(\approx -.01)<0
\end{align*}
Hence  \eqref{eq:L5/4}   holds.
This concludes the proof of the lemma.
\end{proof}

\subsection{Reformulation of the minimization problem in the self-similar variables}\label{sec:self-similar}

In \eqref{eq:trajectories} and \eqref{eq:Hopf_Cole} we use the scaling properties of our problem. Here, we go a step further, as we reformulate the  minimization problem \eqref{eq:minimization}  in self-similar coordinates. 
We transform each trajectory $(\bx(t),\btheta(t))$ for $t\in (0,1)$ into the new $(\by(s),\beeta(s))$, $s\in (-\infty, 0)$ as follows
\begin{equation*}
\begin{cases}
\bx( t ) &=  t^{3/2}\by(\log t),\medskip\\
\btheta(t) &= t \beeta(\log t) .
\end{cases}
\end{equation*}
Note that the endpoint is not changed: $(\by(0),\beeta(0)) = (x,\theta)$. 
The minimization problem \eqref{eq:minimization} is equivalent to the following one:
\begin{equation}\label{eq:min self sim}
	U_{\alpha,\mu}(x,\theta) = \inf \left\{\int_{-\infty}^0 \bL_{\alpha,\mu}(\by(s), \beeta(s), \dot \by(s), \dot \beeta(s)) e^{s}\, ds : (\by(\cdot),\beeta(\cdot)) \in \bA(x,\theta)\right\},
\end{equation} 
where the autonomous Lagrangian $\bL_{\alpha,\mu}$ is given by
\begin{equation}
	\bL_{\alpha,\mu}(y,\eta,v_y,v_\eta) = \frac{1}{4\eta}\left(v_y + \frac32 y\right)^2 + \frac{1}{4}\left(v_\eta + \eta\right)^2 - 1 +  \mu \1_{\left\{y<\alpha\right\}}\, , 
\label{eq:Lagrangian ss} 
\end{equation}
and the set of admissible trajectories is given by
\begin{multline}	
\label{eq:self-similar_admissible}
\bA(x,\theta) =
		\left\{ (\by,\beeta): \R_- \to \R\times\R_+: \bL_{\alpha,\mu}(\by(s), \beeta(s), \dot \by(s), \dot \beeta(s)) e^{s} \text{ is integrable, and } \vphantom{\lim_{s\to-\infty}}\right.
\\
\left.	 	 \lim_{s\to-\infty} e^{3s/2} \by(s) = 0\,, \lim_{s\to-\infty} e^s \beeta(s) = 0\,, (\by(0),\beeta(0)) = (x,\theta)\right\}.
\end{multline}

In view of the discontinuity in the Lagrangian along the line $\left\{y = \alpha\right\}$, we expect interesting dynamics as $\by(s)$ approaches $\alpha$. We prove in the next section that   the line $\left\{y = \alpha\right\}$ acts as a barrier for the optimal trajectories that end in the area $\left\{y\geq \alpha\right\}$, provided that $\mu$ is not too small  and $\alpha$ is not too large, as stated in \Cref{lem:good_trajectories}.

Due to the natural scaling of the problem, it is often convenient notationally to let
\begin{equation}\label{eq:alpha_bar}
	\overline\alpha = \frac{3\alpha}{4}.
\end{equation}

\section{Qualitative properties of trajectories -- \Cref{lem:good_trajectories}} \label{subsec:proof of good trajectories}

The next result is a reformulation of \Cref{lem:good_trajectories} using the self-similar coordinates introduced in \Cref{sec:self-similar}. 
\begin{lemma}\label{lem:barrier}
Suppose that $2 \mu\geq \overline\alpha^{4/3}$. Let $(x,\theta) \in \R\times\R_+^*$ be an endpoint such that $x\geq \alpha$. Then any optimal trajectory $(\by,\beeta) \in \bA(x,\theta)$ of \eqref{eq:min self sim} satisfies $\by(s)\geq \alpha$ for all $s\in (-\infty,0]$.
\end{lemma}
That is, if $\by$ ends beyond the line $\left\{y = \alpha\right\}$, then it never crosses the line.  It is clear that this is a consequence of the following two lemmas.

\begin{lemma}[No single crossing]\label{lem:No single crossing}
With the same assumptions as in Lemma \ref{lem:barrier}, consider a  trajectory which crosses the line $\left\{y=\alpha\right\}$ only once,  that is,  there exists $s_0$ such that for all $s \in \textcolor{bl}{[s_0,0)}$, $\by(s) \geq \alpha$ and for all $s< s_0$, $\by(s) < \alpha$. Then it cannot be an optimal one.  
\end{lemma}

\begin{figure}
\begin{center}
\includegraphics[width = 0.5\linewidth]{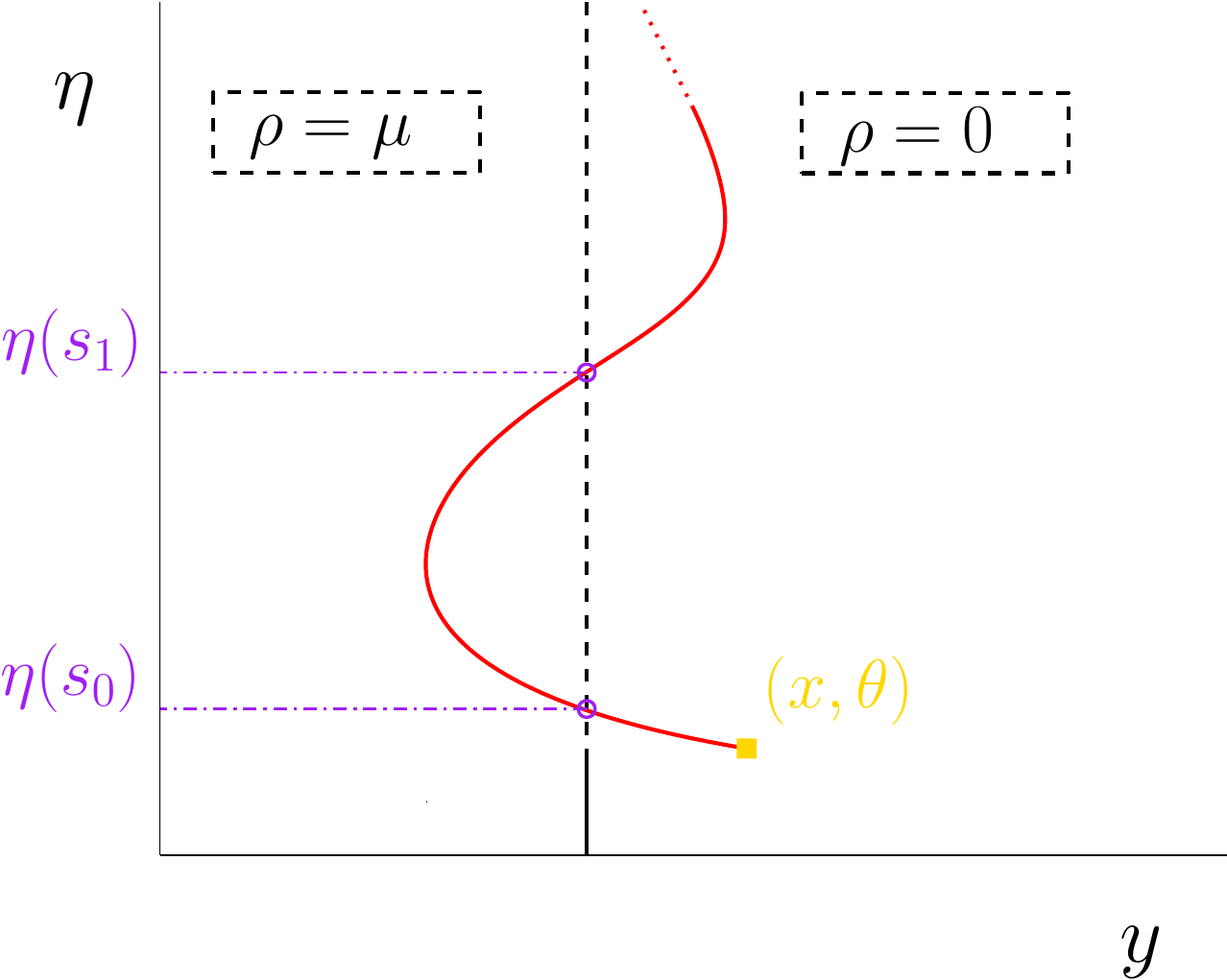} 
\caption{Sketch of a  C-turn as the trajectory crosses the line twice. From \Cref{lem:No single crossing}, we see that this trajectory cannot be optimal.}
\label{fig:downward uturn}
\end{center}
\end{figure}

\begin{lemma}[No C-turn]\label{lem:No downward U-turn}
With the same assumptions as in Lemma \ref{lem:barrier}, consider a trajectory $(\by,\beeta) \in \bA(x,\theta)$, which crosses the line $\left\{y=\alpha\right\}$ at least twice (see \Cref{fig:downward uturn}), i.e.~there exists $s_1<s_0\leq 0$ such that $\by(s_0) = \by(s_1) = \alpha$ and $\by(s) < \alpha$ for all $s \in (s_1,s_0)$. 
Then it cannot be an optimal one.
\end{lemma}

The proof of \Cref{lem:No downward U-turn} uses the following result that deals with the monotonicity of $\beeta$ for any optimal trajectory:

\begin{lemma}[Monotonicity of $\beeta$]\label{lem:monotonicity}
If $(\by,\beeta)$ is an optimal trajectory for \eqref{eq:min self sim}, then $\beeta$ is nonincreasing over $(-\infty,0)$. 
\end{lemma}

The proof of Lemma \ref{lem:monotonicity} is a direct consequence of the Hamiltonian dynamics associated with \eqref{eq:min self sim}. We review it briefly in the next section. The other two statements require additional conditions on $\alpha$ and $\mu$, as in Lemma \ref{lem:barrier}. They are proved in Section \ref{sec:barrier}.

\subsection{A brief overview of Hamiltonian dynamics}
\label{sec:hamiltonian}

In this section we provide some elements of the computation of the optimal trajectories that we use in the article. To this end,  it is instructive to briefly recall the basics of calculus of variations in a smooth setting. 
Let $L(X,V)$ be some smooth  Lagrangian function. Consider, for some admissible set of trajectories $\mA$ with endpoints at $x\in \R^d$, the following problem:
\begin{equation}\label{eq:abstract min}
	U(x)= \inf_{X \in \mA}\ \int_{-\infty}^0 L(X(s),\dot X(s)) e^s\, ds\, .
\end{equation}
When $L$ is smooth, one can write the Euler-Lagrange equation,
\begin{equation*}
\dfrac{d}{ds}\left (  D_V L(X(s),\dot X(s)) e^s \right ) =  D_X L(X(s),\dot X(s)) e^s\,,
\end{equation*}
for an optimal trajectory $X$. As usual, the Hamiltonian $H(X,P)$ and the Lagrangian $L(X,V)$ are related by the following convex duality:
\begin{equation*}
	H(X,P) = \sup_V ( V\cdot P - L(X,V))
	\quad \text{ and } \quad
	L(X,V) = \sup_P ( V\cdot P - H(X,P)).
\end{equation*}
Hence, the action variable,  defined as $P(s) =  D_V L(X(s),\dot X(s))$,  satisfies the following Hamiltonian system, together with the trajectory $X(s)$,
\begin{equation}\label{eq:hamiltonian syst}
\left\{\begin{array}{rl}
\dot X(s)  =&  D_P H(X(s),P(s))\\
\dot P(s) + P(s) =& D_X L(X(s),\dot X(s)) = - D_X H(X(s),P(s)).
\end{array}\right.
\end{equation}
Then, the evolution of the Hamiltonian function $H(X(s),P(s))$ along the characteristic lines, when there is enough regularity, is:
\[
	\dfrac{d}{ds}\left (  H(X(s),P(s))\right )
		= D_X H \cdot \dot X + D_P H \cdot \dot P
		= -(\dot P + P)\cdot D_P H + D_P H \cdot \dot P
		= -P \cdot D_P H.
\]
From our choice of $P$ along with the representation formula for $L$, we see that $H(X,P) + L(X, D_P H(X,P)) = P \cdot D_P H(X,P)$, so that the above becomes $\dot H + H + L = 0$, or, equivalently,
\begin{equation}\label{eq:evolution H}
\dfrac{d}{ds}\left ( H e^s \right ) + L e^s = 0\, .
\end{equation}
We deduce from \eqref{eq:evolution H} and \eqref{eq:abstract min} that $U(X(0)) = - H(X(0),P(0))$.

Finally, we point out a nice relationship between $D_X U$ and $P$:  
\begin{equation}
\label{eq:P nabla U}
P(0) = D_X U (X(0)). 
\end{equation}
Indeed, if we perturb the optimal trajectory $X$ by a constant velocity $\epsilon V$ on the last portion of the time interval $(-\epsilon,0)$, we find by the minimization property \eqref{eq:abstract min}:
\begin{align*}
U(x + \epsilon^2 v) - U(x) &\leq \int_{-\epsilon}^0\left ( L(X(s) + (s+\epsilon)\epsilon V, \dot X(s) + \epsilon V) - L(X(s) , \dot X(s)  ) \right ) e^s ds\\
& \leq \epsilon  \int_{-\epsilon}^0\left ( D_V L(X(s),\dot X(s))\cdot V + \mathcal O(\epsilon)\right ) e^s\, ds.
\end{align*}
Dividing both sides by $\epsilon^2$, and letting $\epsilon\to 0$, we find that $D_XU(x) \cdot V \leq D_V L(X(0),\dot X(0))\cdot V$, for any $V$. Hence, we have $D_XU(x) = D_V L(X(0),\dot X(0))$, which is equivalent to \eqref{eq:P nabla U} by definition.
\bigskip

In our setting, the Hamiltonian associated with \eqref{eq:minimization}, is
\begin{equation}\label{eq:hamiltonian} 
\bH_{\alpha,\mu}(y,\eta,p,q) =-\frac32 y p - \eta q +  \eta |p|^2 +  |q|^2 + 1 - \mu \1_{\left\{y< \alpha\right\}}\, .
\end{equation}  
This follows from~\eqref{eq:Lagrangian ss}, where we solve for the Lagrangian.  Thus, the Hamiltonian system \eqref{eq:hamiltonian syst}  is, for the portion of the trajectories on either of the half-spaces $\left\{y<\alpha\right\}$ and $\left\{y>\alpha\right\}$,
\begin{equation}\label{eq:hamiltonian syst2}
\left\{\begin{array}{ll}
\displaystyle	\dot \by = - \frac{3}{2} \by + 2 \beeta \bp,
		\qquad &
		\displaystyle \dot \bp = \frac{1}{2} \bp,\medskip\\
\displaystyle	\dot \beeta = - \beeta + 2\bq,
	 	&
	 	\displaystyle \dot \bq = -|\bp|^2.
\end{array}\right.
\end{equation}
Here we use the fact that $\1_{\left\{y<\alpha\right\}}$ is constant on each half space. The connection between the two half-spaces must be handled with care, see below for details. The general solution of \eqref{eq:hamiltonian syst2} on any interval of {\em free motion}, {\em i.e.} avoiding the line $\left\{y=\alpha\right\}$, for trajectories ending at $(x,\theta)$ at $s=0$, is, for some constants $A$ and $B$,
\begin{equation}\label{eq:traj_optim}
\begin{cases}
\bp(s) = A e^{\frac12 s}, \medskip\\
\bq(s) = B + A^2 ( 1 - e^{s} ), \medskip\\
\beeta(s) = \theta e^{-s} +  2 B  (1 - e^{-s}) + A ^2 \left ( 2 -  e^{s} - e^{-s}\right ), \medskip\\
\by(s) = x  e^{-\frac32s} + 2  \theta A \left  ( e^{-\frac12 s } - e^{-\frac32s}\right )  +   2  B  A  \left ( e^{\frac12 s } + e^{-\frac32s} - 2 e^{-\frac12s} \right ) \\
\displaystyle  \qquad \quad  +  \frac23 A^3 \left( e^{-\frac32s} - 3e^{-\frac12s} + 3e^{\frac12s} - e^{\frac32s}  \right ).
\end{cases}
\end{equation}
Due to \eqref{eq:Lagrangian ss} and \eqref{eq:hamiltonian syst2}, the running cost on each half-space $\left\{y<\alpha\right\}$ and $\left\{y>\alpha\right\}$ is then given by:
\begin{equation*}
\bL_{\alpha,\mu}(\by(s),\beeta(s),\dot \by(s),\dot \beeta(s)) = \beeta(s) |\bp(s)|^2 + |\bq(s)|^2  - 1 + \mu \1_{(-\infty,\alpha)}(\by(s))\, .
\end{equation*}
An immediate computation yields that this quantity is constant  on each half-space $\left\{y<\alpha\right\}$ and $\left\{y>\alpha\right\}$. In particular, on some interval $(s_0,0)$ such that $\by(s)$ stays on the same side of the line, the running cost is
\begin{equation}\label{eq:running cost}
\bL_{\alpha,\mu}(\by(s),\beeta(s),\dot \by(s),\dot \beeta(s)) = \theta_0 A^2 + B^2  - 1 + \mu \1_{(-\infty,\alpha)}(\by(s))\, .
\end{equation}

We now investigate the portions of $(\by(s),\beeta(s))$ when $\by(s) = \alpha$ for an open interval of time $s\in(s_1,s_0)$.  It is convenient to extract the dynamics from the Lagrangian function \eqref{eq:Lagrangian ss} when the trajectory has been confined to the line. When confined to this line, the Lagrangian is
\begin{equation}\label{eq:running cost line}
\bL_{\left\{y = \alpha\right\}}(\eta,v_\eta) = \frac{\overline \alpha^2}{\eta} + \frac{1}{4}\left(v_\eta + \eta\right)^2  - 1\, ,  
\end{equation}
which is obtained from \eqref{eq:Lagrangian ss} by setting $v_y = 0$, and $y = \alpha$, and $\mu  \1_{(-\infty,\alpha)}(y)  = 0$.   Recall that, as given by \eqref{eq:alpha_bar}, $\overline\alpha = 3\alpha/4$. 
The corresponding Hamiltonian function is obtained through the Legendre transform with respect to the partial velocity variable $v_\eta$:
\begin{equation*}
\bH_{\left\{y = \alpha\right\}}(\eta,q) = - \frac{\overline \alpha^2}{\eta} - \eta q + |q|^2  + 1\, .  
\end{equation*}
The corresponding Hamiltonian dynamics are computed exactly as above:
\begin{equation}\label{eq:ODE line}
\dot\beeta = -\beeta + 2 \bq\,,
	\quad \text{and} \quad
\dot \bq = -  \dfrac{\overline\alpha^2}{\beeta^2}\,. \end{equation}
Moreover, similarly to above, we also obtain
\begin{equation}
\label{evolution-Halpha}
\dfrac{d}{ds}\left ( \bH_{\left\{y = \alpha\right\}} e^s \right ) + \bL_{\left\{y = \alpha\right\}} e^s = 0\, .
\end{equation}

\subsection{Better stay on the right side -- Lemma \ref{lem:barrier}}
\label{sec:barrier}

We now establish that any trajectory that ends to the right of the line $\left\{y=\alpha\right\}$ must always be to the right of this line.  
Our approach, in each lemma, is a careful analysis of the minimizing trajectories, which we can write down semi-explicitly thanks to the computations performed in \Cref{sec:hamiltonian}.  In each case, we show that, were the undesired behavior to occur, we may construct a related trajectory with a lower cost, contradicting the fact that the offending trajectory was a minimizer.

We first prove the monotonicity of optimal trajectories in $\eta$.  This is an important step in establishing \Cref{lem:No downward U-turn}.

\begin{proof}[Proof of Lemma \ref{lem:monotonicity}]
Let $(\by,\beeta)\in\bA(x,\theta)$ be the optimal trajectory.    
We begin by obtaining a differential inequality for the second derivative $\ddot \beeta$ in the distributional sense.  We note that we have not established the continuity of $\dot \beeta$ or any regularity of $\ddot \beeta$, so we are forced to work with this distributional inequality.

Fix any $\epsilon>0$ and any $0 \leq \phi \in C_c^\infty(\R_-^*)$. Notice that $(\by,\beeta + \epsilon e^{-s}\phi) \in \bA(x,\theta)$.  Thus, we have,
\[
	\int_{-\infty}^0 \bL_{\alpha,\mu}(\by,\beeta,\dot \by,\dot \beeta)e^s ds
		= U_{\alpha,\mu}(x,\theta)
		\leq \int_{-\infty}^0 \bL_{\alpha,\mu}(\by,\beeta + \e e^{-s}\phi, \dot \by, \dot \beeta +     \e (e^{-s} \dot\phi -e^{-s} \phi)) e^s ds.
\]
Writing out the expressions and re-arranging the terms, we see,
\[
	0
		\leq \int_{-\infty}^0\left(\frac{ \left (\dot \by(s)+\frac 3 2 \by(s)\right )^2}{4(\beeta(s) + \e e^{-s}\phi(s))} - \frac{(\dot \by(s) + \frac{3}{2}\by(s))^2}{4\beeta(s)} + \frac{\e}{2}e^{-s}\dot\phi(s) (\dot\beeta(s) + \beeta(s)) \right) e^s ds + O(\e^2).
\]
Expanding the first term  and  dividing by $\e$ yields,
\[
	O(\e)
		\leq \int_{-\infty}^0\left(\frac{- \phi(s)  \left (\dot \by(s)+\frac 3 2 \by(s)\right )^2}{2\beeta(s)(\beeta(s) + \e  e^{-s}\phi(s))} + \dot\phi(s)(\dot\beeta(s) + \beeta(s)) \right) ds.
\]
Applying the monotone convergence theorem, we get, 
\[
	\int_{-\infty}^0\eta\left (\ddot \phi - \dot \phi\right ) \leq \int_{-\infty}^0\frac{- \phi(s)  \left (\dot \by(s)+ \frac 3 2 \by(s)\right )^2}{2\beeta(s)^2} ds  \leq 0 .
\]
Since this is true for all $\phi$, it follows that $\ddot \beeta + \dot \beeta \leq  0$ in the sense of distributions, from which it follows that $\frac d{ds} (e^s \dot \beeta) \leq  0$ holds in the sense of distributions.

We now conclude the proof by choosing an appropriate test function.  If $\beeta$ is not non-increasing, then there exists a $0 \leq \psi \in C_c^\infty(\R_-^*)$ such that $\int \dot \beeta \psi  e^s ds = \gamma > 0$ and $\int \psi ds =1$.  Fix any $s' < \inf\supp(\psi)$ and  $\e  > 0$ such that $\e <  - \textcolor{bl}{s'}$.  Let $\phi_{\e}$ be a standard mollifier with $ \supp \phi_{ \e} \subset (-\e,+\e)$.  Then, define the   smooth test function
\[
	\chi_{\e}(s) = \left(\int_{-\infty}^{s} \phi_{ \e }(s'-\bar s)d\bar s\right)+\left(\int_{s}^0 \psi(\bar s)d\bar s\right) -1.
\]
Note that from our choice of $\textcolor{bl}{s'}$ and $\e$, the above test function is positive and compactly supported in $\R_-^*$.
From our choice of $\phi_{\e}$ and $\psi$ along with the differential inequality established above,
\[
	-\int_{-\infty}^0 e^{s} \dot \beeta(s) \phi_{\e}(s'-s) ds + \gamma
		= -\int_{-\infty}^0 e^{s} \dot \beeta(s) \left(\phi_{\e}(s'-s) - \psi(s)\right) ds
		= -\int_{-\infty}^0 e^{s} \dot \beeta(s) \dot \chi_{\e}(s) ds
		\leq 0.
\]
Multiplying both sides by $e^{-s'}$ and integrating over $(s_1,s_0)$ for any $s_0 < \inf \supp(\psi)$, we find
\[
	-\int_{-\infty}^0 e^s \dot\beeta(s) \int_{s_1}^{s_0}  e^{-s'} \phi_{\e}(s'-s) ds' ds
		= -\int_{s_1}^{s_0} e^{-s'}\int_{-\infty}^0 e^{s} \dot \beeta(s) \phi_{\e}(s'-s) ds ds'
		\leq \gamma( e^{-s_0} - e^{-s_1}).
\]
We may take $\e\to0$ in the interior integral on the left hand side to obtain
\[
	\beeta(s_1) - \beeta(s_0)
		= -\int_{s_1}^{s_0} \textcolor{bl}{ \dot\beeta(s) }ds
		\leq \gamma( e^{-s_0} - e^{-s_1}).
\]
Hence $\beeta(s_1) \to -\infty$ as $s_1 \to -\infty$, which contradicts the fact that $\beeta \geq 0$, by definition.  This concludes the proof.
\end{proof}

\begin{proof}[Proof of Lemma \ref{lem:No single crossing}]
We argue by contradiction. Suppose that a trajectory crossing the line $\left\{y = \alpha\right\}$ only once is optimal.  
Let $s_0$ be the time such that $\by(s) < \alpha$ for all $s < s_0$ and $\by(s) \geq \alpha$ for all $s\in(s_0,0)$, and let denote $\theta_0 = \beeta(s_0)$.   By the dynamic programming principle, the trajectory is also optimal on the interval $(-\infty,s_0)$. By a translation of time $s - s_0$, we can assume without loss of generality that $(\by(0),\beeta(0)) = (\alpha,\theta_0)$ and that $(\by,\beeta)$ is a minimizing trajectory in $\bA(\alpha,\theta_0)$.  By assumption, we note that $\by(s) < \alpha$ for all $s < 0$.  Then, it is a global solution of the system \eqref{eq:hamiltonian syst2} with $x = \alpha$. 

From~\eqref{eq:traj_optim} we have
\begin{equation*}
\begin{cases}
\beeta(s) = \theta_0 e^{-s} +  2 B  (1 - e^{-s}) + A ^2 \left ( 2 -  e^{s} - e^{-s}\right ), \quad\text{ and}\medskip\\
\by(s) = \alpha  e^{-\frac32s} + 2  \theta_0 A \left  ( e^{-\frac12 s } - e^{-\frac32s}\right )  +   2  B  A  \left ( e^{\frac12 s } + e^{-\frac32s} - 2 e^{-\frac12s} \right )\medskip \\
\displaystyle	\qquad\quad   +  \frac23 A^3 \left( e^{-\frac32s} - 3e^{-\frac12s} + 3e^{\frac12s} - e^{\frac32s}  \right )
\end{cases}
\end{equation*}
for some $A, B \in \R$.
On the one hand, multiplying the first and second equality in the previous line by, respectively,  $e^{s}$ and $e^{\frac{3s}{2}}$, and then taking the limit $s\rightarrow -\infty$ and using 
the conditions  in~\eqref{eq:self-similar_admissible} implies the following equations:
\begin{align}\label{eq:condition at oo lem 3.2}
\begin{cases}
\theta_0  = 2B + A^2\,, \quad \text{and}\medskip\\
\alpha  = 2\theta_0 A - 2BA - \dfrac23A^3,
\end{cases}
&\quad \text{ which is equivalent to }\quad
\begin{cases}
\theta_0  = 2B + A^2\,, \quad \text{and}\medskip\\
\alpha = \theta_0 A+ \dfrac13A^3.
\end{cases}
\end{align}
Since $A\mapsto \theta_0 A+ \frac13A^3$ is increasing, $A$ is uniquely determined.  
On the other hand, computing $\dot{\by}(0)$ and using the condition $\dot{\by} (0)\geq 0$ implies
\begin{equation}\label{eq:xdot>0}
\theta_0 A\geq \overline\alpha,
\end{equation}
where we recall that $\overline \alpha = 3\alpha/4$. 
Finally, from \eqref{eq:running cost}, the global cost of the trajectory equals the (constant) running cost:
\begin{align*}
	\bL_{\alpha,\mu}(\by,\beeta,\dot \by,\dot\beeta)
	 &= \theta_0 A^2 +  B^2 + \mu - 1
	 = \theta_0 A^2 + \frac{(\theta_0-A^2)^2}{4} +  \mu - 1 \\
	 &= \dfrac{\theta_0 A^2}{2} + \dfrac{\theta_0^2}{4} + \dfrac{A^4}{4}+ \mu- 1.
\end{align*}

This global cost can be compared with the cost of the steady trajectory located at the same endpoint.  Indeed, let $(\tilde \by(s),\tilde \beeta(s))= (\alpha,\theta_0)$ for all $s\in (-\infty,0)$.  It is clear that $(\tilde \by, \tilde \beeta) \in \bA(\alpha,\theta_0)$. From (\ref{eq:running cost line}), the associated cost is 
\begin{equation*}
\bL_{\alpha,\mu}\left (\tilde \by, \tilde \beeta, \dot{\tilde \by}, \dot{\tilde \beeta}\right ) = \dfrac{\overline \alpha^2}{\theta_0} + \dfrac{\theta_0^2}{4} - 1\, .
\end{equation*}
This trajectory is by no means globally optimal; however, it has a lower cost than the trajectory $(\by(s),\beeta(s))$ under the assumptions of Lemma \ref{lem:barrier}. Indeed, we wish to show that $\bL_{\alpha,\mu}(\by,\beeta,\dot \by,\dot\beeta) > \bL_{\alpha,\mu}(\tilde \by, \tilde \beeta, \dot{\tilde \by}, \dot{\tilde \beeta})$, which contradicts the fact that $(\by,\beeta)$ is a minimizing trajectory.  This is equivalent to showing that
\begin{equation}\label{eq:Jdelta J0}
	\frac{\theta_0 A^2}{2} + \frac{A^4}{4} - \frac{\overline\alpha^2}{\theta_0} \geq - \mu.
\end{equation}
According to (\ref{eq:condition at oo lem 3.2}) and~\eqref{eq:xdot>0}, we have the following constraints on the values of $A$:
\begin{equation}\label{eq:constraints}
\dfrac{\theta_0 A}{\overline\alpha}\geq  1
	\quad\text{ and }\quad
\dfrac{3\theta_0 A}{4\overline \alpha}+ \dfrac{A^3}{4\overline \alpha} = 1.
\end{equation}
This suggests that we use the new variables $a$ and $b$ such that
\begin{equation*}
a = \frac{\theta_0 A}{\overline \alpha}, \quad b = \frac{A}{\overline\alpha^{1/3}}\, , \quad \dfrac{3a}{4} + \frac{b^3}{4} = 1\, .
\end{equation*}
According to the definition of $a$, and by \eqref{eq:xdot>0} and \eqref{eq:constraints}, we have
\[
a\in \left[1,\frac{4}{3}\right]\,  \quad \text{and}\quad b\in [0,1]\, .
\]
With the definitions of $a$ and $b$, the inequality \eqref{eq:Jdelta J0} is equivalent to
\begin{equation}\label{eq:cost_condition}
	\frac{\overline \alpha^{4/3}}{\mu} \left (\dfrac{a b}{2} + \dfrac{b^4}{4} - \dfrac{b }{a}\right )\geq - 1.
\end{equation}
We now prove~\eqref{eq:cost_condition}, which finishes the proof. 
Since $b^3 = 4 - 3a$, then
\begin{equation}\label{eq:c1}
	\left (\dfrac{a b}{2} + \dfrac{b^4}{4} - \dfrac{b }{a}\right )
		=  b \left(1 - \frac{a}{4} - \frac{1}{a}\right)
		= (4 - 3a)^{1/3}\left(1 - \frac{a}{4} - \frac{1}{a}\right).
\end{equation}
The right hand side of this expression is increasing for $a\in (1,\frac43)$: its derivative with respect to $a$ is
\begin{equation*}
\dfrac{\left (4 - 3 a\right )^{-2/3}}{a^2}\left ( 4 - 2 a - 2 a^2 + a^3 \right )
	= \dfrac{\left (4 - 3 a\right )^{-2/3}}{a^2} (2-a)(2-a^2)
	> 0.
\end{equation*}
Thus, we may bound the right hand side of~\eqref{eq:c1} by its value at $a=1$, which implies that
\[
	\left (\dfrac{a b}{2} + \dfrac{b^4}{4} - \dfrac{b }{a}\right)
		\geq \left(4 - 3 a\right)^{1/3} \left( 1 - \frac{a}{4} - \frac{1}{a}\right){\bigg|}_{a=1}
		= -\frac{1}{4}.
\]
Hence, we obtain
\[
	\frac{\overline \alpha^{4/3}}{\mu}\left (\dfrac{a b}{2} + \dfrac{b^4}{4} - \dfrac{b }{a}\right)
		\geq -\frac{\overline \alpha^{4/3}}{4\mu}
		> -1,
\]
where we used the condition $\overline\alpha^{4/3} < 4 \mu$ in the last inequality.  Hence, we have established~\eqref{eq:Jdelta J0}, contradicting the fact that $(\by,\beeta)$ is a minimizing trajectory.  This concludes the proof.
\end{proof}

Note that we have used the weaker condition $\overline\alpha^{4/3} < 4 \mu$ instead of $\overline\alpha^{4/3} \leq 2 \mu$. In fact the next proof   requires a more stringent condition on the parameters.

\begin{proof}[Proof of Lemma \ref{lem:No downward U-turn}]
To proceed with the non-optimality of the  C-turn, we make the following reduction.  As above, the dynamic programming principle implies that we may suppose, without loss of generality, that $s_0 = 0$ and $s_1 < 0$ (see \Cref{fig:downward uturn}).   

Since the trajectory does not cross the line $\left\{y=\alpha\right\}$ during the time interval $(s_1,0)$, the optimal trajectory $(\by,\beeta)$ is given by~\eqref{eq:traj_optim}, with $x=\alpha$, for some constants $A,B\in\R$. We point out that, by \Cref{lem:monotonicity},
\begin{equation}
\label{eq:theta0 etas1}
\theta_0=\beeta(0)\leq \beeta(s_1).
\end{equation}
Further, since $\by(0) = \alpha$ and $\by(s) < \alpha$ for $s \in (s_1,0)$, it follows that $\dot \by(0) \geq 0$.  Hence~\eqref{eq:xdot>0} is valid.     There seems to be no natural way to compare the trajectory with a steady trajectory as in the proof of Lemma \ref{lem:No single crossing}.  Alternatively, we compare the trajectory $(\by(s),\beeta(s))$ to the trajectory $(\tilde \by(s), \beeta(s))$, where we define
\[
\tilde{\by}(s)=\begin{cases}
\by(s)& \text{ for }s<s_1, \\
\alpha& \text{ for } s_1\leq s<0.
\end{cases}
\]  
In short, $(\tilde \by,\beeta)$ is obtained by projecting the portion between $s_1$ and 0, the C-turn, onto the line.  It is clear that $(\tilde \by, \beeta) \in \bA(\alpha,\theta_0)$.

To show that $(\tilde \by, \beeta)$ has a lower cost than $(\by,\beeta)$, it is enough to compare the partial costs on the interval $(s_1,0)$. The  cost for $(\by,\beeta)$ is, via~\eqref{eq:running cost},
\begin{equation*}
	J_{\rm orig}
		:= \int^0_{s_1} \bL_{\alpha,\mu}(\by,\beeta,\dot \by,\dot\beeta) ds
		= \int_{s_1}^0\left ( \theta_0 A^2 +  B^2 + \mu - 1\right ) e^s ds.
\end{equation*}
The cost of $(\tilde \by, \beeta)$ on $(s_1,0)$ is, 
\begin{align*}
J_{\rm new}
	&:= \int_{s_1}^0 \bL_{\alpha,\mu}(\tilde \by, \beeta, \dot{\tilde \by}, \dot{\beeta}) ds
	= \int_{s_1}^0\left ( \dfrac{\overline \alpha^2}{\beeta(s)} +  \frac{1}{4}\left | \dot\beeta(s)+\beeta(s) \right |^2 - 1\right ) e^s\, ds \\
& = \int_{s_1}^0\left ( \dfrac{\overline\alpha^2}{\beeta(s)} + \left |  B + A^2\left (1 - e^s\right ) \right |^2 - 1\right ) e^s\, ds ,
\end{align*}
where we have obtained the second equality by using the expression for $\beeta$ in (\ref{eq:traj_optim}) and computing $\dot\beeta$. We now consider the difference $J_{\rm orig} - J_{\rm new}$.  The above formulas imply
\begin{equation}\label{eq:Jmu - J0}
\begin{split}
J_{\rm orig} - J_{\rm new} & = \int_{s_1}^0 \left[ A^2 e^s \bigg ( \theta_0 e^{-s} - 2 B \left ( e^{-s} - 1 \right ) - A^2 \left ( e^{-s} - 2  + e^s\right )\bigg ) - \dfrac{\overline\alpha^2}{\beeta(s)} + \mu\right ] e^s\, ds\\
 & = \int_{s_1}^0 \left[ A^2 e^s  \beeta(s) - \dfrac{\overline\alpha^2}{\beeta(s)} \right ] e^s\, ds + \mu\left (1 - e^{s_1}\right ),
\end{split}
\end{equation}
where to obtain the last equality we have used the expression for $\beeta$ in~\eqref{eq:traj_optim}. Since the integrand is increasing with respect to $\beeta$, it is fruitful to bound $\beeta(s)$ from below. In view of \eqref{eq:traj_optim}, this amounts to bounding $B$ from above. In parallel with the proof of Lemma \ref{lem:No single crossing}, we shall use the information at $s = s_1$ in order to gain an estimate for $B$.  Evaluating at $s_1$ the expression for $\beeta$ in~\eqref{eq:traj_optim}, and then using \eqref{eq:theta0 etas1}, yields
\begin{align*}
2B\left (e^{-s_1} - 1\right )
	 = \theta_0 e^{-s_1} - \eta(s_1) +  A ^2 \left ( 2 -  e^{s_1} - e^{-s_1}\right )
	 \leq \theta_0 \left (e^{-s_1} - 1\right ) +  A ^2 \left ( 2 -  e^{s_1} - e^{-s_1}\right ).
\end{align*}
Notice  $(2- e^{s_1}-e^{-s_1}) = (e^{-s_1}-1)(e^{s_1}-1)$. Using this, along with the bound above, we see, for all $s \in (s_1,0)$,
\begin{equation}\label{eq:bound eta}
\begin{split}
\beeta(s)
	&= \theta_0 e^{-s} + 2 B(1-e^{-s}) + A^2 (1- e^s)(1-e^{-s})\\
	& \geq  \theta_0 e^{-s} +  \left ( \theta_0 + A ^2\left ( e^{s_1} - 1\right )  \right )  (1 - e^{-s}) + A ^2 (1- e^s)(1-e^{-s})\\
	& = \theta_0 + A^2\left (  \left ( e^{s_1}-1\right ) \left (1- e^{-s} \right ) +  (1- e^s)(1-e^{-s})\right )
	= \theta_0 + A^2   \left (e^{s_1} - e^s \right ) \left (1 -  e^{-s} \right ) .
\end{split}
\end{equation}
Let
\[
	I := \dfrac{1}{1 - e^{s_1}} \int_{s_1}^0 \left[ A^2 \left ( \theta_0 e^s + A^2   \left ( e^{s} - e^{s_1} \right ) \left ( 1 -  e^{s} \right ) \right )  - \dfrac{\overline\alpha^2}{\theta_0  } \right ] e^s\, ds\, .
\]
In view of~\eqref{eq:bound eta}, along with~\eqref{eq:Jmu - J0}, we find
\[
	J_{\rm orig} - J_{\rm new} \geq I + \mu,
\]
where we have used the bound~\eqref{eq:bound eta} for the first occurrence of $\beeta(s)$ in \eqref{eq:Jmu - J0}, but the less precise estimate $\beeta(s)\geq \theta_0$ for the second occurrence.  Thus, in order to control the sign of $J_{\rm orig} - J_{\rm new}$, it is sufficient to show $I > -\mu$.

We now establish the lower bound on $I$. An explicit computation yields
\begin{equation*}
I =  \dfrac{\theta_0 A^2}2 (1+e^{s_1})  + \dfrac{A^4}{6} \left (1 - e^{s_1}\right )^2 - \dfrac{\overline\alpha^2}{\theta_0}\,.
\end{equation*}
Recall that, due to~\eqref{eq:xdot>0}, $\theta_0 A \geq \overline\alpha$. Hence,
\[
	I \geq \frac{\overline\alpha^2}{2\theta_0} (1+e^{s_1}) + \frac{\overline\alpha^4}{6 \theta_0^4}(1 - e^{s_1})^2 - \frac{\overline\alpha^2}{\theta_0}
		= -\frac{\overline\alpha^2}{2\theta_0}(1-e^{s_1}) + \frac{\overline\alpha^4}{6 \theta_0^4}(1 - e^{s_1})^2.
\]
The quantity on the right hand side is minimized (with respect to $\theta_0$) when $\theta_0^3 = 4\overline\alpha^2(1-e^{s_1})/3$.  Thus, we have
\begin{equation*}
	I
		\geq -\frac{3^{4/3} (1-e^{s_1})^{2/3}\overline\alpha^{4/3}}{2^{11/3}}.
\end{equation*}
Recall that $\overline\alpha^{4/3}\leq  2\mu$.  Also, notice that $1-e^{s_1}\leq 1$.  Hence,
\[
	I 
	\geq -\mu \frac{2 \cdot 3^{4/3} (1-e^{s_1})^{2/3}}{2^{11/3}}
	\geq -\mu \left(\frac{3}{4}\right)^{4/3}
	> - \mu.
\]
In view of the definition of $I$, this implies that $J_{\rm orig} - J_{\rm new} > 0$.  Thus $(\by,\beeta)$ cannot be a minimizer, since the cost of $(\tilde \by, \beeta)$ is strictly smaller.  This concludes the proof.
\end{proof}

\section{The explicit characterization of $\alpha^*$ -- \Cref{thm:alpha^*}}\label{sec:alpha^*}

En route to proving  \Cref{thm:alpha^*}, the exact shape of the function $U_\alpha$ must be deciphered, at least when restricted to endpoints $(\alpha,\theta)$. This involves a careful handling of the connection between the portions of the trajectory which moves freely in $(\alpha,\infty)\times \R_+^*$, and the portions that stick to the line $\left\{y = \alpha\right\}$.

In order to begin the discussion, we first establish the uniqueness of optimal trajectories.  The proof also establishes the convexity of $U_\alpha(x,\theta)$ on the domain $[\alpha,\infty)\times\R_+^*$.  This is the content of the following lemma.

\begin{lemma}
\label{cor:convexity}
If $\alpha \in [0,4/3]$, then $U_\alpha(x,\theta)$ is strictly convex on the domain $[\alpha,\infty)\times \R_+^*$. Moreover, for all $(x,\theta)$ in $[\alpha,\infty)\times \R_+^*$ there is a unique optimal trajectory. 
\end{lemma}

Knowing that optimal trajectories are unique allows us to completely characterize them. This characterization relies on good properties of an auxiliary function $\bQ:\R_+^*\to \R$, (see \Cref{sec:trajectories} for a precise definition).  Here, we rely only on the useful properties that
\begin{equation}\label{eq:thetadiam}
\frac{\bQ(\theta)}{\theta} \text{ is strictly increasing}
	\quad \text{ and }\quad
	\ \exists\Thetastar \text{ such that } \bQ(\Thetastar) = \dfrac{\Thetastar}4
\end{equation}
that separates those trajectories that make an excursion to the right versus those that ``stick'' to the line $\left\{y=\alpha\right\}$ (see \Cref{prop:trajectories}.(iii) for a precise statement of the latter property).  From $\bQ$ and $\Thetastar$, we also define the entire family:
\[
	\bQ_\alpha(\theta) = \overline\alpha^{2/3} \bQ\left (\dfrac{\theta}{ \overline\alpha^{2/3}}\right )
		\quad \text{ and } \quad
	\Thetastar_\alpha =  \overline\alpha^{2/3} \Thetastar,
\]
where recall from (\ref{eq:alpha_bar}) that $\bar{\alpha}=3\alpha/4$. In particular, we have $\bQ\equiv \bQ_{4/3}$.

The next proposition gathers useful properties of the optimal trajectories. Useful notation is illustrated in Figure \ref{fig:notation contact} for the reader's sake.

\begin{figure}
\begin{center}
\includegraphics[width = 0.5\linewidth]{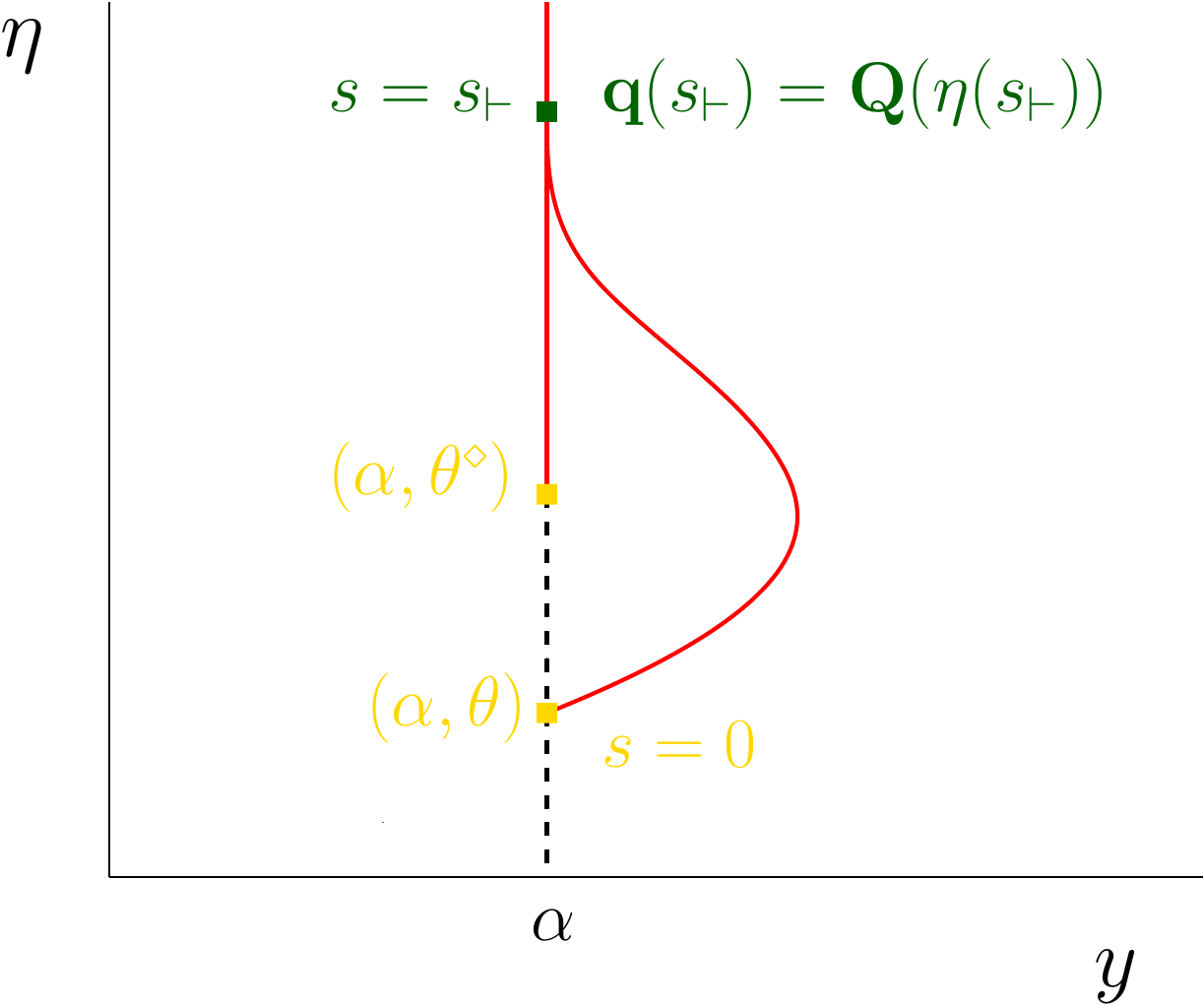} 
\caption{Illustration of the qualitative behavior of optimal trajectories  outlined by \Cref{prop:trajectories}.   Let $(\alpha, \theta)$ be some endpoint on the line $\left\{y = \alpha\right\}$. There exists a contact time $\shh\leq 0$ such that the trajectory sticks to the line if $s\leq \shh$. Moreover, $\shh = 0$ if and only if $\theta \geq \Thetastar$, where $\Thetastar$ is some threshold value on the $\eta$ coordinate. Finally, at the time $\shh$, and beyond $s<\shh$, there is a nonlinear relationship between $\bq(s)$ and $\beeta(s)$, solutions of \eqref{eq:ODE line}, involving the function $\bQ$ which only depends on the value of $\alpha$.}
\label{fig:notation contact}
\end{center}
\end{figure}

\begin{proposition}\label{prop:trajectories}
Let $\theta> 0$ and let $(\by,\beeta)$ be the optimal trajectory in $\bA(\alpha,\theta)$.  Then $(\by,\beeta)$ satisfies the following conditions:
\begin{enumerate}[(i)]
	\item \label{item:sh}There exists   $\shh = \shh(\theta)\leq 0$ such that $\by(s) = \alpha$ if and only if $s \leq \shh$ and, further, $\theta\mapsto \beeta\left (\shh(\theta)\right )$ is  a continuous function;
	\item \label{item:Q}  For all $\theta>0$, we have $\beeta(\shh(\theta)) \geq \Thetastar$;
	\item  \label{item:sh=0}$\shh = 0$ if and only if $\theta \geq \Thetastar$;
	\item \label{item:continuous}  For $s \in (\shh,0)$, $(\by,\beeta,\bp,\bq)$ solves~\eqref{eq:hamiltonian syst2} and    is such that~\eqref{eq:evolution H} holds with the Hamiltonian given by~\eqref{eq:hamiltonian}. 
	For $s \in (-\infty,\shh)$, $(\beeta, \bq)$ solves~\eqref{eq:ODE line} and is such that~\eqref{evolution-Halpha} holds. 
	In the interval $(-\infty,\shh)$, we may continue $\bp$ as $\bp=\overline{\alpha}/ \beeta$ in order to be consistent with $(\by,\dot{\by})=(\alpha,0)$ in~\eqref{eq:hamiltonian syst2}. With this convention, $(\by,\beeta,\bp,\bq)$  is continuous in $(-\infty,0]$ and we have $U_\alpha(\alpha,\theta) = - \bH_\alpha(\alpha,\theta,\bp(0),\bq(0))$;
	\item \label{item:after sh} For all $s\leq  \shh$, we have $\bq(s) = \bQ_\alpha(\beeta(s))$;
	\item \label{item:A,B}If $s \in (\shh,0]$, then $(\by,\beeta)$ 
	solves~\eqref{eq:traj_optim} with 
	\begin{equation*}
	A =  \dfrac{\overline\alpha}{\beeta(\shh)} e^{-\frac 1 2\shh}\quad \text{and}\quad   B = \bQ_\alpha(\beeta(\shh)) - A^2\left (1-e^{\shh}\right )\, .
	\end{equation*}
\end{enumerate}
\end{proposition}

We postpone the proof of this important list of results to \Cref{sec:trajectories}. However, we can make a few comments about some quantitative statements there. Firstly, we find that the optimal trajectory sticks precisely to the line $\left\{y = \alpha\right\}$ for some interval $(-\infty, \shh]$, with $\shh\leq 0$ (in fact, $\shh = 0$ if and only if $\theta \geq \Thetastar$). We refer to $\shh$ as the {\em contact time}. Secondly, and quite importantly, $\bq$ and $\beeta$ are linked by the relationship~\eqref{eq:ODE line} when $s\leq \shh$.  It turns out that the constraint at $s = -\infty$, $\beeta(s) = o(e^{-s})$, selects one branch of the family of solutions, and we can identify and describe this explicitly using the identity $\bq = \bQ_\alpha(\beeta)$.

\begin{figure}
\begin{center}
\includegraphics[width = 0.5\linewidth]{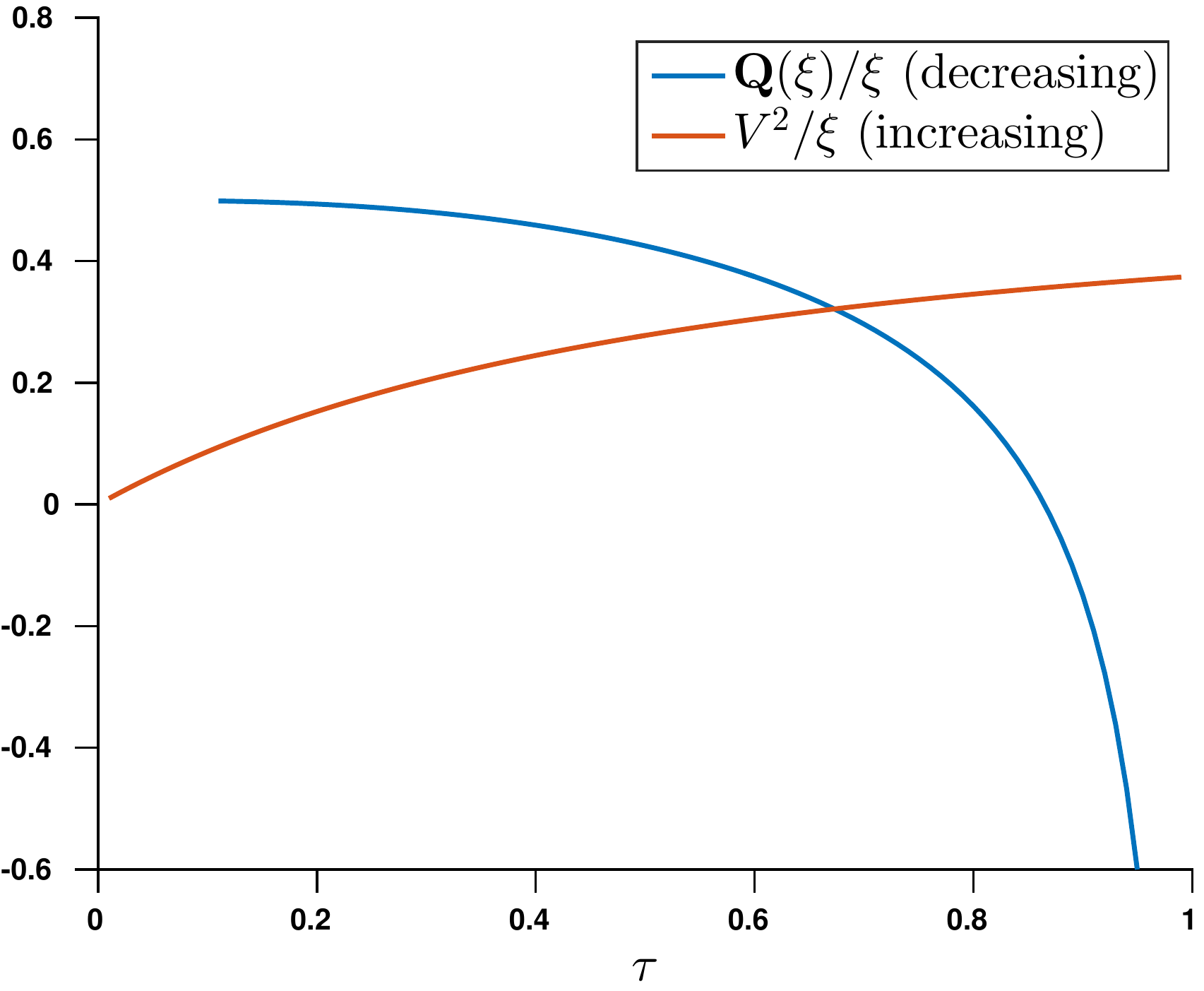} 
\caption{The curves $\tau\mapsto \dfrac{\bQ(\xi(\tau))}{\xi(\tau)}$ (in blue), and $\tau\mapsto\dfrac{V(\tau)^2}{\xi(\tau)}$ (in red), have a unique intersection time $\tau_0$ due to opposite monotonicity.}
\label{fig:unique_time}
\end{center}
\end{figure}

To be able to identify the contact time with an analytical equation, we also need the following technical lemma on real functions, which is going to be used with the change of unknown $\tau =  \exp(s/2) = t^{1/2}\in (0,1)$. 
\begin{lemma}
\label{lem:VxiQR}
Let $\mR$ be defined by \eqref{eq:riccati}, and  $V$ be defined by \eqref{eq:V}. Define the function $\xi$ by
\begin{equation}
\label{eq:xi}
	\xi(\tau) = \frac{(1-\tau^2)^\frac12}{\tau V(\tau)}.
\end{equation}
 Then
\begin{equation}
\label{eq:VandQ}
V(\tau)^2 = \bQ(\xi(\tau))
\end{equation}
if and only if $V(\tau)^4-\xi^{-1}(\tau) > \Xi_0$ and
\begin{equation}
\label{eq:VandR}
V(\tau)^2=\mR\left (V(\tau)^4-\xi^{-1}(\tau)\right ).
\end{equation}
Moreover, there is at most one $\tau_0\in (0,1)$ such that (\ref{eq:VandQ}) holds.
\end{lemma}

The uniqueness of $\tau_0$ 
 is proved by a monotonicity argument: dividing each side of the equation by $\xi(\tau)$, we find that the left-hand side $V(\tau)^2/ \xi(\tau)$ and the right-hand side $\bQ(\xi(\tau))/\xi(\tau)$ have opposite monotonicity, as illustrated in Figure \ref{fig:unique_time}. The difficulty arises in showing the monotonicity of $\theta\mapsto \bQ(\theta)/\theta$.

We prove \Cref{cor:convexity}, \Cref{prop:trajectories} and \Cref{lem:VxiQR} in \Cref{sec:trajectories}.  We now show how to conclude \Cref{thm:alpha^*} from these  three results.

\subsection{The homogeneity of $U_\alpha$ -- \Cref{thm:alpha^*}.(i)}

We begin by establishing the homogeneity of $U_\alpha$.  While this neither relies on \Cref{cor:convexity}, \Cref{prop:trajectories}, nor \Cref{lem:VxiQR}, it is used to establish \Cref{thm:alpha^*}.(ii).

\begin{proof}[Proof of \Cref{thm:alpha^*}.(i)]
Fix any $\alpha_0, \alpha_1 \in (0,4/3]$. Let $\theta>0$. Let  $(\bx_1,\btheta_1) \in \mA(\alpha_1,\theta)$ be an optimal trajectory.  Then
\[
	U_{\alpha_1}(\alpha_1,\theta)
		= \int_0^1 L_{\alpha_1}(t,\bx_1,\btheta,\dot \bx_1, \dot{\btheta}_1) dt.
\]

Let $\beta = \alpha_0 / \alpha_1$.  Define a new trajectory $( \bx_0,  \btheta_0) = (\beta \bx_1, \beta^{2/3} \btheta_1)$.  Then $( \bx_0,  \btheta_0) \in \mA( \alpha_0, \beta^{2/3}\theta)$.   By definition, it follows that
\[\begin{split}
	U_{ \alpha_0}( \alpha_0, \beta^{2/3} \theta)
		&\leq \int_0^1 L_{ \alpha_0}\left (t,  \bx_0,  \btheta_0, \dot{ \bx}_0,\dot{ \btheta}_0\right )dt\\
		&= \int_0^1 \left( \beta^{4/3}\left(\frac{|\dot{\bx_1}|^2}{4\btheta_1} + \frac{|\dot{ \btheta}_1|^2}{4}\right) - 1 + \1_{\left (-\infty,\alpha_0 t^{3/2}\right )}(\beta \bx_1(t)) \right) dt.
\end{split}\]
By \Cref{lem:good_trajectories}, we know that, for all $t\in[0,1]$, $\bx_1(t) \geq \alpha_1 t^{3/2}$ and, hence, $\beta \bx_1(t) \geq \alpha_0 t^{3/2}$.  Using this, we see that
\[\begin{split}
	U_{ \alpha_0}( \alpha_0, \beta^{2/3} \theta)
		&\leq \int_0^1 \left( \beta^{4/3}\left(\frac{|\dot{\bx_1}|^2}{4\btheta_1} + \frac{|\dot{ \btheta}_1|^2}{4}\right) - 1\right) dt\\
		&= \beta^{4/3}\int_0^1  \left(\frac{|\dot{\bx_1}|^2}{4\btheta_1} + \frac{|\dot{ \btheta}_1|^2}{4} - 1 \right) dt - 1 + \beta^{4/3}
		= \beta^{4/3} U_{\alpha_1}(\alpha_1,\theta) - 1 + \beta^{4/3}.
\end{split}\]
By symmetry, we have
\[
	U_{\alpha_1}(\alpha_1,\theta)
		\leq \beta^{-4/3} U_{\alpha_0}(\alpha_0, \beta^{2/3}\theta) - 1 + \beta^{-4/3}.
\]
Using both inequalities together, we find
\[
	U_{ \alpha_0}( \alpha_0, \beta^{2/3} \theta) = \beta^{4/3}U_{\alpha_1}(\alpha_1, \theta) - 1 + \beta^{4/3}.
\]
Taking the minimum with respect to $\theta\in \R_+^*$ on both sides, we obtain:
\[
	\min_\theta U_{ \alpha_0}( \alpha_0,  \theta) = \left ( \dfrac{\alpha_0}{\alpha_1}\right )^{4/3}\min_\theta U_{\alpha_1}(\alpha_1, \theta) - 1 + \left ( \dfrac{\alpha_0}{\alpha_1}\right )^{4/3}.
\]
This concludes the proof.
\end{proof}

\begin{remark}[Scaling of optimal trajectories]\label{rem:scaling}
The argument above, together with the uniqueness of trajectories (Lemma \eqref{cor:convexity}), clearly shows that the optimal trajectory $(\bx_1, \btheta_1)$ ending at $(\alpha_1,\theta_1)$, with parameter $\alpha_1$, is bound to the optimal trajectory $( \bx_0,  \btheta_0) $ ending at $(\alpha_0,\beta^{2/3}\theta_1)$, with parameter $\alpha_0$, as follows: 
\begin{equation}\label{eq:link optimal traj}
 \dfrac{\btheta_0(s)}{\alpha_0^{2/3}} = \dfrac{\btheta_1(s)}{\alpha_1^{2/3}}\, .
\end{equation}
\end{remark}

\subsection{The analytical value of $\alpha^*$ -- \Cref{thm:alpha^*}.(ii)}
\label{subsec:alpha-anal}

In this subsection, we show how to get an algebraic equation for  $\alpha^*$. In order to compute this value, due to  \Cref{thm:alpha^*}.(i), it is enough to fix $\alpha = 4/3$, and to compute $\min_\theta U_{4/3} (4/3,\theta)$. 
We first show that such a minimum is attained at a unique point $\theta_{\min}$. We then identify the optimal trajectory  ending at $(4/3,\theta_{\min})$. The identification of this trajectory relies on the computation of the contact time $\shh(\theta_{\min})$.
Once the optimal trajectory is characterized, we can compute the value of $U_{4/3} (4/3,\theta_{\min})$, hence $U_{\alpha} (\alpha,\theta_{\min})$ by homogeneity. Figure \ref{fig:Ualpha} represents the function $U_{\alpha} (\alpha,\cdot)$ at $\alpha = \alpha^*$, for the sake of illustration.

\begin{figure}
\begin{center}
\includegraphics[width = 0.5\linewidth]{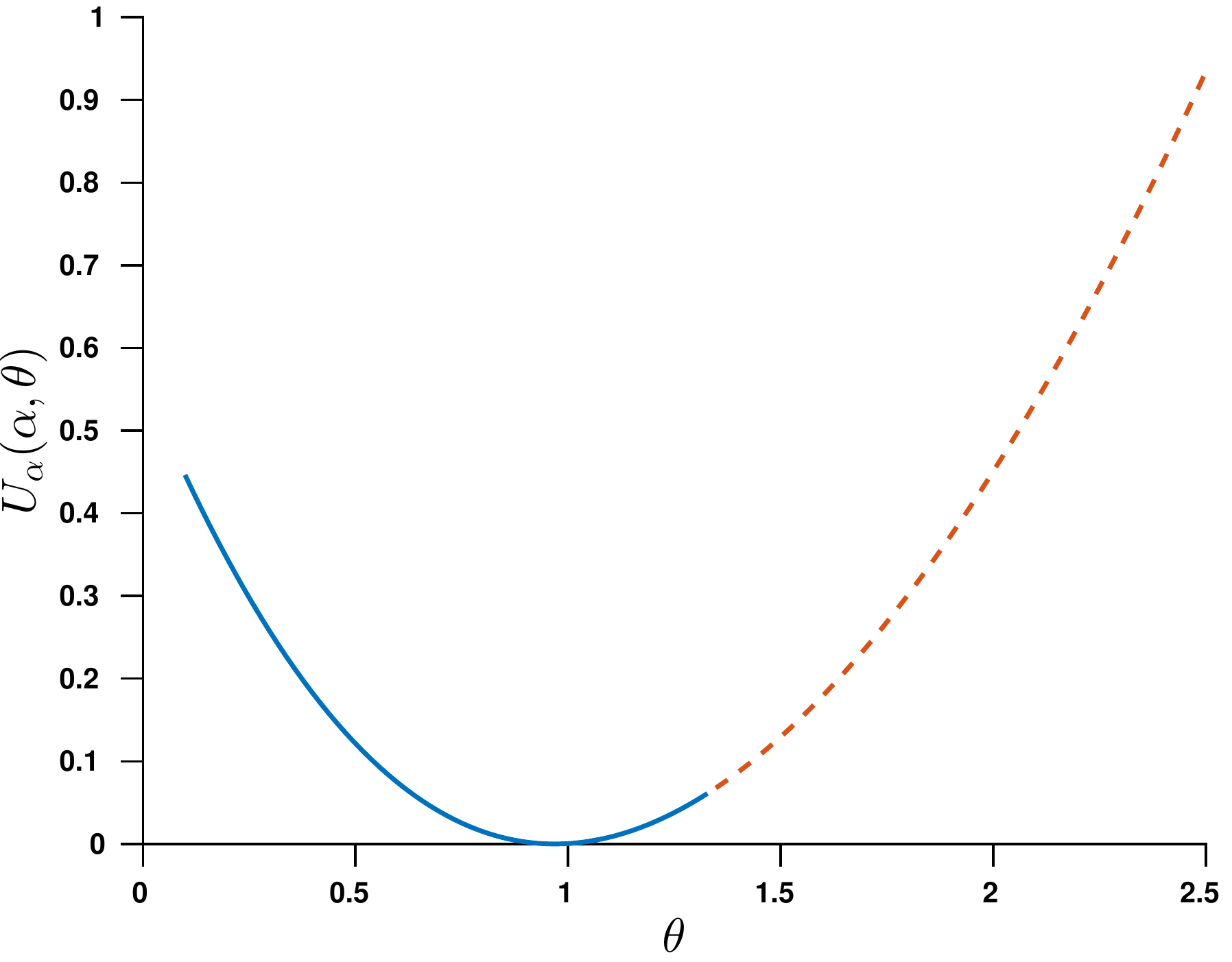} 
\caption{The function $\theta\mapsto U_{\alpha}(\alpha,\theta)$ for the critical value $\alpha = \alpha^*$. The portion of the curve where $\shh(\theta) < 0$ is in plain line, whereas the portion where $\shh(\theta) = 0$ is in dashed line, that is $\theta \geq \Thetastar$.}
\label{fig:Ualpha}
\end{center}
\end{figure}

\begin{proof}[Proof of \Cref{thm:alpha^*}.(ii)]
\noindent{\bf \# Existence  and uniqueness of a minimum in $\theta$: }
To ensure the existence of an interior minimum point of $\theta\mapsto U_{4/3}(4/3, \theta)$, we seek an interior critical point in $\theta$.  The strict convexity of $U_{4/3}$ in $\theta$ given by \Cref{cor:convexity} implies that any critical point in $\theta$ is the unique minimum of $U_{4/3}$.    To find such a critical point, we seek a $\theta_0$ such that the optimal trajectory $(\by, \beeta)$ with endpoint at $(4/3,\theta_0)$ satisfies
\[
	\shh(\theta_0)<0 \quad \text{and}\quad \bq(0)=0\,.
\]
Indeed, if we have such a trajectory, we find, 
\[
\partial_\theta U_{4/3}(4/3,\theta_0) = \partial_\theta U_{4/3} (\by(0) ,\beeta(0)) = \bq(0)=0,
\]
where the  second equality follows by \eqref{eq:P nabla U}. This is precisely the characterization of a critical point of $U_{4/3}(4/3,\cdot)$, implying that $\theta_0$ is indeed the unique minimum of $U_{4/3}(4/3, \cdot)$.

We now prove that there is indeed such a point $\theta_0$ in the interval $(0,\Thetastar)$. Note first that  by Proposition \ref{prop:trajectories}.(\ref{item:sh=0}), $\theta<\Thetastar$ implies that  $\shh(\theta)<0$. Moreover, by Proposition \ref{prop:trajectories}.(\ref{item:A,B}),  we obtain that $\bq(s)= B+A^2(1-e^s)$ holds for $s\in (\shh, 0]$.
Thus, $\bq(0)=0$ is equivalent to  $B=0$, which, by  Proposition \ref{prop:trajectories}.(\ref{item:A,B}) is equivalent to, 
\begin{equation}\label{eq:c42}
	\bQ(\beeta(\shh)) = \frac{e^{-\shh} - 1}{\beeta(\shh)^2}   ,
\end{equation}
(recall $\overline \alpha =1$).

We use the intermediate value theorem to find a  $\theta_0\in (0, \Thetastar)$ satisfying \eqref{eq:c42}.  For $\theta = \Thetastar$, we have $\shh(\Thetastar) = 0$, and $\beeta(\shh(\Thetastar)) = \beeta(0) = \Thetastar$, by \Cref{prop:trajectories}.(\ref{item:sh=0}).  Hence, the left hand side of~\eqref{eq:c42} is $\bQ(\Thetastar) = \Thetastar/4>0$ (we have used \eqref{eq:thetadiam}, the definition of $\Thetastar$), whereas the right hand side is zero.  Next, we show that the left hand side is smaller than the right hand side as $\theta \to 0$.  To this end, we use the combination of \eqref{eq:traj_optim} (at $s=\shh$) and  \Cref{prop:trajectories}.(\ref{item:A,B}) to get,
\begin{align*}
	\beeta(\shh)  & = \theta e^{-\shh} + 2 B\left (1 - e^{-\shh}\right ) + A^2 \left ( 2 - e^{\shh} - e^{-\shh} \right )\\
		 &= \theta + 2 \left ( \bQ(\beeta(\shh)) - \frac{e^{-\shh} - 1}{\beeta(\shh)^2}   \right ) \left (1 - e^{-\shh} \right ) - \dfrac{1}{\beeta(\shh)}\left ( e^{-\shh} - 1 \right )^2\, .
\end{align*}
Therefore, we have:
\begin{equation*}
2\left (\frac{e^{-\shh} - 1}{\beeta(\shh)^2} - \bQ(\beeta(\shh))\right )   = - \dfrac{\theta}{e^{-\shh} - 1} + \dfrac{\beeta(\shh)}{e^{-\shh} - 1} + \dfrac{e^{-\shh} - 1}{\beeta(\shh)} \geq - \dfrac{\theta}{e^{-\shh} - 1} + 2\, .
\end{equation*}
To conclude, it is enough to show that the right hand side in the latter expression has a positive limit as $\theta \to 0$. This is clear as $\liminf \shh (\theta) >0$ as $\theta \to 0$ by  \Cref{prop:trajectories}.(\ref{item:Q}), and the fact that the optimal trajectory $\beeta$ does not become singular as $\theta\to 0$. 
In view of the discussion above, an interior minimum occurs at some $\theta_{\min} \in (0,\Thetastar)$ since we can solve \eqref{eq:c42}.

\medskip

\noindent{\bf \#  Identification of the contact time $\shh(\theta_{\min})$:}
Letting $(\by,\beeta)$ be the optimal trajectory ending at $(\alpha,\theta_{\min})$, we have an explicit expression for $(\by,\beeta)$ in terms of $A$ and $B$. In addition, we know $B=0$   from the discussion preceding (\ref{eq:c42}).  
For notational ease, let $\htheta = \beeta(\shh)$, and  $\tauhh = e^{\shh/2}\in (0,1)$. (The latter is the square root of the original time). We shall show that $\tauhh$ is exactly the $\tau_0$ defined by (\ref{eq:V_condition}).

Recall that $\by(\shh) = \alpha = 4/3$ by definition.
Then, by \Cref{prop:trajectories}.(\ref{item:A,B}), $(\by,\beeta)$ are given by~\eqref{eq:traj_optim} with $A = (\htheta\tauhh)^{-1}$ and $B = 0$. The fourth line (multiplied by $e^{\frac32s}$) and the third line (multiplied by $e^{s}$) of ~\eqref{eq:traj_optim} yield,
\begin{equation}\label{eq:c43}
\begin{cases}
	&\dfrac{4\tauhh^3}{3}
		=\displaystyle \by(\shh)\tauhh^3
		= \frac{4}{3} - 2 \theta_{\min} A\left (1 - \tauhh^2\right ) + \frac{2}{3} A^3\left (1 - \tauhh^2\right )^3 \quad \text{ and}\medskip\\
	&\dfrac{\tauhh}{A}
		\displaystyle=\htheta\tauhh^2
		= \beeta(\shh) \tauhh^2
		= \theta_{\min} - A^2\left (1 - \tauhh^2\right )^2,
\end{cases}
\end{equation}
which, upon rearranging, imply,
\begin{equation}\label{eq:c45}
	A = \left(\frac{(1-\tauhh)^2(2+\tauhh)}{2(1-\tauhh^2)^3} \right)^\frac13 .
\end{equation}
We point out that, according to the definitions of the functions $V$ and $\xi$ given by (\ref{eq:V}) and (\ref{eq:xi}), we have
\begin{equation}
\label{eq:AandV}
A= \frac{V(\tauhh)}{(1-\tauhh^2)^{1/2}} \   \text{ and }\    \frac{1}{A\tauhh}= \frac{(1-\tauhh^2)^{\frac12}}{\tauhh V(\tauhh)}= \xi(\tauhh). 
\end{equation}
In addition, the expression for $\bq$ in (\ref{eq:traj_optim}),  \Cref{prop:trajectories}.(\ref{item:continuous}), which guarantees the continuity of $\bq$, and \Cref{prop:trajectories}.(vi) yields
\[
	A^2(1-\tauhh^2)
		= \bq(\shh)
		= \bQ(\htheta).
\]
Recalling $A =(\htheta\tauhh)^{-1}$, the previous line can be reformulated as:
\begin{equation*}
A^2(1-\tauhh^2)=\bQ\left ((A\tauhh)^{-1}\right ).
\end{equation*}
According to (\ref{eq:AandV}), the previous line is equivalent to,
\[
V(\tauhh)^2 = \bQ(\xi(\tauhh)).
\]
Applying Lemma \ref{lem:VxiQR}, we deduce that $\tauhh\in (0,1)$ is the unique $\tau_0$ such that (\ref{eq:VandR}), and hence (\ref{eq:V_condition}), hold.

\medskip

\noindent{\bf \# Computing $\alpha^*$ in terms of $\tauhh$: }
With the knowledge of $\theta_{\min}$ in hand, via $\shh(\theta_{\min})$, we now compute $\alpha^*$ using \Cref{thm:alpha^*}.(i), which means that we need only to compute $U_{4/3}(4/3,\theta_{\min})$.  Here we use the discussion in \Cref{sec:hamiltonian} and  \Cref{prop:trajectories}.(\ref{item:continuous}).  Again, denote by $(\by,\beeta)$ the optimal trajectory in $\bA(\alpha,\theta_{\min})$ associated with the costate variables  $(\bp,\bq)$.  It follows from~\eqref{eq:hamiltonian}  and \Cref{prop:trajectories}.(\ref{item:continuous}) that
\begin{equation*}
	U_{4/3}(4/3,\theta_{\min})
	 	= - \bH_{\alpha}(\by(0), \beeta(0), \bp(0),\bq(0))
	 	= - \bH_{\alpha}(4/3, \theta_{\min}, A, 0)
	 	= 2 A - \theta_{\min}A^2 - 1.
\end{equation*}
Rearranging the second line in (\ref{eq:c43}) yields, $\theta_{\min}
		= {\tauhh}/{A} + A^2 (1-\tauhh^2)^2$. Using this on the right-hand side of the previous line we find,
		\[
		U_{4/3}(4/3,\theta_{\min}) = 2A-A^2\left(\frac{\tauhh}{A} + A^2 (1-\tauhh^2)^2\right) -1  = A(2-\tauhh)-A^4(1-\tauhh^2)^2 -1 .
		\]
Next, using the expression for $A$ given by (\ref{eq:c45}), we obtain,
\begin{equation*}
\begin{split}
U_{4/3}(4/3,\theta_{\min})+1  &= (2-\tauhh)\left(\frac{(1-\tauhh)^2(2+\tauhh)}{2(1-\tauhh^2)^3} \right)^\frac13 
- \left(\frac{\left((1-\tauhh)^2(2+\tauhh)\right)^\frac13 }{2^\frac13 (1-\tauhh^2)} \right)^4(1-\tauhh^2)^2\\
&= (2-\tauhh)\left(\frac{(1-\tauhh)^2(2+\tauhh)}{2(1-\tauhh^2)^3} \right)^\frac13
- \frac{\left((1-\tauhh)^2(2+\tauhh)\right)^\frac43 }{2^\frac43 (1-\tauhh^2)^2} .
\end{split}
\end{equation*}
Hence, using \Cref{thm:alpha^*}.(i), we see,
\[
	\alpha^*
		= \frac{4}{3} \left((2-\tauhh) \left(\frac{(1-\tauhh)^2(2+\tauhh)}{2(1-\tauhh^2)^3}\right)^\frac13
			 -  \frac{\left((1-\tauhh)^2(2+\tauhh)\right)^\frac43 }{2^\frac43 (1-\tauhh^2)^2}  \right)^{-\frac34},
\]
which, upon simplifying, is equivalent to (\ref{eq:alpha^* TH}).
\end{proof}

\section{Characterizing the optimal trajectories}\label{sec:trajectories}

\Cref{prop:trajectories} is proved piecemeal throughout the sequel.  We do not make note immediately when any portion is proved.  Instead, we compile the proof in the last \Cref{sec:proposition_compilation}

\subsection{Uniqueness of minimizing trajectories -- \Cref{cor:convexity}}

We switch back to the original variables for the proof of this Lemma.

\begin{proof}[Proof of \Cref{cor:convexity}]
First, we show that $U_\alpha$ is convex.  Define the function $F(v,\theta) = v^2/(4\theta)$.  It is jointly convex in $(v,\theta)$, as can be seen readily from its Hessian
\begin{equation*}\renewcommand{\arraystretch}{1.2} 
D^2 F(v,\theta) = \begin{pmatrix}
\frac{1}{2\theta}  & -\frac{v}{2\theta^2}\\
-\frac{v}{2\theta^2}  & \frac{v^2}{2\theta^3}
\end{pmatrix}.\renewcommand{\arraystretch}{1} 
\end{equation*}

Fix $(x_0,\theta_0), (x_1,\theta_1) \in [\alpha,\infty)\times(0,\infty)$, and two optimal trajectories $(\bx_0, \btheta_0) \in \mA(x_0,\theta_0)$ and $(\bx_1,   \btheta_1) \in \mA(x_1, \theta_1)$ respectively. Let $\lambda \in (0,1)$.  According to Lemma \ref{lem:good_trajectories}, we have 
\begin{equation}
\label{eq:apply lemma}
\bx_0(t)\geq \alpha t^{3/2},\quad  \text{and}\quad \bx_1(t)\geq \alpha t^{3/2},
\end{equation}
 and hence $(1-\lambda)\bx_0(t) + \lambda \bx_1(t)\geq \alpha t^{3/2}$ for all $t\in (0,1)$. It is clear that $(1-\lambda)\bx_0 + \lambda \bx_1 \in \mA((1-\lambda) x_0 + \lambda x_1,\theta)$.  Thus, recalling the definitions of $U_\alpha$ and $L_\alpha$, we find
\begin{equation}\label{eq:convexity_computation}
\begin{split}
	U_{\alpha}&((1-\lambda)x_0 + \lambda x_1, (1-\lambda) \theta_0 + \lambda \theta_1)\\
		&\leq \int_0^1 L_{\alpha}(t, (1-\lambda)\bx_0 + \lambda \bx_1, (1-\lambda)\btheta_0 + \lambda \btheta_1,(1-\lambda)\dot \bx_0 + \lambda \dot{\bx}_1, (1-\lambda)\dot\btheta_0 + \lambda \dot{\btheta}_1)dt\\
		&\leq \int_0^1\left((1-\lambda)F(\dot \bx_0, \btheta_0) + \lambda F(\dot{\bx}_1, \btheta_1) + (1-\lambda)\frac{|\dot\btheta_0|^2}{4} + \lambda \frac{|\dot{\btheta}_1|^2}{4}\right) dt - 1\\
		&= (1-\lambda)\int_0^1\left(F(\dot \bx_0, \btheta_0) + \frac{|\dot\btheta_0|^2}{4} - 1\right) dt
			+ \lambda \int_0^1\left(F(\dot{\bx}_1, \btheta_1) + \frac{|\dot{\btheta}_1|^2}{4}-1\right) dt\\
		&= (1-\lambda)\int_0^1 L_\alpha (t, \bx_0, \btheta_0, \dot \bx_0, \dot\btheta_0) dt 
			+ \lambda \int_0^1 L_\alpha(\bx_1, \btheta_1, \dot{\bx}_1, \dot\btheta_1)dt\\
		&= (1-\lambda) U_\alpha(x_0, \theta_0) + \lambda U_\alpha (x_1, \theta_1).
\end{split}
\end{equation}
In the second-to-last line, we have again used (\ref{eq:apply lemma}) and the definitions of $L_\alpha$ and $U_\alpha$. In the last line, we used that $\bx_0$ and $\bx_1$ are minimizing.  Thus, $U_\alpha$ is convex.

Now, suppose that $(\bx_0, \btheta_0)$ and $(\bx_1,   \btheta_1)$ have the same endpoint $(x,\theta)$. Then, the series of inequalities in \eqref{eq:convexity_computation} are all equalities because they coincide on each side. By the strict convexity of quadratic functions, it must be that $\dot \btheta_0 = \dot{\btheta}_1$, and thus $\btheta_0 = \btheta_1$.  Since $F$ is strictly convex in the $v$ variable, we also have $\dot \bx_0 = \dot{\bx}_1$.  We conclude that $(\bx_0,\btheta_0) = (\bx_1, \btheta_1)$.  Hence, optimal trajectories are unique.

The strict convexity of $U_\alpha$ follows immediately from a similar argument.  As such, we omit it.
\end{proof}

\subsection{Trajectories have at most one interval of free motion}

We refer to each portion of the trajectory not intersecting the line $\left\{y = \alpha\right\}$ as ``free motion.''
A preliminary observation is that free motion for all time is not permitted.

\begin{lemma}\label{lem:hits_line}
	Suppose  $\theta\in  \R_+^*$ and  $(\by,\beeta) \in \bA(\alpha,\theta)$ is the optimal trajectory for $U_\alpha(\alpha, \theta)$.  Then there exists a negative time $s_0 \in \R_-^*$ such that $\by(s_0) = \alpha$.
\end{lemma}
\begin{proof}
We proceed by contradiction.  
Suppose that  $\by(s) > \alpha$ for all $s<0$.  By~\eqref{eq:traj_optim}, we have, as $s\to-\infty$,
\begin{equation}\label{eq:sec72}
\begin{split}
		&\by(s) = e^{-3s/2}\left(\alpha - 2\theta A + 2 B A + \frac{2}{3} A^3 \right) + e^{-s/2}\left ( 2\theta A -4 BA - 2 A^3 \right ) + O\left (e^{s/2}\right )
		 \quad \text{ and }\quad\\
	&\beeta(s) = e^{-s}\left(\theta - 2B - A^2\right) + O(1).
\end{split}
\end{equation}
The growth conditions in the definition (\ref{eq:self-similar_admissible}) of $\bA(\alpha,\theta)$ imply,
\[
	\alpha = 2\theta A - 2 B A - \frac{2}{3} A^3,
		\quad\text{and}\quad
	\theta = 2B + A^2.
\]
Returning to~\eqref{eq:sec72}, these conditions imply the strong asymptotic behavior $\by(s) = O(e^{s/2}) \to 0$, as $s\to -\infty$.  This obviously violates the hypothesis that $\by(s) > \alpha$ for all $s$.
\end{proof}

We now investigate the dynamics of a trajectory as it comes into contact with the line.  If $s_0<0$ is a contact time, we expect that $\dot \by(s_0) = 0$, since $\by(s_0) = \alpha$ is a local minimum.  To obtain this, we need to establish sufficient regularity of the optimal trajectories.  From~\eqref{eq:hamiltonian syst2}, this allows us to define $\bp(s_0) = \overline{\alpha}/\beeta(s_0)$ in a continuous way.  Regularity is the purpose of the next statement.

\begin{lemma}[Continuity of $\dot \by$]\label{lem:matching p lem}
Let the assumptions of \Cref{lem:hits_line} hold. Then $\by \in \mathcal C^{1,\frac12}_{\rm loc}(-\infty,0)$.   In particular, if $s_0 \in (-\infty,0)$ is such that $\by(s_0) = \alpha$, then $\dot \by(s_0) = 0$ and $\bp(s_0) = \overline\alpha / \beeta(s_0)$. In addition, $\bq$ is a continuous function.
\end{lemma}

The proof relies on a preliminary Lipschitz bound on $\beeta$.  We state and prove this now, and then continue with the proof of \Cref{lem:matching p lem}.

\begin{lemma}[Lipschitz bounds on $\beeta$]\label{lem:Lipschitz}
Under the assumptions of \Cref{lem:hits_line}, $\beeta \in  W^{1,\infty}_{\rm loc}(-\infty,0)$. In addition, $\bq$ is locally bounded.
\end{lemma}
\begin{proof}
We begin by smoothing the Lagrangian, in order to use classical theory.  For any $\e\in(0,1)$, let
\begin{equation*}
\chi_\e(y) = \dfrac1\e (y-\alpha)_-^2 \, .
\end{equation*}
It is non-negative, convex, twice differentiable, and it takes value 0 if $y \geq \alpha$. 
Then define
\[
	\bL_\alpha^\e(y,\eta,v_y,v_\eta) = \frac{1}{4\eta}\left(v_y + \frac{3}{2} y\right)^2 + \frac{1}{4} \left(v_\eta + \eta\right)^2  - 1+ \chi_\e.
\]
The Lagrangian $\bL_\alpha^\e$ approximates $\bL_{\alpha,+\infty}$ with the state constraint condition that trajectories must lie on the set $\left\{y\geq \alpha\right\}$. By standard arguments,  any sequence of minimizing trajectories $(\by^\e,\beeta^\e)$ associated to $\bL_\alpha^\e$ with endpoint $(\alpha,\theta)$  converges to the minimizing trajectory $(\by,\beeta)$ associated to $\bL_{\alpha}$.

Since $\bL_\alpha^\e$ is smooth, we use~\eqref{eq:hamiltonian syst2} to write the Hamiltonian system for $(\by^\e, \beeta^\e, \bp^\e, \bq^\e)$:
\begin{equation}\label{eq:smoothed_H_system}
\left\{\begin{array}{ll}
\displaystyle	\dot \by^\e = - \frac{3}{2} \by^\e + 2 \beeta^\e \bp^\e,
		\qquad &
		\displaystyle \dot \bp^\e = \frac{1}{2} \bp^\e + \chi_\e'(\by^\e),\medskip\\
\displaystyle	\dot \beeta^\e = - \beeta^\e + 2\bq^\e,
	 	&
	 	\displaystyle \dot \bq^\e = -|\bp^\e|^2.
\end{array}\right.
\end{equation}
Since $\chi_\e$ is regular, each of the quantities above is well-defined. Further, we obtain
\begin{equation*}
	\ddot \beeta^\e + \dot \beeta^\e = 2\dot \bq^\e.
\end{equation*}
First we show that $\bq^\e \in W^{1,1}_{\rm loc}(\R_-^*)$, with bounds in this space independent of $\e$. In the sequel, by saying that a sequence $f_\e$ is ``bounded uniformly in $X_{\rm loc}$,'' we mean that for every compact set $K\subset (-\infty,0]$, there is a constant $C_K$, depending only on $K$, such that $\|f_\e\|_{X(K)} \leq C_K$ for all $\e\in(0,1)$.%

From~\eqref{eq:smoothed_H_system}, we get
\begin{equation*}
	\bp^\e = \frac{1}{2\beeta^\e}\left ( \dot \by^\e + \frac{3}{2} \by^\e\right ).
\end{equation*}
From the formula of $\bL_\alpha^\e$ and the fact that the trajectory is minimizing, it follows that $e^{s/2} \sqrt{\beeta^\e}\bp^\e$ is bounded in $L^2$ uniformly in $\e$.   An argument similar to the one in the proof of \Cref{lem:monotonicity} shows that $\dot\beeta^\e \leq 0$ and, hence, $\beeta^\e(s) \geq \theta$ for all $s$.  In fact, this is easier to prove since~\eqref{eq:smoothed_H_system} is a smooth Hamiltonian system.  From the above, we conclude that $\bp^\e$ is uniformly bounded in $L^2_{\rm loc}$.  This, in turn, implies that $\dot\bq^\e$ is uniformly bounded in $L^1_{\rm loc}$.  

The formula of $\bL_\alpha^\e$  allows us to deduce that $\partial_s(e^s \beeta^\e)= 2e^s\bq^\e(s)$ is bounded uniformly in $L^2_{\rm loc}$, and, thus, in $L^1_{\rm loc}$ as well.  First, we conclude that $\bq^\e$ belongs to  $W^{1,1}_{\rm loc}$ as claimed above.  Hence it is locally uniformly bounded, independent of $\e$, and so is $\bq$ after taking the limit $\e\to 0$.  Next, it follows from differentiating~\eqref{eq:smoothed_H_system} and using the bound on $\dot \bq^\e$ that $\partial_s(e^s \dot \beeta^\e)$ is bounded uniformly in $L^1_{\rm loc}$.  Lastly, we observe that $e^s \beeta^\e$ is bounded uniformly in $L^\infty_{\rm loc}$ as a consequence:
\[
	|e^s \beeta^\e(s)| = \Big|\theta + \int_0^s \partial_s( e^{s'} \beeta^\e(s')) ds'\Big|
		\leq \theta + \| \partial_s(e^{\cdot} \beeta^\e)\|_{L^1([s,0])}.
\]
Combining all three bounds, we see that $e^s \beeta^\e$ is uniformly bounded in $W^{2,1}_{\rm loc}$.

The Sobolev embedding theorem then implies that $e^s \beeta^\e(s)$ is uniformly bounded in $W^{1,\infty}_{\rm loc}$.  From this, it follows that $\beeta^\e$ and $\dot \beeta^\e + \beeta^\e$ are bounded uniformly in $L^\infty_{\rm loc}$.  Taking a linear combination of these two locally bounded functions, we find that $\dot \beeta^\e$ is uniformly bounded in $L^\infty_{\rm loc}$.  Passing to the limit $\e\to0$ yields the local Lipschitz bound on $\beeta$, and the proof is completed.
\end{proof}

With the local $L^\infty$ bound on $\dot \beeta$ from \Cref{lem:Lipschitz}, we are ready to tackle of the continuity of $\dot \by$.

\begin{proof}[Proof of \Cref{lem:matching p lem}]
We follow the same lines as in the previous proof. However, we differentiate the first equation in \eqref{eq:smoothed_H_system} so as to get:
\begin{align*}
\ddot \by^\e &= - \frac{3}{2}  \dot\by^\e + 2 \dot\beeta^\e \bp^\e + 2 \beeta^\e \dot\bp^\e\\
&= - \frac{3}{2} \left ( - \frac{3}{2} \by^\e + 2 \beeta^\e \bp^\e\right )  + 2 \left ( - \beeta^\e + 2\bq^\e \right ) \bp^\e + 2 \beeta^\e  \left ( \frac{1}{2} \bp^\e + \chi_\e'(\by^\e) \right ).
\end{align*}
All terms on the right hand side are bounded except the last one. To handle it, we multiply by $\ddot \by^\e/ \beeta^\e$ on both sides to get
\begin{equation*}
 \dfrac{|\ddot \by^\e|^2}{\beeta^\e} =   \bff^\e \dfrac{\ddot \by^\e}{\beeta^\e} + 2 \chi_\e'(\by^\e) \ddot \by^\e\,, 
 \end{equation*} 
 where $\bff^\e$ is uniformly bounded in $L^2_{\rm loc}(\R_-^*)$. As noted above, $\beeta^\e$ is non-increasing. 
Therefore, dividing by $\beeta^\e$ on a compact sub-interval of $\R_-^*$ is not an issue. To conclude, let multiply by a given test function $\phi \in C_c^\infty(\R_-)$, and integrate by parts:
\begin{align*}
 \int_{-\infty}^0 \dfrac{|\ddot \by^\e|^2}{\beeta^\e} \phi \, ds &= \int_{-\infty}^0 \bff^\e \dfrac{\ddot \by^\e}{\beeta^\e} \phi \, ds + 2 \int_{-\infty}^0 \chi_\e'(\by^\e)  \ddot \by^\e \phi \, ds  \\
 & \leq \left ( \int_{-\infty}^0 (\bff^\e)^2 \beeta^\e  \phi \, ds \right )^{1/2} \left ( \int_{-\infty}^0 \dfrac{|\ddot \by^\e|^2}{\beeta^\e} \phi \, ds \right )^{1/2} - 2 \int_{-\infty}^0  \dfrac{d}{ds} \left ( \chi_\e'(\by^\e)  \phi  \right ) \dot \by^\e \, ds\\
& \leq \left ( \int_{-\infty}^0 (\bff^\e)^2 \beeta^\e  \phi \, ds \right )^{1/2} \left ( \int_{-\infty}^0 \dfrac{|\ddot \by^\e|^2}{\beeta^\e} \phi \, ds \right )^{1/2}\\
& \quad  - 2 \int_{-\infty}^0 \chi_\e''(\by^\e) |\dot \by^e|^2    \phi    \, ds - 2 \int_{-\infty}^0 \dfrac{d}{ds}\left ( \chi_\e(\by^\e)\right ) \dot \phi   \, ds \\
& \leq \left ( \int_{-\infty}^0 (\bff^\e)^2 \beeta^\e  \phi \, ds \right )^{1/2} \left ( \int_{-\infty}^0 \dfrac{|\ddot \by^\e|^2}{\beeta^\e} \phi \, ds \right )^{1/2} + 2 \int_{-\infty}^0    \chi_\e(\by^\e)  \ddot\phi   \, ds\, .
\end{align*} 
We conclude by noticing that $\int_{-\infty}^0    \chi_\e(\by^\e)  \ddot\phi   \, ds$ is bounded uniformly in $\e$  as it appears  in integral form in the variational problem associated with $\bL_\alpha^\e$. As a result, $\ddot \by^\e$ is in $L^2_{\rm loc}$ independent of $\e$. Passing to the limit, we get that $\ddot \by$ is also  in $L^2_{\rm loc}$. As such, $\dot\by$ belongs to $\mathcal C^{\frac12}$, locally over $\R_-^*$. 

Now, if $s_0<0$ is such that $\by(s_0) = \alpha$, then $\dot \by(s_0) = 0$ since $s_0$ is the location of a minimum of $\by$.  On the other hand, since~\eqref{eq:hamiltonian syst2} is satisfied whenever $\by(s) > \alpha$, then we find,
\[
	\lim_{s\to s_0,\, \by(s) > \alpha} \bp(s) = \frac{\overline\alpha}{\beeta(s_0)}.
\]
This concludes the proof of the continuity of $\by$ and $\dot\by$.

The continuity of $\bq$ is a consequence of this.  Indeed, the local boundedness of $\by^\e$ and $\dot \by^\e$,~\eqref{eq:smoothed_H_system}, \Cref{lem:Lipschitz}, and the fact that $\beeta^\e$ is non-increasing imply that $\bp^\e$ is uniformly bounded in $L^\infty_{\rm loc}$.  Since $\dot\bq^\e = -|\bp^\e|^2$, then $\dot \bq^\e$ is uniformly bounded in $L^\infty_{\rm loc}$ as well.  The (Lipschitz) continuity of $\bq$ follows after taking the limit $\e\to0$.
\end{proof}

\begin{figure}
\begin{center}
\includegraphics[width = 0.5\linewidth]{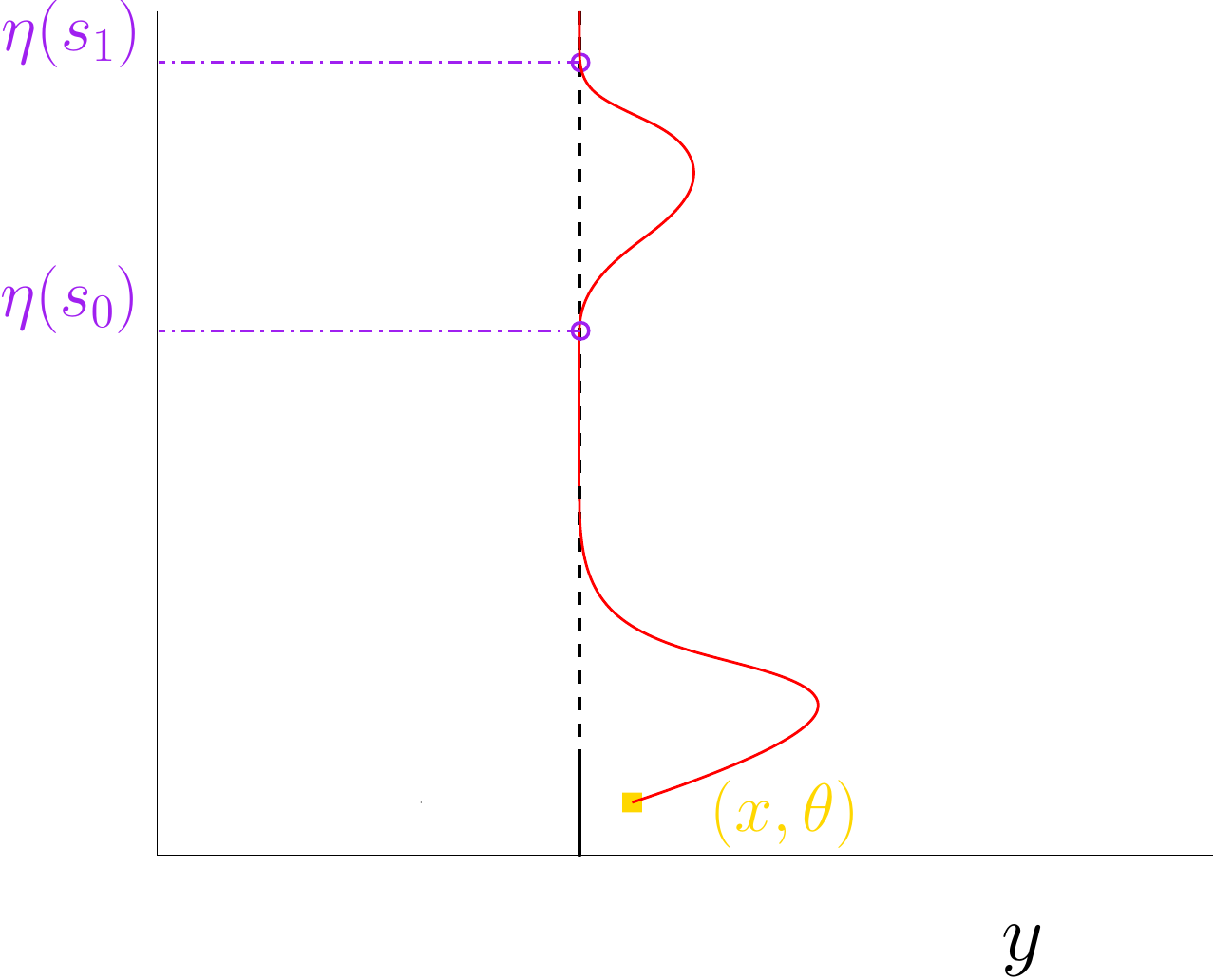} 
\caption{Sketch of a right D-turn between two negative times.  This trajectory cannot be optimal.}
\label{fig:right uturn}
\end{center}
\end{figure}

We now show that situations as in Figure~\ref{fig:right uturn} cannot occur.  This is the last step in proving the preliminary heuristic statement that optimal trajectories must look like those in Figure~\ref{fig:cartoon trajectories}: they stick to the line $\left\{y = \alpha\right\}$ until a critical time $s_0 \in (-\infty,0]$, when they possibly detach to make an excursion in $\left\{y > \alpha\right\}$ (if $s_0<0$) until reaching the endpoint $(x,\theta)$. 

\begin{lemma}[The only D-turn occurs at $s=0$]
\label{lem:one D turn}
Assume that the conditions of \Cref{lem:hits_line} hold.  Let $\theta\in \R_+^*$ and  $(\by,\beeta) \in \bA(\alpha,\theta)$ be an admissible trajectory such that $\by(s) > \alpha$ for all $s \in (s_1,s_0)$ with $s_1 < s_0 < 0$ and $\by(s_0) = \by(s_1) = \alpha$. Then  $(\by,\beeta)$ cannot be an optimal trajectory.  
\end{lemma}

\begin{proof}
We argue by contradiction.  
Since $\by(s) > \alpha$ for all $s\in(s_1,s_0)$, it follows that $(\by,\beeta)$ satisfies~\eqref{eq:traj_optim}; however, it remains to determine the matching conditions for $\bp$.  The fact that $s_1, s_0<0$ enables us to use \Cref{lem:matching p lem} to get that $\dot \by(s_0) = \dot \by(s_1) = 0$ and
\begin{equation}\label{eq:matching_condition}
	\lim_{s \nearrow s_0} \bp(s) = \frac{\overline \alpha}{\beeta(s_0)}
		\quad \text{and}\quad
	\lim_{s \searrow s_1} \bp(s) = \frac{\overline \alpha}{\beeta(s_1)}.
\end{equation}
Let us introduce $\theta_0 = \beeta(s_0)$ . Up to a translation in time, we may assume that $s_0 = 0$ and accordingly, $\dot \by(0) = 0$ and $\bp(0) = \overline\alpha/\theta_0$.

We shall obtain a contradiction by considering the local convexity of the the trajectory $\by$ as it comes in contact with the line.  During free motion, $\by$ is smooth.  Since $\by(0)= \alpha$ is a local minimum and $\dot \by(0) = 0$, it must be that
\begin{equation}\label{eq:convex}
	\limsup_{s\nearrow 0} \ddot \by(s) \geq 0.
\end{equation}
We note that we may not conclude that $\ddot \by(0) \geq 0$ since we have not established the global $\mathcal C^2$ regularity of $\by$.  The weaker claim~\eqref{eq:convex} does not require this extra smoothness and is sufficient for our purposes.  We now use~\eqref{eq:hamiltonian syst2} on the free portion, $s \in (s_1,0)$, to collect some identities that we use to contradict~\eqref{eq:convex}.

We introduce $\tau_1 = e^{s_1/2}$, the square root of the contact time in the original variables.  Then, using~\eqref{eq:traj_optim} along with~\eqref{eq:matching_condition}, we see that
\begin{equation}\label{eq:A}
	A
		= \bp(0)
		= \dfrac{\overline\alpha}{\theta_0}
	\quad \text{and} \quad
	A\tau_1
		= A e^{s_1/2}
		= \bp(s_1)
		= \dfrac{\overline\alpha}{\beeta(s_1)}.
\end{equation}
In particular, we have $\theta_0 = \beeta(s_1)\tau_1$.  Looking at the $(\by,\beeta)$ component of the trajectory~\eqref{eq:traj_optim} at $s=s_1$, we get
\[\begin{cases}
&\displaystyle 	\theta_0 \tau_1 = \tau_1^2 \beeta(s_1) = \theta_0 - 2 B(1-\tau_1^2) - A^2  (1-\tau_1^2)^2\,,
		\quad \text{and}\quad\medskip \\
&	\displaystyle \frac{4\overline\alpha}{3}\tau_1^{3} = \tau_1^{3}\by(s_1) = \frac{4\overline\alpha}{3} - 2\theta_0 A (1 - \tau_1^2) + 2BA (1-\tau_1^2)^2 + \frac{2}{3} A^3 (1-\tau_1^2)^3.
\end{cases}\]
Solving for $B$ and using the identity $A = \overline \alpha / \theta_0$ yields
\[
	\frac{4\overline \alpha}{3} \tau_1^3 = \frac{4\overline\alpha}{3} - 2 \overline\alpha (1- \tau_1^2) + \overline\alpha (1 - \tau_1)(1 - \tau_1^2) - \frac{1}{3}\left(\frac{\overline\alpha}{\theta_0}\right)^3 (1-\tau_1^2)^3.
\]
Collecting all terms that are linear in $\overline \alpha$, the previous line becomes
\begin{equation}\label{eq:cubic}
	\left( \frac{\overline\alpha}{\theta_0}\right)^3 (1-\tau_1^2)^3 = \overline \alpha ( 1- \tau_1)^3.
\end{equation}

We now compute $\limsup_{s\nearrow 0} \ddot \by(s)$ explicitly, using the identities above and using the trajectories given by~\eqref{eq:hamiltonian syst2}.  Indeed, using~\eqref{eq:hamiltonian syst2} along with the fact that $\dot \by(0) = 0$, we have
\begin{align*}
\limsup_{s\nearrow 0}\ddot \by(s) &= \limsup_{s\nearrow 0}\left(- \frac32 \dot \by(s) + 2(-\beeta(s) + 2 \bq(s))\bp(s) + \beeta(s)\bp(s)\right) \\
	&= 0 + 2(-\theta_0 + 2B)A + \theta_0 A
	= (4B - \theta_0) A.
\end{align*}
From above we have  $A = \overline\alpha /\theta_0$ and $B =(\theta_0 (1-\tau_1) - A^2 (1-\tau_1^2)^2)/(2(1-\tau_1^2))$.  We use these identities, along with~\eqref{eq:cubic}, to find
\begin{align*}
\limsup_{s\nearrow 0} \ddot \by(s) 
	&= (4B - \theta_0)A
	= \left(4\left(\theta_0 \frac{1-\tau_1}{2(1-\tau_1^2)} - \frac12 A^2 (1-\tau_1^2)\right) - \theta_0\right) \left(\frac{\overline\alpha}{\theta_0}\right)\\
	 &=\left(\theta_0 \frac{2}{1+\tau_1} - 2 \left(\frac{\overline\alpha}{\theta_0}\right)^2 (1-\tau_1^2) - \theta_0\frac{1+\tau_1}{1+\tau_1}\right) \left(\frac{\overline\alpha}{\theta_0}\right)\\
	 &=\left(\theta_0 \frac{1-\tau_1}{1+\tau_1} - 2 \left(\frac{\overline\alpha}{\theta_0}\right)^2 (1-\tau_1^2)\right) \left(\frac{\overline\alpha}{\theta_0}\right)
	 =\left(\theta_0 \frac{1-\tau_1}{1+\tau_1} - 2 \theta_0 \frac{(1-\tau_1)^3}{(1-\tau_1^2)^2}\right) \left(\frac{\overline\alpha}{\theta_0}\right)\\
	 &=\left( \frac{1-\tau_1}{1+\tau_1} -  \frac{2(1-\tau_1)}{(1+\tau_1)^2}\right) \overline\alpha
	 =\left( \frac{1-\tau_1^2}{(1+\tau_1)^2} -  \frac{2(1-\tau_1)}{(1+\tau_1)^2}\right) \overline\alpha
	 = -\left( \frac{1 - \tau_1}{1+\tau_1}\right)^2 \overline\alpha  .
\end{align*}
This contradicts~\eqref{eq:convex} as $\tau_1 < 1$.  This closes the proof of \Cref{lem:one D turn}. 
\end{proof}

\subsection{The Airy function and related ones} \label{subsubsec:Airy}
The goal of this subsection is to construct the function $\bQ$ involved in \Cref{prop:trajectories}  that plays a key role in establishing \Cref{thm:alpha^*}.  \Cref{fig:bQ} provides an illustration of $\bQ$.

For that purpose, we need to collect some facts about the Airy function $\Ai$, and introduce several auxiliary functions that are useful to prove monotonicity properties of $\bQ$. 
First, we recall that $\Ai$ satisfies
\begin{equation}
\label{eq:Ai}
\Ai''(\xi)=\xi\Ai(\xi).
\end{equation}
We know the precise asymptotics of $\Ai$ as $\xi \to \infty$.  Indeed, from~\cite[Equations 10.4.59, 10.4.61]{AbramowitzStegun},
\begin{equation*}
\begin{cases}
\Ai(\xi)
	=  \dfrac{1}{2\pi^{1/2}\xi^{1/4}} \exp\left (-\dfrac23 \xi^{3/2}\right )
		\left ( 1
			- \dfrac{15}{216} \left ( \dfrac23 \xi^{3/2} \right )^{-1}
			+ o_{\xi \to \infty}\left ( \xi^{-3/2} \right )  \right )\medskip\\
\Ai'(\xi) = 
	-\dfrac{\xi^{1/4}}{2\pi^{1/2}} \exp\left (-\dfrac23 \xi^{3/2}\right )
		\left ( 1 
			+ \dfrac{21}{216} \left ( \dfrac23 \xi^{3/2} \right )^{-1} 
			+ o_{\xi \to \infty}\left ( \xi^{-3/2} \right ) \right ),
\end{cases}
\end{equation*}
and, for $\mR$ defined by (\ref{eq:riccati}),
\begin{equation}\label{eq:R asymptotics}
\mR(\xi) = \xi^{1/2} +  \dfrac14 \xi^{-1} - \frac{5}{32}\xi^{-5/2} + \frac{15}{64 }\xi^{-4} + o_{\xi\to\infty}(\xi^{-4}) . 
\end{equation}
Recall that $\Xi_0$ is the largest zero of $\Ai$.  The asymptotics of $\mR$ near $\Xi_0$ are also known.  In particular,
\begin{equation*}
	\lim_{\xi \searrow \Xi_0} \mR(\xi) = -\infty
		\quad \text{ and } \quad
	\lim_{\xi \searrow \Xi_0} \mR'(\xi) = -\infty.
\end{equation*}

\subsubsection{The auxiliary function $\mE = \mR'$}
Next we introduce one more function.   For $\xi\in (\Xi_0, \infty)$, let
\[
\mE(\xi) = \mR(\xi)^2 - \xi.
\]
By the definition of $\mR$ in (\ref{eq:riccati}) and by (\ref{eq:Ai}), we have
\begin{equation}
\label{eq:R' and E}
\mR'(\xi) = -\left(\frac{\Ai(\xi)\Ai''(\xi) - (\Ai'(\xi))^2}{\Ai(\xi)^2}\right) = \mR^2(\xi)-\xi= \mE(\xi).
\end{equation}
We summarize further facts in the next lemma.
\begin{lemma}
\label{lem:E}
We have,
\begin{equation}\label{eq:c23}
	\lim_{\xi \searrow \Xi_0} \mE(\xi) =+\infty
		\quad \text{ and } \quad
	\mE(\xi) = \frac12 \xi^{-1/2} + \mathcal O(\xi^{-2})\,,\; \text{as $\xi \to +\infty$}.
\end{equation}
For $\xi>\Xi_0$, we have,
\begin{align}
\label{eq:E'}
&\mE'(\xi) = 2\mR(\xi) \mE(\xi) - 1,\\
&\mE''(\xi) = 2 \mE(\xi)^2 + 2 \mR(\xi) \mE'(\xi). \label{eq:E''}
\end{align}
Finally, for all $\xi\in (\Xi_0, \infty)$, we have $\mE'(\xi)<0$ and $\mE(\xi)>0$.
\end{lemma}
\begin{proof}

The behavior at $\infty$ claimed in~\eqref{eq:c23} follows from the asymptotics in~\eqref{eq:R asymptotics}. Equation (\ref{eq:E'}) follows from the definition of $\mE$ and (\ref{eq:R' and E}). Finally, (\ref{eq:E''}) is obtained by differentiating (\ref{eq:E'}) and again using (\ref{eq:R' and E}).

Next, we prove that $\mE'(\xi)<0$ for all $\xi$. For the sake of contradiction, suppose that there is a critical point, $\xi_0$, of $\mE$.  
Then, by (\ref{eq:E''}), we have 
\[
\mE''(\xi_0)
		= 2\mE(\xi_0)^2.
\]
In addition, \eqref{eq:E'} informs us that $\mE(\xi_0)\neq 0$. Therefore, $\xi_0$ is a strict local minimum.

The limiting behavior of $\mE$ at $\infty$ implies that there is also a strict local maximum $\xi_M\in (\xi_0,\infty)$.  On the other hand, the argument that showed that $\xi_0$ is a strict local minimum applies to $\xi_M$ as well.  We conclude that $\xi_M$ is both a strict local maximum and a strict local minimum, which is a contradiction. Hence, there is no critical point, and we have $\mE' < 0$, and $\mE>0$.
\end{proof}

\subsubsection{The auxiliary function $\mF$}

\begin{lemma}\label{lem:mF>0}
For $\xi>\Xi_0$, define $\mF$ via 
\begin{equation}\label{eq:mF}
\mF(\xi) = 
 \mE(\xi)^2 + 2 \mR(\xi)^2 \mE(\xi) - \mR(\xi)\, . 
\end{equation}
We have $\mF(\xi)>0$ for all $\xi$. 
\end{lemma}
\begin{proof}
 To begin, we notice that $\mF(\xi) \to \infty$ as $\xi \searrow \Xi_0$. Next, we understand the behavior of $\mF(\xi)$ as $\xi\to\infty$.  Using~\eqref{eq:R asymptotics} and the definition of $\mE$, we have, as $\xi\to \infty$,
 \[
	\mE(\xi) = \frac{1}{2}\xi^{-1/2} - \frac{1}{4}\xi^{-2}+\frac{25}{64} \xi^{-7/2} + o_{\xi\rightarrow\infty}(\xi^{-7/2}), \quad \text{and}\quad \mathcal{E}(\xi)^2 = \frac{1}{4}\xi^{-1} - \frac{1}{4}\xi^{-5/2} + o_{\xi\rightarrow\infty}(\xi^{-5/2}).
 \]
 Using the relationship $\mR^2(\xi)=\mE(\xi)+\xi$ in the second term of (\ref{eq:mF}) gives  $\mF(\xi) = 3\mE(\xi)^2 + 2 \xi \mE(\xi) - \mR(\xi)$. Then a straightforward computation yields,
 \[\begin{split}
 	\mF(\xi)
		&= \frac{3}{16}\xi^{-5/2}+o_{\xi\rightarrow\infty}(\xi^{-5/2}).
\end{split}\]
Hence, $\mF(\xi)$ is positive for all sufficiently large $\xi$.

We argue by contradiction to prove that $\mF(\xi)>0$ everywhere. 
Suppose that $\mF$ hits zero at $\xi_1>\Xi_0$.  Since $\mF(\xi) >0$ when $\xi \searrow \Xi_0$ or $\xi \gg 1$, then there exists $\xi_0 \geq \xi_1$ such that $\mF(\xi_0) = 0$ and $\mF'(\xi_0) \geq 0$. 
Evaluating~\eqref{eq:mF} at $\xi_0$ yields that
\begin{equation}\label{eq:c31}
\mE(\xi_0)^2 = \mR(\xi_0) - 2\mR(\xi_0)^2\mE(\xi_0)\end{equation}.
The derivative of $\mF$ (\ref{eq:mF}) can be calculated unsing the relations \eqref{eq:R' and E} and (\ref{eq:E'}):
\begin{align*}
\mF'(\xi) & = 2 \mE(\xi)\left ( 2 \mR(\xi) \mE(\xi) - 1 \right ) + 4\mR(\xi) \mE(\xi)^2 + 2\mR(\xi)^2 \left ( 2 \mR(\xi) \mE(\xi) - 1 \right )  - \mE(\xi)  \\
& = 8\mR(\xi) \mE(\xi)^2 + 4 \mR(\xi)^3\mE(\xi) - 2\mR(\xi)^2 - 3\mE(\xi).
\end{align*}
Evaluating at $\xi_0$, we can simplify further using \eqref{eq:c31}:
\begin{align*}
\mF'(\xi_0)
& = 8\mR(\xi_0) \mE(\xi_0)^2 + 2 \mR(\xi_0) (\mR(\xi_0) - \mE(\xi_0)^2)  - 2\mR(\xi_0)^2 - 3\mE(\xi_0)\\
& = 3\mE(\xi_0)   ( 2 \mR(\xi_0) \mE(\xi_0)    - 1) = 3\mE(\xi_0)\mE'(\xi_0) < 0 .
\end{align*}
This is a contradiction.
\end{proof}

\subsubsection{Definition and properties of $\bQ$}
\label{Subsec:Q}
We are now ready to introduce $\bQ$.
\begin{figure}
\begin{center}
\includegraphics[width = 0.493\linewidth]{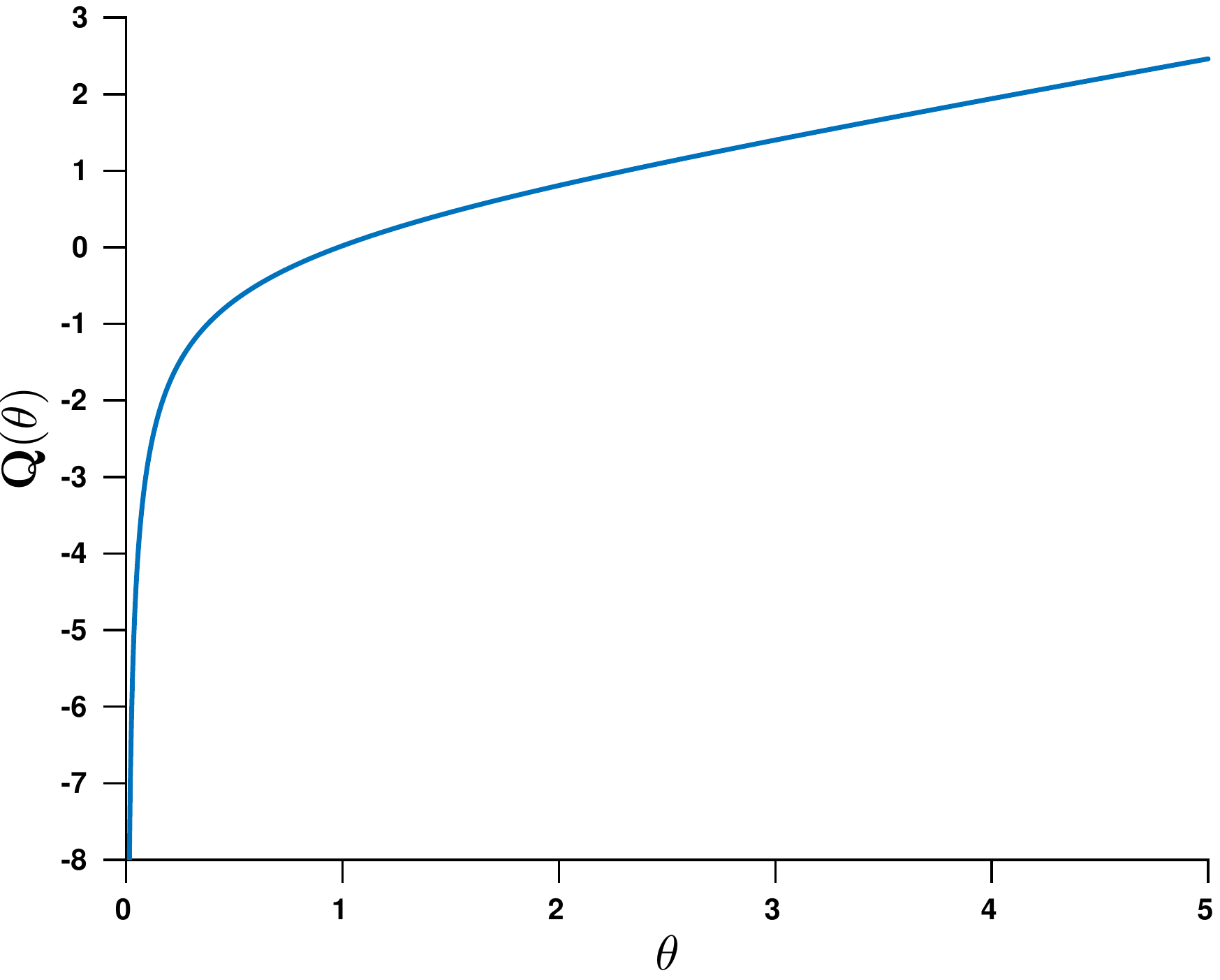} \
\includegraphics[width = 0.493\linewidth]{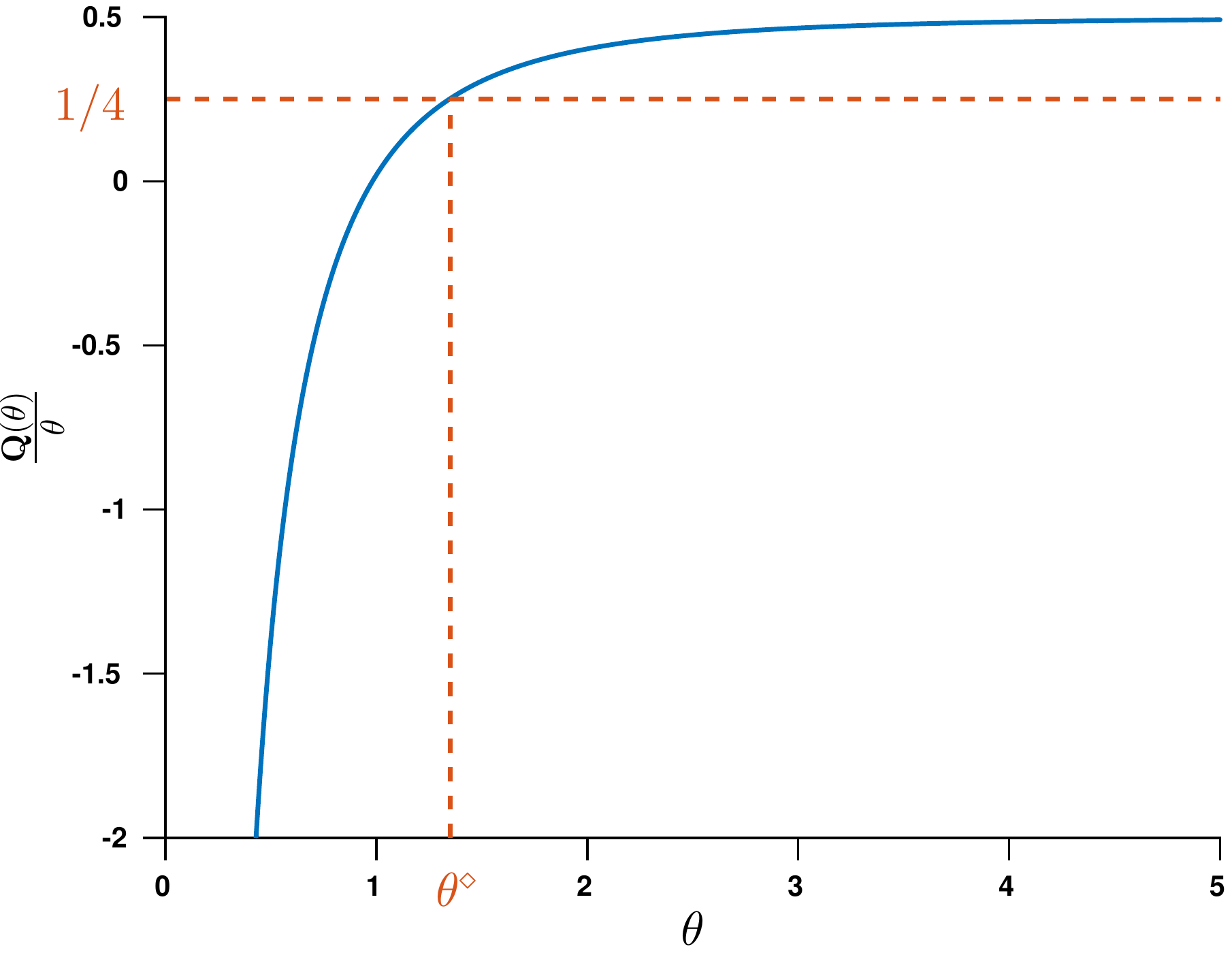}
\caption{Sketch of $\bQ$.  Notice that both $\bQ$ and $\bQ(\theta)/\theta$ are increasing in $\theta$.}
\label{fig:bQ}
\end{center}
\end{figure}

\begin{lemma}\label{lem:Qtheta}
\begin{enumerate}[(i)]
\item For each $\theta>0$, there is a unique solution of $q = \mR\left ( q^2 - \theta^{-1}\right )$ such that $q^2 - \theta^{-1}>\Xi_0$. 
We define the function $\bQ(\theta)$ as the root of this equation. Alternatively speaking, $\bQ(\theta)$ is defined via the following implicit relationship,
\begin{equation}
\label{eq:bQimp}
\bQ(\theta) = \mR\left (\bQ(\theta)^2- \theta^{-1}\right ).
\end{equation}
\item \label{item:bQinc} The function $\theta\mapsto \bQ(\theta)/\theta$ is strictly increasing, continuously differentiable, and it converges to $1/2$ as $\theta\to +\infty$. 
\end{enumerate}
\end{lemma}
\begin{proof}
We begin with the proof of (i).  First, from Lemma \ref{lem:E}, $\mE(\xi) = \mR(\xi)^2 - \xi$ is a bijective function from $(\Xi_0,\infty)$ to $(0,\infty)$.  Therefore, for each $\theta$, there 
is a unique $\xi_0 \in (\Xi_0, \infty)$ such that $\mR(\xi_0)^2 - \xi_0 = \theta^{-1}$. 

If $\mR(\xi_0) \geq 0$, we let $q_0 = \left ( \xi_0 + \theta^{-1}\right ) ^{1/2}$.  
It is clear that $q_0$ is a root of $q = \mR\left ( q^2 - \theta^{-1}\right )$.   In order to see that $q_0$ is unique, suppose that $q_1 = \mR( q_1^2 - \theta^{-1})$ and $q_1^2 - \theta^{-1}> \Xi_0$.  Then, letting $\xi_1 = q_1^2 - \theta^{-1}$, we find $\mR(\xi_1)^2 - \xi_1 = \theta^{-1}$.  The fact that $\mR(\xi)^2 - \xi$ is bijective implies that $\xi_1 = \xi_0$.  It then follows that $q_1 = \mR(\xi_1) = \mR(\xi_0) = q_0$.  If, on the other hand, $\mR(\xi_0) < 0$, then the proof is similar, after choosing $q_0 = - \sqrt{\xi_0 + \theta^{-1}}$.

We now present the proof of (\ref{item:bQinc}). 
The fact that $\bQ$ is differentiable is a simple result of the implicit function theorem and Lemma \ref{lem:E}, in which we prove that $2\mR \mR' = 2\mR \mE < 1$ everywhere.

We now establish that $\theta \mapsto \bQ(\theta)/\theta$ is strictly increasing.  
It is equivalent to the positivity of $\theta\bQ'(\theta) - \bQ(\theta)$. Letting $\xi_0 = \bQ(\theta)^2 - \theta^{-1}$ and differentiating (\ref{eq:bQimp}) yields,
\begin{equation*}
\bQ'(\theta) = \mR'(\xi_0)\left ( 2 \bQ(\theta) \bQ'(\theta) + \dfrac{1}{\theta^2} \right ),
\end{equation*}
which, upon rearranging becomes,
\begin{equation*}
	\theta \bQ'(\theta) = \dfrac{\mR'(\xi_0)}{\theta\left (1 - 2\bQ(\theta)\mR'(\xi_0)\right )}.
\end{equation*}
Now, recall the following relationships previously established:
\begin{equation*}
\bQ(\theta) = \mR(\xi_0)\, ,\quad \dfrac{1}{\theta} =  \bQ(\theta)^2 - \xi_0 = \mR(\xi_0)^2 - \xi_0 = \mE(\xi_0)\, , \quad \text{and}\quad \mR'(\xi_0) = \mE(\xi_0)\, .
\end{equation*}
Using these, we find,
\begin{equation*}
	\theta\bQ'(\theta) - \bQ(\theta) > 0
		\quad \text{ if and only if } \quad
		\dfrac{\mE(\xi_0)^2}{1 - 2 \mR(\xi_0) \mE(\xi_0)}    - \mR(\xi_0)> 0\, .
\end{equation*}
Recall, from Lemma \ref{lem:E} and (\ref{eq:E'}), that $0 > \mE'(\xi) = 2 \mR(\xi) \mE(\xi) - 1 $ holds for all $\xi\in(\Xi_0,\infty)$. Thus, together with the definition of $\mF$ in (\ref{eq:mF}) we find the equivalence,
\begin{equation*}
	\theta\bQ'(\theta) - \bQ(\theta) > 0
		\quad \text{ if and only if } \quad
		\mF(\xi_0) > 0.
\end{equation*}
The result follows from  Lemma \ref{lem:mF>0}.

Finally, we need to check that $\bQ(\theta)/\theta$ converges to $1/2$ as $\theta\to +\infty$.  
 First, we establish, 
 \begin{equation}
 \label{eq:bQ>0}
 \bQ(\theta) > 0 \text{ for $\theta$ large enough.}
 \end{equation}
We recall that $\Ai'$ has a largest zero, which we denote $\Xi_1$ ($\Xi_1\approx -1.02$), such that $\Xi_0< \Xi_1 <0$. We have $\mR(\xi)>0$ for $\xi>\Xi_1$. Therefore,  for $\theta>|\Xi_1|^{-1}$, we have $q^2-\theta^{-1}> \Xi_1$, and hence $\mR(q^2-\theta^{-1})>0$  for all $q$. Recalling the definition of $\bQ$ concludes the proof of (\ref{eq:bQ>0}). 

Since $\bQ(\theta)/\theta$ is increasing, it follows that $\bQ(\theta)$ tends to infinity with $\theta$.  We are then justified in using the asymptotic expansion \eqref{eq:R asymptotics} in  \eqref{eq:bQimp}:
\[
	\bQ(\theta)
		= \left(\bQ(\theta)^2 - \frac{1}{ \theta}\right)^{1/2} + \frac{1}{4}\left(\bQ(\theta)^2 - \frac{1}{ \theta}\right)^{-1} + o_{\theta\to\infty} \left(\bQ(\theta)^2 - \frac{1} {\theta}\right)^{-1}.
\]
Dividing by $\theta$ and expanding the first  and second terms on the right hand side, we see that
\begin{equation*}
	\frac{\bQ(\theta)}{\theta}
		= \frac{\bQ(\theta)}{\theta}\left(1 - \frac{1}{2\theta\bQ(\theta)^2} + O_{\theta\to\infty}\left(\frac{1}{\theta^2\bQ(\theta)^4}\right)\right)\\
		  +  \frac{1}{4}\frac{1}{\theta\bQ(\theta)^2}  + o_{\theta\to\infty} \left(\frac{1}{\theta\bQ(\theta)^2}\right).
\end{equation*}
Since $\bQ(\theta)\to +\infty$ as $\theta\to +\infty$, it follows that
\[
	0 = -\frac{1}{2\theta^2\bQ(\theta)} + \frac{1}{4\theta \bQ(\theta)^2} + o_{\theta\to\infty}\left( \frac{1}{\theta \bQ(\theta)^2}\right)\,,
\]
from which   we obtain, after multiplying this by $2\theta \bQ(\theta)^2 $,
\[
	\frac{\bQ(\theta)}{\theta}
		= \frac{1}{2} + o_{\theta\to\infty}(1),
\]
which concludes the proof.
\end{proof}

We next prove Lemma \ref{lem:VxiQR}, which crucially relies on the monotonicity of $\bQ(\theta)/\theta$.

\begin{proof}[Proof of Lemma \ref{lem:VxiQR}]
We recall that, according to Lemma \ref{lem:Qtheta}.(ii), $\bQ(\theta) = \mR(\bQ(\theta)^2- \theta^{-1})$. The equivalence of (\ref{eq:VandQ}) and (\ref{eq:VandR}) follows directly from this after taking $\theta=\xi(\tau)$ as long as  $V(\tau)^4 - \xi^{-1}(\tau) > \Xi_0$. 

Next, elementary calculations yield
\[
	\xi(\tau) 
		= \left(1 + \frac{1}{\tau}\right)\dfrac{(1-\tau)^\frac13}{(1+\tau/2)^\frac13},
\]
and, hence, $\xi$ is strictly decreasing in $\tau$.

We now claim that $V(\tau)^2/\xi(\tau)$ is strictly increasing in $\tau$. Indeed, a short computation implies
\[
	\dfrac{V(\tau)^2}{\xi(\tau)}
		= \dfrac{\tau (1+\tau/2)}{(1+\tau)^2}
		= \frac12 \left(1 - \frac{1}{(1+\tau)^2}\right).
\]  
This, combined with the fact that $\xi$ is strictly decreasing, and Lemma \ref{lem:Qtheta}.(\ref{item:bQinc}) implies that $\bQ(\xi(\tau))/\xi(\tau) - V(\tau)^2/\xi(\tau)$ is strictly decreasing in $\tau$.  This implies that there is at most one $\tau$ such that (\ref{eq:VandQ}) holds. See \Cref{fig:unique_time}.
\end{proof}

\subsection{The dynamics on the line}\label{sec:dynamics_on_line}

With Lemma \ref{lem:one D turn} at hand, we know that trajectories make at most one excursion to the right of the line $\left\{y = \alpha\right\}$.  In the sequel, we show that this excursion occurs if and only if the endpoint $(\alpha,\theta)$ is such that  $\theta \in (0,\Thetastar)$, where we refer to   the threshold $\Thetastar$ given in \eqref{eq:thetadiam}, which,  according to Lemma \ref{lem:Qtheta}, is uniquely defined.  One key step to understanding this is the dynamics on the line.  It should be noted, however, that the results in this section are used for more than this one consequence.

In order to prepare the computation of the trajectory off the line (Section \ref{sec:analytic}), two constants of integration are needed: $A$ and $B$, see {\em e.g.}\ \eqref{eq:traj_optim}. In the sequel, we gather enough additional equations at the junction with the line in order to resolve the problem. The cornerstone is the relation between $\bq(\shh)$ and $\beeta(\shh)$, which is established in \Cref{lem:Airy} below. This enables us to bring  the condition at $s=-\infty$ in the definition of $\bA(\alpha,\theta)$ in  \eqref{eq:self-similar_admissible} down to a condition at the contact time $s = \shh$. Note that the latter is well defined by Lemma \ref{lem:hits_line} and Lemma \ref{lem:one D turn}:

\begin{definition}[Contact time]
Suppose $\theta>0$, and  let $(\by,\beeta)\in\bA(\alpha,\theta)$ be the optimal trajectory. There exists a unique $\shh = \shh(\theta)\leq 0$ such that  $\by(s) = \alpha$ if and only if $s \leq \shh$. 
\end{definition}

\begin{lemma}\label{lem:Airy}
Suppose $\theta>0$, and  let $(\by,\beeta)\in\bA(\alpha,\theta)$ be the optimal trajectory. Let $\shh \in\R_-$ be the contact time. Then $\bq$, defined by~\eqref{eq:ODE line}, satisfies, for $s< \shh$,
\begin{equation}\label{eq:relation airy}
	 \bq(s) = \bQ_\alpha \left (\beeta(s)\right ). 
\end{equation}
\end{lemma}
\begin{proof}
In this proof, we assume that $\overline{\alpha} =1$ without loss of generality. The appropriate relationship \eqref{eq:relation airy} can be recovered afterwards from the scaling $\beeta(s) = \overline{\alpha}^{2/3} \beeta_{4/3}(s)$, see \Cref{rem:scaling}. 

As above, we may also assume without loss of generality that $\shh = 0$ up to a time shift of the trajectory. First, we recall that $(\by,\beeta) \in \bA(\alpha,\theta)$ implies
\begin{equation}\label{eq:growth condition}
	\lim_{s\to-\infty}  e^s \beeta(s) = 0.
\end{equation}
Second, due to~\Cref{lem:monotonicity}, we recall that  $\beeta$ is non-increasing.  In view of~\eqref{eq:ODE line}, this implies  $\beeta \geq 2\bq$.

The first conclusion we make from these two facts is that $\bq$ and $\beeta$ both tend to infinity.  Indeed, first suppose that $\beeta$ is bounded.  Since $\beeta$ is monotonic, then there exists $\beeta_\infty$  such that $\beeta(s) \in (\theta, \eta_\infty)$ for all  $s<0$.  It follows that $\bq(s)$ remains bounded from above as well.  From~\eqref{eq:ODE line}, we find
\[
	\bq(s)
		= \bq(0) + \int_s^0  \frac{1}{\beeta(s')^2}ds'
		\geq \bq(0) + \int_s^0 \frac{1}{\eta_\infty^2} ds'.	
\]
After taking $s\to-\infty$, we see that $\bq(s) \to \infty$, which  contradicts the boundedness of   $\beeta$. 

Similarly, if $\bq$ is bounded, there exists $q_\infty$ such that $\bq(s) \leq q_\infty$ for $s<0$.  Since $\beeta$ tends to infinity, choose $S>0$ large enough that $\beeta(-S) > 2q_\infty$.  Then, using this as well as~\eqref{eq:ODE line}, we obtain, for $s<0$,
\begin{align*}
\beeta(s-S)e^{s-S} &= \beeta(-S)e^{-S}-\int_{s-S}^{-S} \dfrac d{ds'}\left (\beeta(s')e^{s'}\right )\, ds'= \beeta(-S)e^{-S}-\int_{s-S}^{-S} 2q(s') e^{s'} \, ds'\\
&\geq \beeta(-S)e^{-S} -2q_\infty \int_{s-S}^{-S}  e^{s'} \, ds'= \beeta(-S)e^{-S} -2q_\infty\left(e^{-S}-e^{s-S}\right).
\end{align*}
By taking the asymptotic limit, we find: 
\[
\liminf_{s\rightarrow -\infty} e^{s-S} \beeta(s-S) \geq \left ( \beeta(-S) -2q_\infty\right )e^{-S} >0,
\]
where the second inequality follows from our choice of $S$. However, this is impossible due to \eqref{eq:growth condition}. Thus, $\bq$ cannot be uniformly bounded.

In addition, the first equation in~\eqref{eq:ODE line}  implies that $\bq$ is monotonic. We conclude that the following limits hold true, 
\begin{equation}\label{eq:condition at oo}
	\lim_{s\to-\infty} \beeta(s)
		= \lim_{s\to-\infty} \bq(s)
		= + \infty.
\end{equation}
Now, consider the combination of $\beeta$ and $\bq$ given by
\begin{equation}\label{eq:def h}
\bxi(s) =  \bq(s)^2  - \frac{1}{\beeta(s)}\, .
\end{equation}
It follows from \eqref{eq:ODE line} that $\dot \bxi = -1/\beeta$.  Let $\phi(s) = \mR(\bxi(s)) -  \bq(s)$.  First notice that, using~\eqref{eq:def h}, along with~\eqref{eq:R asymptotics}, it follows that $\phi(s) \to 0$ as $s\to-\infty$. 
Second, we find,
\[
	\dfrac{d}{ds} \mR(\bxi(s))
		= \mR'(\bxi) \dot \bxi
		= -  \left( \mR(\bxi)^2 - \bxi \right) \frac{1}{\beeta}		 
		=   -   \frac{ \mR(\bxi)^2 }{\beeta} + \frac{\bq^2}{ \beeta} - \frac{1}{\beeta^2} ,
\]
where we have used~\eqref{eq:R' and E} to obtain the second equality. Thus, we obtain,
\[
	\dot \phi
		= -  \left ( \frac{\mR(\bxi)^2 -  \bq^2}{\beeta}\right )
		= -  \phi \left (  \frac{  \mR(\bxi) + \bq}{\beeta}\right ).
\]
For any $s<0$, integrating this from $s$ to $0$ yields the identity
\begin{equation}\label{eq:c21}
	\phi(s) = \phi(0) \exp\left (\int_s^0 \left (  \frac{ \mR(\bxi(s')) + \bq(s')}{\beeta(s')}\right ) ds'\right ).
\end{equation}
The definition of $\bxi$ in (\ref{eq:def h}) and~\eqref{eq:condition at oo} imply that $\lim_{s\rightarrow-\infty} \bxi(s)= + \infty$. This, together with the asymptotics for $\mR$ in (\ref{eq:R asymptotics}), and (\ref{eq:condition at oo}), imply that $\overline\alpha^{2/3}\mR(\bxi(s')) + \bq(s')$ is  positive for $s'<S$ negative enough (it even tends to infinity). We deduce that 
\begin{equation}\label{eq:c22}
	\liminf_{s\to-\infty} \exp\left (\int_s^0 \left (  \frac{  \mR(\bxi(s')) + \bq(s')}{\beeta(s')}\right ) ds'\right  ) \geq   \exp\left (\int_S^0 \left (  \frac{  \mR(\bxi(s')) + \bq(s')}{\beeta(s')}\right ) ds'\right ) >0 .
\end{equation}
The fact that $\phi(s) \to 0$ as $s\to -\infty$, along with~\eqref{eq:c21} and~\eqref{eq:c22}, imply $\phi(0) = 0$.  Using this information, with~\eqref{eq:c21} again, shows that $\phi(s)= 0$ for all $s<0$.  Thus, according to the definition of $\phi$, we have,
\[
	\bq(s) =  \mR(\bxi(s)) =  \mR\left( \bq(s)^2  - \frac{1}{\beeta(s)}\right),
\]
which is equivalent to $ \bq(s) =   \bQ\left (\beeta(s) \right )$.
This concludes the proof.
\end{proof}

We conclude this section with 
a relatively precise description of the behavior as $s\to -\infty$ (or, equivalently, $t\to 0$).

\begin{lemma}[Anomalous behavior as $s\to -\infty$]
\label{lem:anomalous}
The following asymptotics hold for optimal trajectories:
\begin{equation*}
\begin{cases}
&\by(s) = \alpha \medskip\\
&\beeta(s)\sim \dfrac32 \alpha^{2/3}|s|^{1/3} 
\end{cases}\qquad \mathrm{as}\; s\to -\infty.
\end{equation*} 
\end{lemma}
Note that~\eqref{eq:theta scaling sec 3}, the anomalous scaling in the original variables, follows from \Cref{lem:anomalous}.

\begin{proof}
We may   assume $\overline{\alpha} = 1$ up to a scaling argument \eqref{eq:link optimal traj}, as in the previous proof. 

The first item is obvious by definition of $\shh>-\infty$. The second one can be deduced from a combination of Lemma \ref{lem:Airy}, \eqref{eq:def h}, and~\eqref{eq:ODE line}. 
Indeed, we have the following asymptotics:
\[
\dot \bq = - \dfrac1{\beeta^2} = -\left ( \bq^2 - \bxi \right )^2 = - \left ( \mR(\bxi)^2 - \bxi \right )^2 
 = - \left ( \dfrac12  \bxi^{-1/2} + o_{s\to-\infty}\left ( \bxi^{-1/2} \right )  \right )^2 
\sim -  \frac{1}{4\bq^2}\, .
\]
Hence, we see that
\begin{equation*}
	\bq(s) \sim \left (\frac34 \right )^{1/3}|s|^{1/3} \quad \mathrm{as}\; s\to -\infty\, .
\end{equation*}
Similarly, we deduce that
\begin{equation*}
\beeta(s)\sim 2 \left (\frac34 \right )^{1/3}|s|^{1/3} \quad \mathrm{as}\; s\to -\infty\, .
\end{equation*}
\end{proof}

\subsection{The dynamics off the line}

\label{sec:analytic}
In this subsection, we fix $\theta >0$ and the associated optimal trajectory $(\by,\beeta)\in \bA(\alpha,\theta)$. Recall~\eqref{eq:thetadiam}, the definition of the threshold value $\Thetastar$ such that $\Thetastar = 4\bQ_\alpha(\Thetastar)$.

The first step is to show that the contact time is non zero ($\shh<0$) if $\theta < \Thetastar$. Alternatively speaking, for endpoints below the threshold, the trajectory makes a free motion excursion in $\left\{y>\alpha\right\}$. 

\begin{lemma}\label{lem:exit the line}
There cannot exist $s_1< s_0\leq 0$ such that $\by(s) = \alpha$ and $\beeta(s) < \Thetastar$ for all $s \in (s_1,s_0)$.
\end{lemma}

\begin{proof}
We argue by contradiction. Suppose there exist such times $s_1 < s_0$.  Then, we test the optimality $(\by,\beeta)$ against a perturbation $(\by+\bepsilon,\beeta)$ compactly supported in $(s_1,s_0)$, and such that $\bepsilon\geq 0$ in order to preserve the condition $\by + \bepsilon\geq \alpha$. Then, by the optimality of the trajectory with respect to the Lagrangian \eqref{eq:Lagrangian ss},
\begin{equation*}
	0 \leq \int_{-\infty}^0   \dfrac1{\beeta(s)}\left (\dot\bepsilon(s) + \frac32 \bepsilon(s)\right)e^s\, ds.
\end{equation*}
Integration by parts yields
\begin{equation*}
	0\leq \int_{-\infty}^0  \left ( \frac{\dot\beeta(s)}{\beeta(s)^2}  + \frac{1}{2\beeta(s)}  \right )  \bepsilon(s) e^s\, ds\, .
\end{equation*}
Since $\bepsilon$ is compactly supported in $(s_0,s_1)$ and since $(\by,\beeta)$ satisfies~\eqref{eq:hamiltonian syst2} on $(s_1,s_0)$, it follows that $\dot\beeta + \beeta = 2\bq$.  Hence,
\[
	0 \leq \int_{-\infty}^0 \left (   4 \bq(s) -   \beeta(s)   \right )  \frac{\bepsilon(s)}{2\beeta(s)^2} e^s\, ds\, .
\]
By the arbitrariness of $\bepsilon\geq 0$, it follows that $4\bq(s) \geq \beeta(s)$ for all $s\in (s_0,s_1)$.  However, using \Cref{lem:Airy}, this implies that $4\bQ_\alpha(\beeta(s))\geq \beeta(s)$, which cannot hold if $\beeta(s)<\Thetastar$ by the definition of $\Thetastar$ in~\eqref{eq:thetadiam} and the monotonicity established in \Cref{lem:Qtheta}. 
\end{proof}

We set some notation.  Given $\theta\geq 0$, let $\htheta(\theta) = \beeta(\shh(\theta))$, where $(\by,\beeta)$ is the optimal trajectory associated to $(\alpha, \theta)$. We had used this notation already in the proof of  \Cref{thm:alpha^*}.(ii).

We continue with a characterization of $\htheta$ at the contact time. We remark that the map $\theta \mapsto \htheta$, defined on the line $\left\{y = \alpha\right\}$, connects the two values $\beeta(0)$ and $\beeta(\shh)$ at the two extremities of the free excursion. 

\begin{lemma}\label{lem:hitting_monotonicity}
The maps $\theta\mapsto \shh(\theta)$ and $\theta\mapsto \htheta =  \beeta(\shh(\theta))$ are continuous. 
\end{lemma}
\begin{proof} 
We begin with the continuity properties. Let $(\alpha,\theta_n)\to (\alpha,\theta)$ be a sequence of endpoints, with the associated sequence of optimal trajectories $(\by_n,\beeta_n)$. Examining the proof of Lemma \ref{lem:Lipschitz} we find a locally uniform $H^1_{\rm loc}$ bound on both $(\by_n)$ and $(\beeta_n)$. By a diagonal extraction argument, we can extract a subsequence such that $(\by_{n_k},\beeta_{n_k})$ converges to some trajectory $(\by,\beeta)$ weakly in $H^1_{\rm loc}$. Fatou's lemma and the lower semi-continuity of $\bL_\alpha$ enables us to conclude, as in the proof of \Cref{lem:minimizer_existence}, that
\[
	\int_{-\infty}^0  \bL_\alpha(\by, \beeta, \dot \by, \dot \beeta) e^s ds
		\leq \liminf_{n\to\infty} U_\alpha(\alpha, \theta_n).
\]
Since $(\by,\beeta)\in \bA(x,\theta)$, then the left hand side is no smaller than $U_\alpha(\alpha,\theta)$. 
On the other hand, the convexity of $U_\alpha$ implies its continuity and, hence, that $\liminf_{n\to\infty} U_\alpha(\alpha,\theta_n) = U_\alpha(\alpha,\theta)$.   Taken together, this implies that $(\by,\beeta)$ is the minimizing trajectory associated to $(\alpha,\theta)$.

Next, for the sake of contradiction, consider a subsequence $\left (\shh(\theta_{n_k})\right )$ converging to some $s_0\neq \shh(\theta)$. For any $\delta>0$, we have $\by_{n_k}(s) = \alpha$ on $(-\infty,s_0-\delta)$ for all $k$ sufficiently large. Passing to the limit $k\to \infty$, and then $\delta\to 0$, we get that $\by(s) = \alpha$ on $(-\infty,s_0)$, and therefore $s_0\leq \shh(\theta)$.

We use the rigidity of the expression of the optimal trajectories from~\eqref{eq:traj_optim} in order to rule out possible jumps. Indeed, suppose that $s_0<\shh(\theta)$. Then, passing to the limit on the parameters $A_{n_k}$ and $B_{n_k}$ (up to another extraction), we get a polynomial function (in the variable $\tau = e^{s/2}$) which coincides with $\alpha$ on $(s_0, \shh)$, due to the convergence of $\by_{n_k}$ to $\by$. This can only happen if $A = B = 0$.  In this case, $\shh(\theta) =0$.  On the other hand, we find that $A_{n_k}, B_{n_k} \to 0$, and, thus, $\shh(\theta_{n_k}) \to 0$.  This implies that $s_0 = 0 = \shh(\theta)$, which is a contradiction.  We conclude that the whole sequence $\left (\shh(\theta_{n})\right )$ converges to $\shh(\theta)$. Therefore, $\theta\mapsto \shh(\theta)$ is continuous.

The same conclusion holds for  $\theta\mapsto   \beeta(\shh(\theta))$ because $(\beeta_n)$ converges locally uniformly thanks to the $H^1_{\rm loc}$ estimate.  
\end{proof}

As already discussed to motivate the statement in Lemma \ref{lem:exit the line}, we have $\shh<0$ if the endpoint is such that $\theta < \Thetastar$. In fact, the converse is true. 

\begin{lemma}\label{lem:never leave}
	If $\theta \geq \Thetastar$, then $\shh = 0$.
\end{lemma}

\begin{proof}
We consider $\overline{\alpha} = 1$ without loss of generality.
To begin with, we collect some useful identities at the time of contact.  
By definition, we have $\by(\shh) = \alpha$ and $\beeta(\shh) = \htheta$.
By (\ref{eq:traj_optim}), \Cref{lem:matching p lem},  and \Cref{lem:Airy}, we have,
\[
A e^{\shh/2}  = \bp(\shh) = \frac{1}{\htheta}, \quad \text{and}\quad B + A^2 (1 - e^{\shh}) =\bq(\shh) = \bQ(\htheta),
\]
which, with the usual notation  $\tauhh = e^{\shh/2}$, yields,
\begin{equation}\label{eq:A_B}
A = \frac{1}{\tauhh\htheta}, \quad \text{and}\quad  B= \bQ(\htheta) - \frac{1}{\htheta^2 \tauh^2}(1-\tauh^2).
\end{equation}
On the one hand, since $\by(\shh) = 4/3$, the expression for $\by$ in (\ref{eq:traj_optim}) implies,
\begin{equation*}
\begin{split}
	\frac43 \tauh^3 &= \by(\sh) \tauh^{3}
		= \frac43 + 2\theta  A(\tauh^2-1) + 2BA(1-\tauh^2)^2 + \frac{2}{3} A^3(1 - \tauh^2)^3\\
		&=\frac43
			+ 2\theta \frac{1}{\htheta \tauh}(\tauh^2-1)
			+ 2\left(\bQ(\htheta) - \frac{1}{\tauh^2 \htheta^2}(1 - \tauh^2)\right)\frac{1}{\tauh \htheta}(1-\tauh^2)^2
			+ \frac{2}{3} \left(\frac{1}{\tauh \htheta}\right)^3(1 - \tauh^2)^3.
\end{split}
\end{equation*}
Multiplying both sides by $\tauh/(2(1-\tauh^2))$ and re-arranging the terms implies,
\begin{equation}\label{eq:c35}
	\frac{\theta}{\htheta}
		= \frac{2}{3}\frac{\tauh - \tauh^4}{(1-\tauh^2)}
		+ \frac{\bQ(\htheta)}{ \htheta}(1-\tauh^2)
			- \frac{2}{3} \frac{1}{\tauh^2 \htheta^3}(1 - \tauh^2)^2.
\end{equation}
On the other hand, since $\beeta(\sh) = \htheta$, we have, from the expression for $\beeta$ in (\ref{eq:traj_optim}) and (\ref{eq:A_B}),
\begin{equation*}
	\tauh^2 \htheta
		= \tauh^2 \beeta(\sh)
		= \theta - 2(1-\tauh^2)\left(\bQ(\htheta) - (1 - \tauh^2)\frac{1}{\htheta^2 \tauh^2}\right) - \frac{1}{\htheta^2\tauh^2}(1-\tauh^2)^2.
\end{equation*}
Re-arranging this to obtain an expression for the ratio $\theta/\htheta$, and then plugging it into~\eqref{eq:c35} yields
\begin{equation}\label{eq:preliminary_W}
	3\frac{\bQ(\htheta)}{\htheta}
		= \frac{1}{\tauh^2\htheta^3}(1-\tauh^2) + \frac{\tauh(\tauh+2)}{(1+\tauh)^2}.
\end{equation}

We have  $\htheta(0) \geq \Thetastar$ by Lemma \ref{lem:exit the line}. We deduce from the dynamic programming principle  that $\htheta(\htheta(0)) =  \htheta(0)$. From this observation, we can define 
$\theta_0=\inf \left\{\theta \,| \, \htheta(\theta)=\theta\right\}$. We have $\theta_0\in [\Thetastar, \htheta(0)]$, and also $\htheta(\theta_0)=\theta_0$ by continuity of the map $\theta\mapsto\htheta$ established in Lemma \ref{lem:hitting_monotonicity}. 
Our goal is to show that $\theta_0 = \Thetastar$.

First, we notice that, as a simple consequence of the dynamic programming principle and the uniqueness of optimal trajectories, $\shh(\theta_0) = 0.$

Next, consider a sequence of points $\theta_n \nearrow \theta_0$. Note from the choice of $\theta_0$ that $\sh(\theta_n)<0$, for all $n$. Then, \eqref{eq:preliminary_W} implies that 
$$
 3\frac{\bQ(\htheta(\theta_n))}{\htheta(\theta_n)}
= \frac{1}{\tauh^2(\theta_n)\htheta(\theta_n)^3}(1-\tauh^2(\theta_n)) + \frac{\tauh(\theta_n) (\tauh(\theta_n)+2)}{(1+\tauh(\theta_n))^2}.
$$
We then let $n\to \infty$ and use the continuity of $\tauh$, $\htheta$ and $\bQ$ to obtain that
\[
	3\frac{\bQ(\theta_0)}{\theta_0} = \frac{1 (1 + 2)}{(1 + 1)^2} = \frac{3}{4}.
\]
Due to the definition of $\Thetastar$ in (\ref{eq:thetadiam}), this implies that $\theta_0 = \Thetastar$, as claimed. Finally, it follows from the  dynamic programming principle that $\by(s) = \alpha$ for all $s\in (-\infty,0]$ when $\theta\geq \Thetastar$. This concludes the proof that $\shh(\theta) = 0$ for all $\theta\geq \Thetastar$. 
\end{proof}

\subsection{The complete picture of the trajectories - \Cref{prop:trajectories}}\label{sec:proposition_compilation}

\begin{proof}[Proof of \Cref{prop:trajectories}]
There are a number of items to check.
\begin{enumerate}[(i)]
	\item   The existence of the contact time $\shh$ is  a consequence of Lemma \ref{lem:hits_line} and Lemma \ref{lem:one D turn}. The continuity of the map $\theta\mapsto \beeta\left (\shh(\theta)\right )$ is the purpose of Lemma \ref{lem:hitting_monotonicity}. 
	\item   The property $\beeta(\shh(\theta)) \geq \Thetastar$ is a consequence of Lemma \ref{lem:exit the line}.  
	\item   The fact that $\shh = 0$ if and only if $\theta \geq \Thetastar$ is a consequence of  Lemma \ref{lem:exit the line} and Lemma \ref{lem:never leave}.
	\item  We can separate the dynamics on and off the line, respectively for $s\in (-\infty,\shh)$ and $(\shh,0)$.  On
each interval, the Lagrangian is continuous and so the classical theory can be applied. 
Moreover, $(\by,\beeta,\bp,\bq)$  is globally continuous provided we define $\bp(s) =\overline{\alpha}/\beeta(s) $ for $s\leq \shh$, as shown in Lemma \ref{lem:matching p lem}.  As a by-product of the  classical theory, we have in particular 
 $U_\alpha(\alpha,\theta) = - \bH_\alpha(\alpha,\theta,\bp(0),\bq(0))$.   
	\item   The derivation of the first integral of motion  $\bq(s) = \bQ_\alpha(\beeta(s))$ is the purpose of Lemma \ref{lem:Airy};
	\item   The formula for $A$ at the contact time  is clear (see, e.g.,~\eqref{eq:A}).  The formula for $B$ follows from the continuity of $\bq$ along with the matching condition at $s=\sh$ coming from the combination of \eqref{eq:traj_optim} and \Cref{lem:Qtheta}.
\end{enumerate}
\end{proof}

\section{Conclusion and perspectives}

We have shown a weak propagation result for the cane toad equation \eqref{eq:intro_toads}. More precisely, we have proven that the front spreads slower than the linear problem without saturation. In fact, the linear problem was previously shown to spread as $(4/3) t^{3/2}$, in contrast to the rate $\alpha^* t^{3/2}$, where $\alpha^* \approx 1.315$, obtained here. However, our spreading result is quite weak, and oscillatory behavior could not be ruled out. 

Dumont performed intensive numerical computations on a large domain to investigate the long time asymptotics of \eqref{eq:intro_toads}. The methods and the results  are described in the following appendix. He does not report any oscillatory behavior. The spatial density appears to be monotonic non-increasing with respect to the space variable. In addition, all level lines propagate at the same rate $\mathcal O(t^{3/2})$ with the same prefactor. Furthermore, the numerical spatial density converges to a Heaviside function with unit saturated value $\1_{ \left\{x< \alpha_h t^{3/2}\right\}}$ in the self-similar spatial variable $x/t^{3/2}$, for some numerical critical value $\alpha_h$. 

This suggests that  Theorem \ref{thm:propagation} could be strengthened towards a strong spreading result stating that all level lines  propagate as $\alpha^* t^{3/2}$. Accordingly, we conjecture that the value function $U_{\alpha,1}$ is a good candidate to describe the asymptotic behavior associated with the exponential ansatz discussed in \Cref{sec:strategy} as $t\to \infty$. An alternative would be to seek a stationary profile adapted to the various scales of the problem, as discussed in Figure \ref{fig:data-thierry-2}. 

We are not aware of any other reaction-diffusion problem related to the Fisher-KPP equation where the saturation term hinders the propagation at first order. Usually, the non-linear term acts on the next order correction  of the front location, as in the Bramson logarithmic delay \cite{Bramson83,HNRR13,BHR_LogDelay, Penington}. Our analysis unravels the interplay between unbounded diffusion, curved trajectories due to the twisted Laplacian $\theta \partial^2_x + \partial^2_\theta$, and non-local competition among individuals at the same  location, but having different dispersal abilities, as shown in Figure \ref{fig:linear determinacy}.  We believe that the methodology developed here could be extended to other related problems.

\newpage

\appendix

\begin{figure}
\begin{center}
\subfloat[]{
\includegraphics[width = 0.48\linewidth]{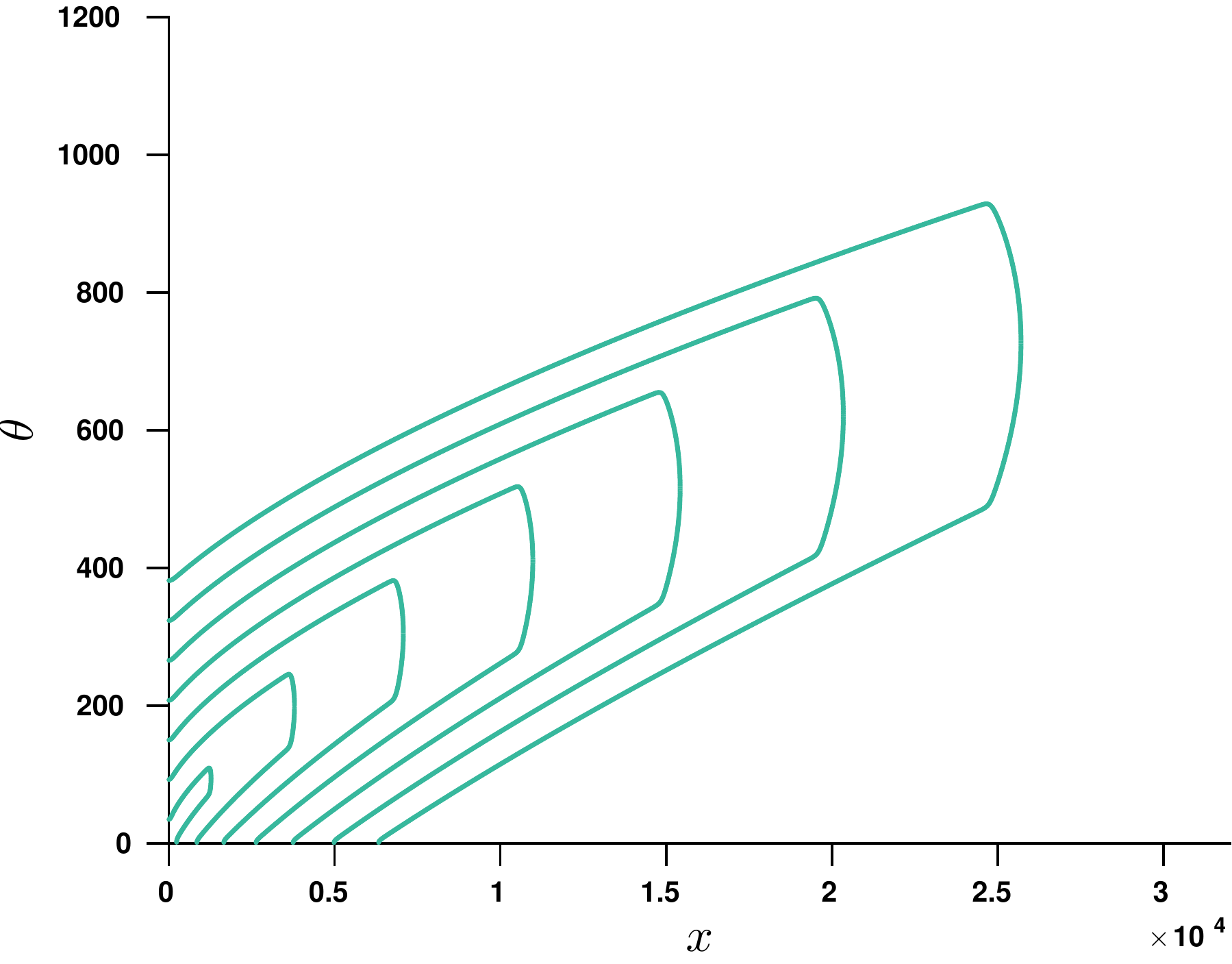}}
\;
\subfloat[]{
\includegraphics[width = 0.48\linewidth]{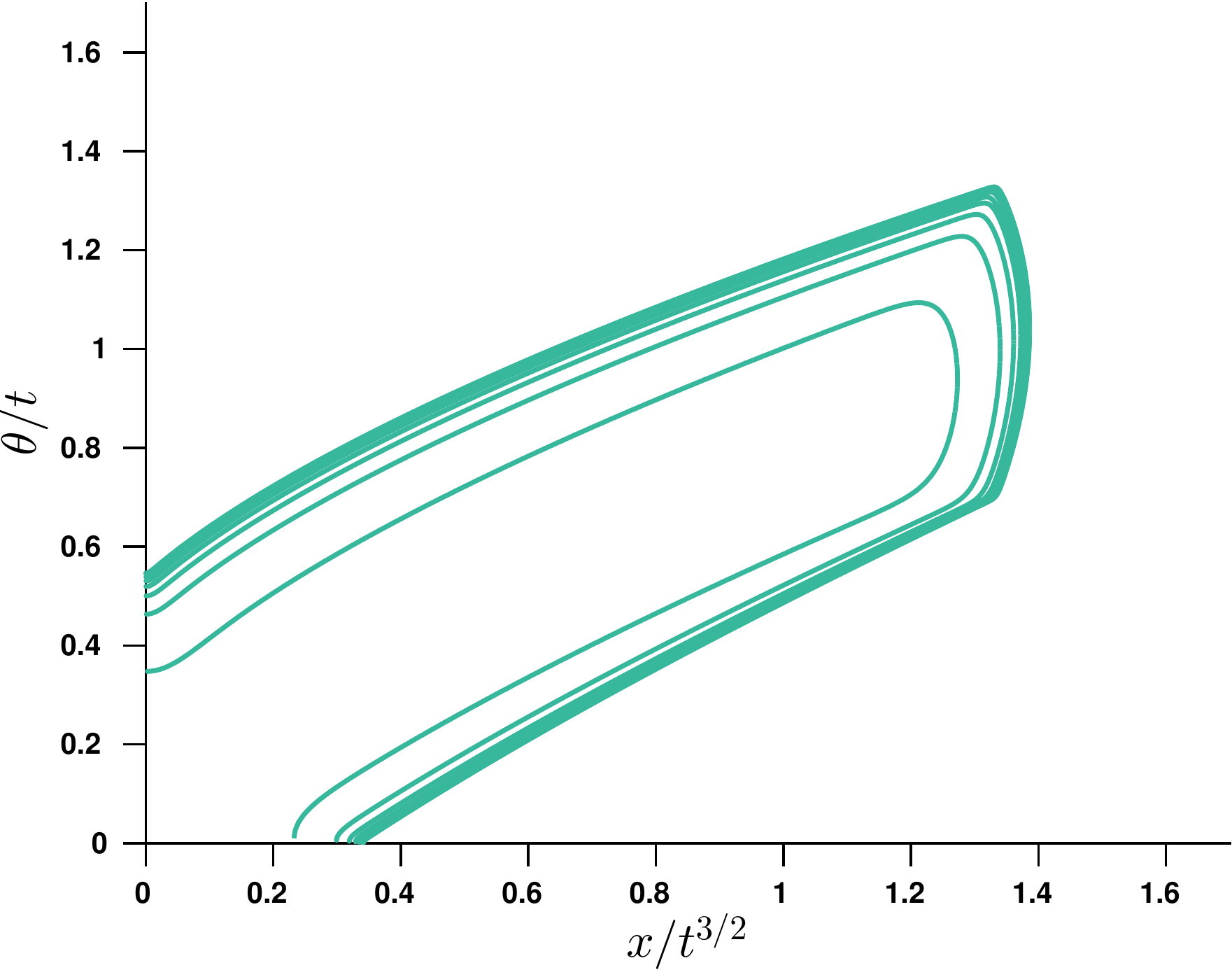}} \caption{One level line of the function $U = -t\log f$ is plotted for successive times $t = 100, 200, 300, 400, 500, 600, 700$ with respect to the original variables $(x,\theta)$ in (A) or with respect to the self-similar variables $(x/t^{3/2},\theta/t)$ in (B). One clearly sees the joint propagation in $(x,\theta)$ towards larger $x$ and higher $\theta$. Moreover, the function $U$ seems to converge to a stationary profile in the self-similar variables, in agreement with the heuristic argument of \Cref{sec:strategy} that $u_h \to u$ where $u_h$ is given by~\eqref{eq:u_h}, $u$ solves~\eqref{eq:formal_equation}, and $u$ and $U$ are connected by a change of variables as in \Cref{sec:self-similar}.}
\label{fig:data-thierry-1}
\end{center}
\end{figure}

\begin{figure}
\begin{center}
\subfloat[]{
\includegraphics[width = 0.48\linewidth]{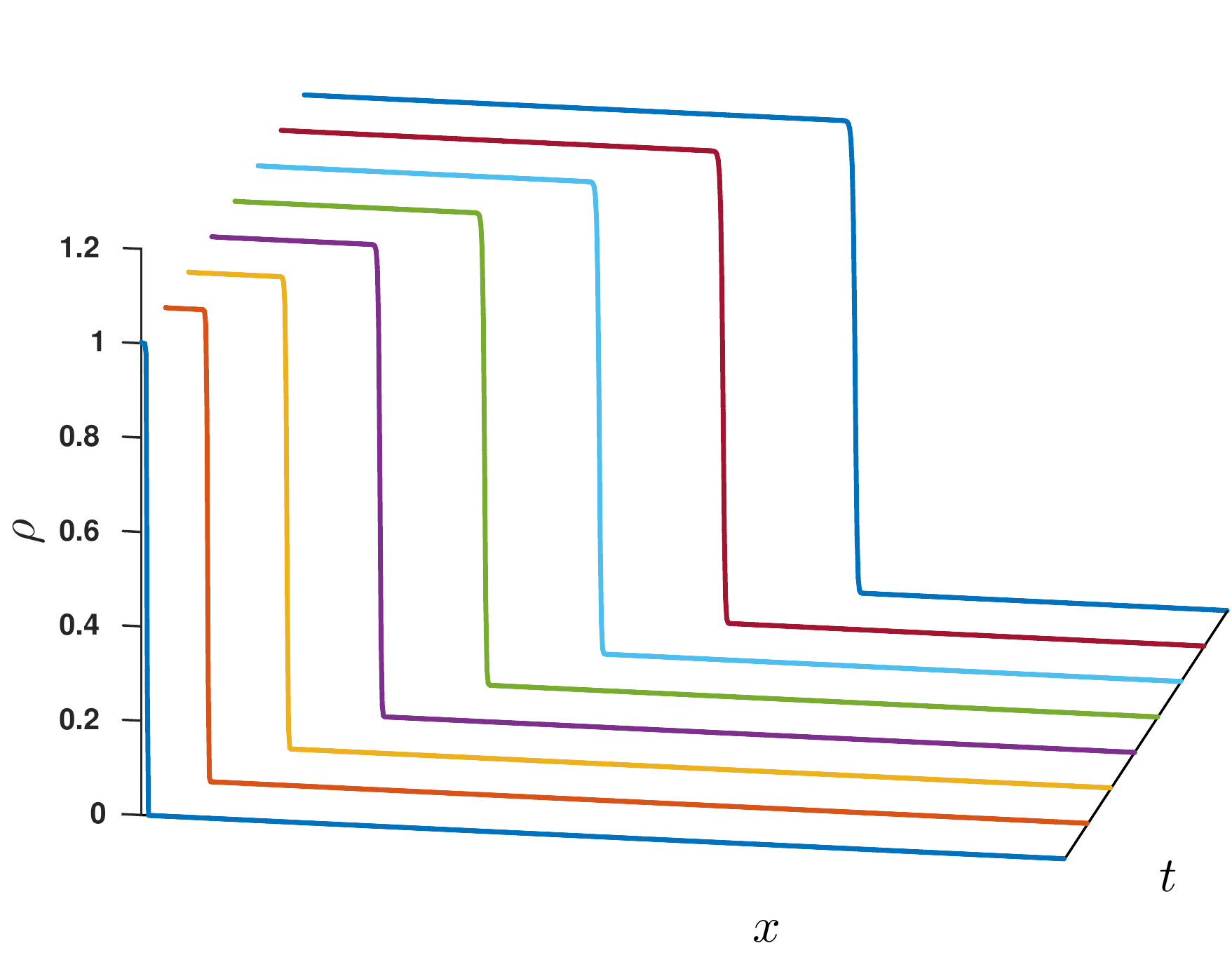}} 
\;
\subfloat[]{
\includegraphics[width = 0.48\linewidth]{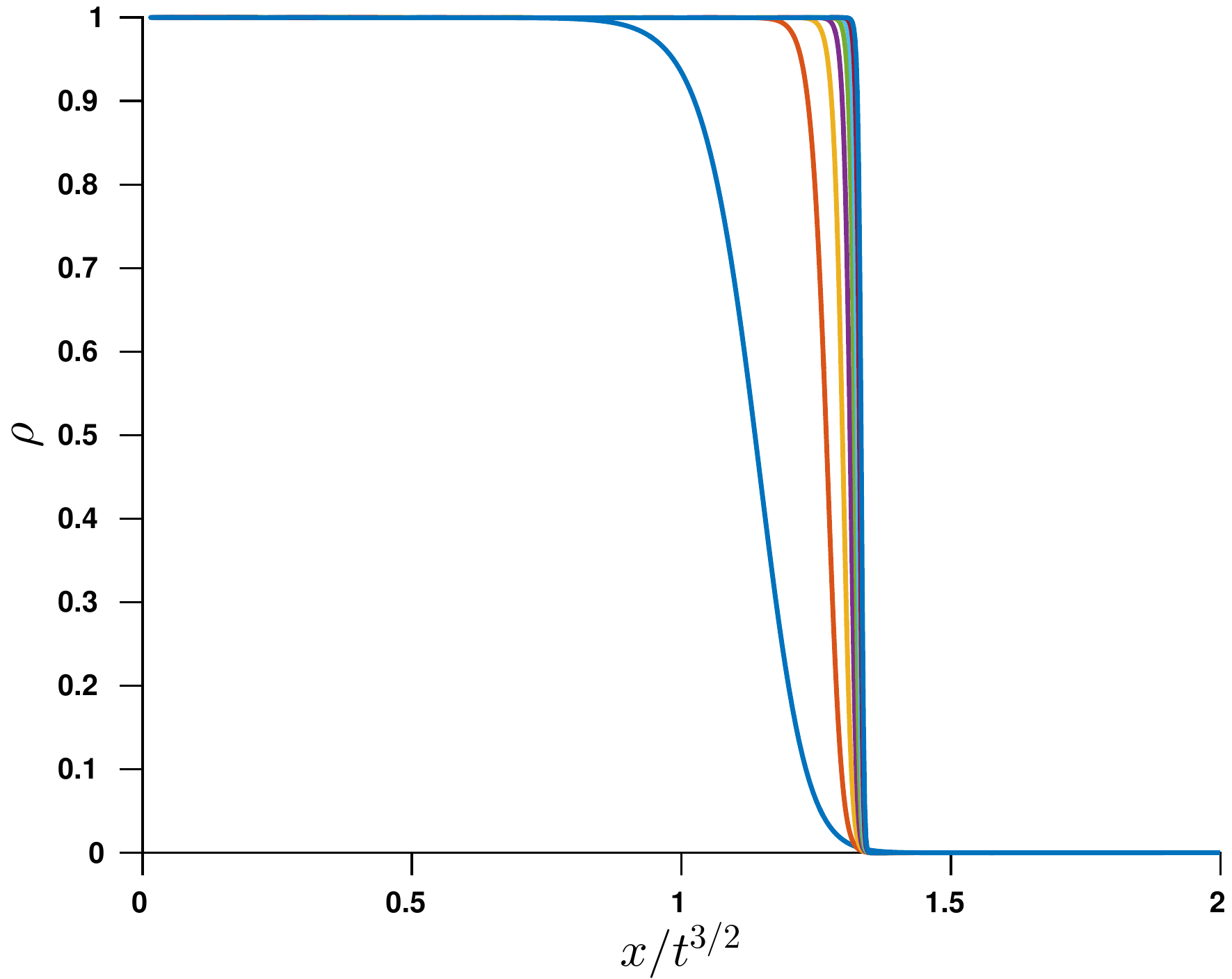}}
\\
\subfloat[]{
\includegraphics[width = 0.48\linewidth]{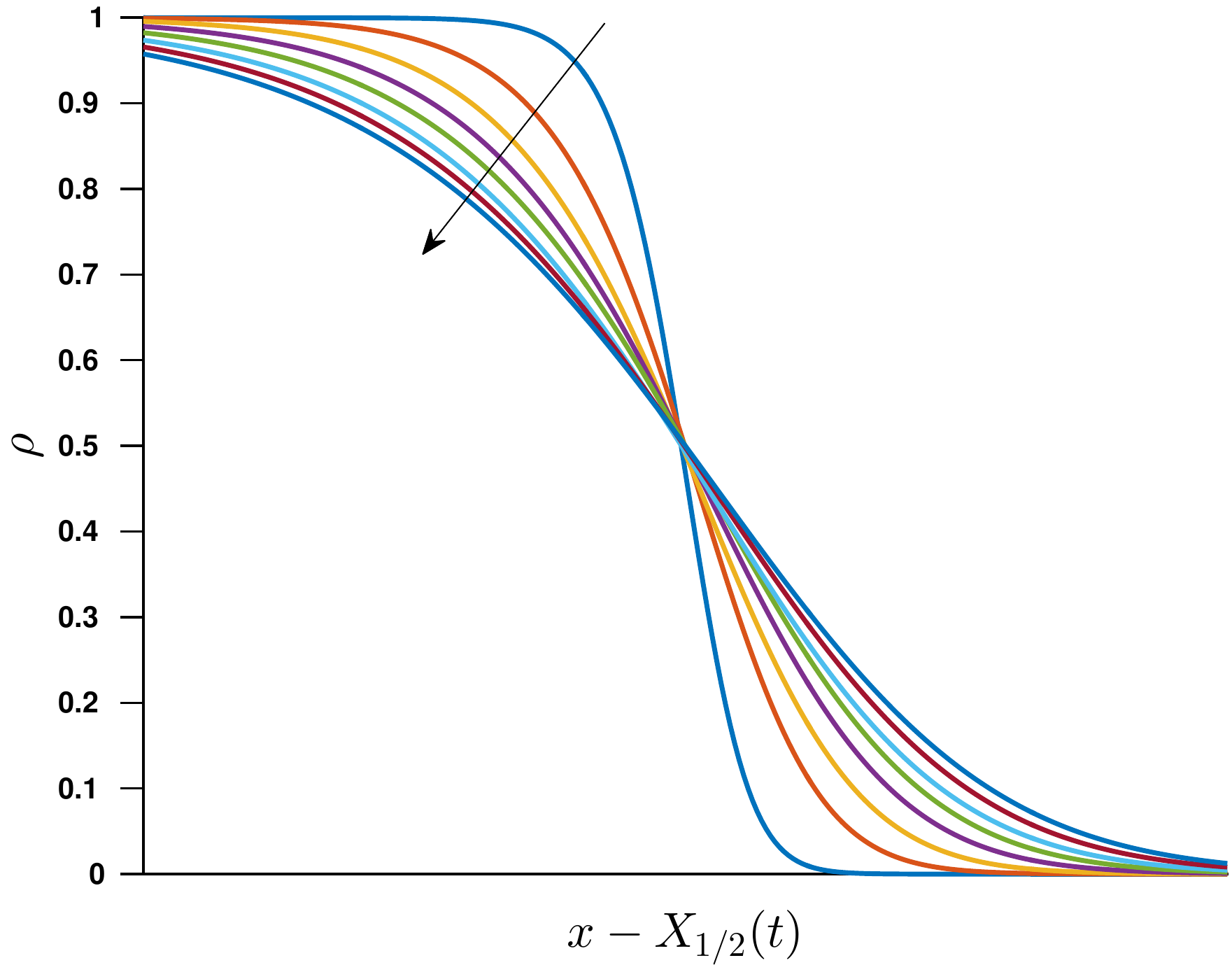}}\;
\subfloat[]{
\includegraphics[width = 0.48\linewidth]{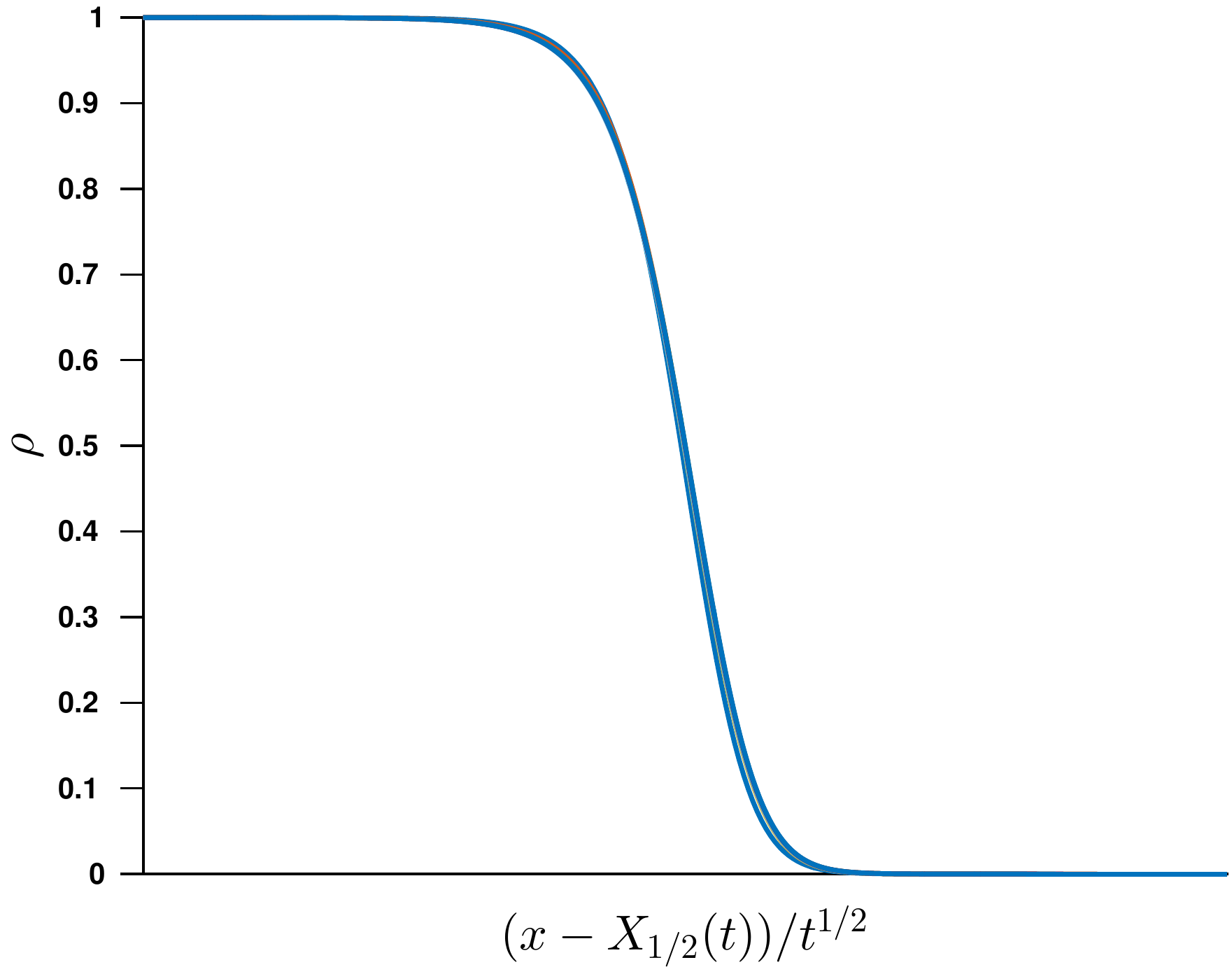}} 
\caption{Spatial propagation. (A) The numerical function $\rho(t,x)$ is plotted for successive times at regular intervals. (B) The spatial density converges to a Heaviside function in the self-similar variable $x/t^{3/2}$, in agreement with the analysis performed in the article.
(C) The same curves are plotted, but in the frame centered at the abscissa $X_{1/2}(t)$ corresponding to the value $\rho=1/2$. Increasing times are figured by an arrow. The front flattens as time increases. (D) By playing with scales, I found that the typical width of the front is of order $t^{1/2}$, as all curves are superposed in this frame.}
\label{fig:data-thierry-2}
\end{center}
\end{figure}

\section{Numerical computations of the long time asymptotics}
The numerical approximation of the Cauchy problem \eqref{eq:intro_toads} raises several challenges:
\begin{enumerate}
\item Handling the non local reaction term  with an implicit
  scheme would result in full non linear systems to be solved at each time step.
\item But as time $t$ increases, diffusion triggers faster and faster time  scales as the solution propagates in $x$. Therefore, an implicit stable  discretization seems necessary~\cite{MR1439506}.
\item Experience shows that a large domain and a thin discretization  is necessary to achieve good spreading numerical results. Thus, the numerical  simulation, whatever the method opted for, requires a large amount
  of computing time, even with a parallel procedure. A strategy for  reducing the computing time appears to be necessary.
\end{enumerate}
\subsection{Methods}
I opted for standard operator splitting techniques. These techniques date back to the  1950's. However, it as been shown recently that they are 
 well adapted to difficult and even very stiff
problems \cite{MR1836914,MR2869398}. 
Being given an initial value and a time step $h$ for the problem  $df/dt = L f +R(f)$, decomposed into the linear (diffusion) part and the non-linear (reaction part), I advance  from time $n h$ to time $(n+1) h$ by solving only partial problems:
$L_{h}:$ $ df/dt= Lf$ and $R_{h}:$  $df/dt = R(f)$ during the time step
$h$. Denoting by $f_n$ the numerical solution
at time $nh$, the Strang scheme \cite{MR0235754}
\[f_{n+1}= L_{h/2} \circ R_{h} \circ L_{h/2}\ (f_{n})\]
is of order 2 provided $L_{h/2}$ and $R_{h}$ are also
approximated by numerical schemes of order 2.
It can be generalized to three operators \cite{MR2869398}, keeping order
2. Define the three partial problems as: 
$L^x_{h}: df/dt= \theta \partial^2 f/\partial x^2$,
$L^\theta_{h}: df/dt= \partial^2 f/\partial \theta^2$, and
$R_{h}: df/dt = R(f)$, and the corresponding Strang scheme as follows 
\[f_{n+1}=  L^x_{h/2} \circ L^\theta_{h/2} \circ R_{h} \circ
L^\theta_{h/2} \circ  L^x_{h/2}\ (f_{n}).\]

\subsubsection{Numerical computation of each sub-problem}
I approximated the operators  $L^x_{h}$ and $L^\theta_{h}$ using the
Crank-Nicolson method, which is of order 2 and A-stable \cite{MR1439506}. The non-local reaction term $R_{h}$  is non 
stiff and was approximated by a second order explicit Runge-Kutta method
(RK2) \cite{hairer1993solving}. 

The spatial discretization was made via  second
order  finite 
differences, on an uniform grid of size $N_x \times N_\theta$. Notice that $L^x_{h}$ and
$L^\theta_{h}$ can be decomposed further  in independent sub-problems acting respectively on the
rows and on the columns of the finite difference grid. Hence, each sub-problem boils down to solving
banded tridiagonal linear systems. Hence the cost
of advancing one step in time these two discrete operators reduces to
only $\mathcal{O} (N_x \times N_\theta)$. Moreover the sub-problems
of $L^x_{h}$ and $L^\theta_{h}$ can be computed in parallel.

The non-local
reaction term  $R_{h}$ involves the rate of growth  $1 - \rho(t,x) = 1- \int_1^\infty f(t,x,\theta) d\theta$ which does not depend on $\theta$. Hence, it can  
be computed   in parallel for each value of $x$ on the grid.

\subsubsection{Time step control} I used the first order
Euler splitting scheme in order  to compute
\[
f^*_{n+1}=  L_{x_{h}} \circ  L_{\theta_h} \circ R_{h}\ (f_{n}).\]
I used the quantity
$||f^*_{n+1}-f_{n+1}||_{L^2} = \mathcal{O}(h^2)$ as an error indicator  to adapt the time step $h$, as usually done for solving ordinary differential equations \cite{hairer1993solving}.

\subsubsection{Implementation}
The code was implemented in \textsc{C++}, \textsc{OpenMP} parallel (in shared
memory). The size of the domain was $(x,\theta)\in [0, 4.5E4]\times[0,1.6E3]$.  
The  size of the grid was $N_x \simeq 2800$ and  $N_\theta = 1000$.   
The time step was adapted at each iteration for small time
$t$, then every 5 steps afterwards. The total wall clock  computing time of the
simulation was approximately 4 days, on a 32 cores (2.2 Ghz clock
frequency) computer.

\subsection{Results}
Starting from an initial datum as in \eqref{eq:initial_data}, but restricted to the positive values of $x$, the spatial propagation was observed in the long term with rate $\mathcal O(t^{3/2})$. The numerical value of the prefactor seems to converge to a value lying between 1.34 and 1.35. This looks misleading. However, it is known that the numerical computation of spreading for monostable travelling waves are challenging, even for the basic equation \eqref{eq:FKPP}. In particular, the speed of the wave is difficult to estimate accurately in the latter problem. 

The analysis performed in this article suggests that accurate numerical schemes should be developed on the auxiliary function $u = -t\log f$, in the self-similar variables, in order to match with the ansatz in \Cref{sec:strategy} and \Cref{sec:self-similar}.  I checked that the numerical approximation of $u = -t\log f$ did converge in self-similar variables to a stationary function (Figure \ref{fig:data-thierry-1}). This stationary solution is likely to be an approximation of the value function $U_{\alpha}$. There is indeed a good match (comparison not shown). 

To investigate further the consistency of the analysis performed in the article, I checked whether the spatial density $\rho(t,x)$ resembles a Heaviside function $\mu \1_{\left\{x <  \alpha t^{3/2}\right\}}$. First, I noticed that, despite the lack of maximum principle, the numerical spatial density $\rho(t,x)$ remains below the unit carrying capacity: $\rho\leq 1$. Moreover, it is monotonic non-increasing in space, and non-decreasing in time, see Figure \ref{fig:data-thierry-2}. The numerical results suggest that the spatial density indeed converges to the Heaviside function $\1_{\left\{ x <  \alpha_h t^{3/2}\right\}}$, where the critical value $\alpha_h$ depends on the numerical approximation of the scheme. 

No stationary behavior seems to be reached in the long term asymptotics (Figure \ref{fig:data-thierry-2}C). More precisely, the shape of the front flattens as time increases. The typical width appears to be of order $\mathcal{O}(t^{1/2})$.

\bibliographystyle{abbrv}
\bibliography{refs}

\end{document}